\documentclass[english,11pt, reqno]{amsart}
\usepackage{amsmath}
\usepackage{amssymb}
\usepackage{amsthm}
\usepackage{mathtools}
\usepackage{bbm,bm}
\usepackage{stmaryrd}
\usepackage{xcolor}
\usepackage[colorlinks=true,linkcolor=blue!95!black, citecolor = green!75!black,bookmarksdepth=3]{hyperref}
\usepackage[margin=1in]{geometry}
\usepackage{multicol}
\usepackage{etoolbox}
\usepackage{enumitem}
\usepackage{float}
\usepackage[mathscr]{euscript}
\usepackage{graphicx}
\usepackage{subcaption}

\allowdisplaybreaks

\parskip 	\smallskipamount

\newcommand{\hair}{\ifmmode\mskip1mu\else\kern0.08em\fi}
\renewcommand{\P}{\mathbb{P}}

\newcommand{\E}{\mathbb{E}}

\renewcommand{\L}{\mathbb{L}}

\newcommand{\R}{\mathbb{R}}

\newcommand{\N}{\mathbb{N}}
\newcommand{\Z}{\mathbb{Z}}

\newcommand{\one}{\mathbbm{1}}
\newcommand{\intint}[1]{\llbracket #1 \rrbracket}
\newcommand{\mc}{\mathcal}
\newcommand{\linelower}{\mathbb L_{\mathrm{low}}}
\newcommand{\lineupper}{\mathbb L_{\mathrm{up}}}

\newcommand{\e}{\varepsilon}

\newcommand{\Elow}{E_{\mathrm{low}}}
\newcommand{\Eup}{E_{\mathrm{up}}}
\newcommand{\philower}{\phi_{\mathrm{low}}}
\newcommand{\phiupper}{\phi_{\mathrm{up}}}

\DeclareMathOperator*{\tf}{TF}

\newcommand{\sep}[1]{\mathrm{sep}^{(#1)}}

\newcommand{\iid}{i.i.d.\ }

\newtheorem{maintheorem}{Theorem}

\newtheorem{theorem}{Theorem}[section]
\newtheorem*{theorem*}{Theorem}
\newtheorem*{proposition*}{Proposition}
\newtheorem{proposition}[theorem]{Proposition}
\newtheorem*{corollary*}{Corollary}

\newtheorem{lemma}[theorem]{Lemma}

\theoremstyle{definition}

\newtheorem{remark}[theorem]{Remark}

\newcounter{countAssumption}

\newtheorem{assumption}[countAssumption]{Assumption}

\title[Optimal tail exponents in general LPP via bootstrapping \& geodesic geometry]{Optimal tail exponents in general last passage percolation via bootstrapping \& geodesic geometry}

\author{Shirshendu Ganguly}
\address{Shirshendu Ganguly, Department of Statistics, U.C. Berkeley, Berkeley, CA, USA}
\email{sganguly@berkeley.edu}

\author{Milind Hegde}
\address{Milind Hegde, Department of Mathematics, Columbia University, New York, NY, USA}
\email{milind.hegde@columbia.edu}


\begin{document}

\begin{abstract}
We consider last passage percolation on $\Z^2$ with general weight distributions, which is expected to be a member of the Kardar-Parisi-Zhang (KPZ) universality class. In this model, an oriented path between given endpoints which maximizes the sum of the \iid weight variables associated to its vertices is called a geodesic. Under natural conditions of curvature of the limiting geodesic weight profile and stretched exponential decay of both tails of the point-to-point weight, we use geometric arguments to upgrade the tail assumptions to prove optimal upper and lower tail behavior with the exponents of $3/2$ and $3$ for the weight of the geodesic from $(1,1)$ to $(r,r)$ for all large finite $r$, and thus unearth a connection between the tail exponents and the characteristic KPZ weight fluctuation exponent of $1/3$. The proofs merge several ideas which are not reliant on the exact form of the vertex weight distribution, including the well known super-additivity property of last passage values, concentration of measure behavior for sums of stretched exponential random variables, and geometric insights coming from the study of geodesics and more general objects called geodesic watermelons. Previous proofs of such optimal estimates have relied on hard analysis of precise formulas available only in integrable models. Our results illustrate a facet of universality in a class of KPZ stochastic growth models and provide a geometric explanation of the upper and lower tail exponents of the GUE Tracy-Widom distribution, the conjectured one point scaling limit of such models.  The key arguments are based on an observation of general interest that super-additivity  allows a natural iterative bootstrapping procedure to obtain improved tail estimates.
\end{abstract}


\maketitle

\setcounter{tocdepth}{1}
\tableofcontents


\section{Introduction, main results, and key ideas}


The 1+1 dimensional Kardar-Parisi-Zhang (KPZ) universality class includes a wide range of models of interfaces suspended over a one-dimensional domain, in which growth in a direction normal to the surface competes with a smoothing surface tension, in the presence of a local randomizing force that roughens the surface. Such interfaces are expected to exhibit characteristic exponents dictating one-point height fluctuations and correlation length. 
While the class is predicted to describe a plethora of models including first passage percolation, last passage percolation, the KPZ stochastic PDE, and the totally asymmetric simple exclusion process, among others, the above predictions  have been rigorously proven only for a very small subset of them.


We now give a brief description of the model of last passage percolation, an important member of this class and the model in consideration in this article.

In last passage percolation (LPP) one assigns \iid non-negative weights $\{\xi_v : v\in \Z^2\}$ to the vertices of $\Z^2$ and studies the weight and geometry of weight-maximising directed paths. The weight of a given up-right nearest neighbor path $\gamma$ is $\ell(\gamma) := \sum_{v\in \gamma} \xi_v$. For given vertices $u=(u_1,u_2), v=(v_1,v_2) \in \Z^2$ with $u_i\leq v_i$ for $i=1$ and $2$ (i.e., the natural partial order), the last passage value $X_{u,v}$ is defined by $X_{u,v} = \max_{\gamma:u\to v} \ell(\gamma)$, where the maximization is over the set of up-right paths from $u$ to $v$; maximizing paths are called \emph{geodesics}. For $r\in \N$, we adopt the shorthand $X_{r} := X_{(1,1), (r,r)}$.

A few special distributions of the vertex weights $\{\xi_v:v\in\Z^2\}$ render the model integrable, i.e., admitting exact connections to algebraic objects such as random matrices and Young diagrams. This allows exact computations which lead to the appearance of 
the GUE Tracy-Widom distribution. Most of the progress in understanding the KPZ universality class has relied primarily on such exactly solvable features. 
%
For concreteness, we highlight next  the special case of exponentially distributed (with rate one) vertex weights. In this case, Johansson proved the following \cite{johansson2000shape} via the development of the aforementioned connections to representation theory.

\begin{theorem}[Theorem 1.6 of \cite{johansson2000shape}]\label{t.tracy-widom convergence}
Let $\{\xi_v :v\in \Z^2\}$ be \iid exponential rate one random variables. As $r\to\infty$ it holds that
$$\frac{X_r - 4r}{2^{4/3}r^{1/3}}\stackrel{d}{\to} F_{\mathrm{TW}},$$
where $F_{\mathrm{TW}}$ is the GUE Tracy-Widom distribution, and $\stackrel{d}{\to}$ denotes convergence in distribution.
\end{theorem}


An important feature of the GUE Tracy-Widom distribution is the ``non-Gaussian" behavior of its upper and lower tails. In particular, it is known, for example from \cite[page 224]{seppalainen1998large} or \cite[Theorem~1.3]{ramirez2011beta}, that as $\theta\to\infty$,
\begin{equation}\label{e.TW tails}
\begin{split}
F_{\mathrm{TW}}\big([\theta,\infty)\big) &= \exp\left(-\frac{4}{3}\theta^{3/2}\left(1+o(1)\right)\right) \qquad\text{and}\\
F_{\mathrm{TW}}\big((-\infty, -\theta]\big) &= \exp\left(-\frac{1}{12}\theta^{3}\left(1+o(1)\right)\right).
\end{split}
\end{equation}
In fact, these tail exponents of $3/2$ and $3$ are more universal in KPZ than just the GUE Tracy-Widom distribution. The latter distribution is
only expected to arise under what is called the \emph{narrow-wedge} initial data; this is seen in Theorem~\ref{t.tracy-widom convergence} by the definition of $X_r$ as the weight of the best path from the \emph{fixed} starting point of $(1,1)$. But the same tail exponents are expected for a much wider class of initial data. For example, the results of \cite{corwin2018kpz} assert that the (suitably scaled) solution to the KPZ stochastic PDE has upper bounds on the one-point upper and lower tails with the same tail exponents (up to a certain depth into the tail) under a wide class of \emph{general} initial data. Similarly, the same exponents are known from \cite{ramirez2011beta} for the entire class of Tracy-Widom($\beta$) distributions (with the GUE case corresponding to $\beta=2$).

Given the distributional convergence asserted by Theorem~\ref{t.tracy-widom convergence}, it is natural to ask whether tail bounds similar to \eqref{e.TW tails} are satisfied by $X_r$ at the finite $r$ level. Indeed, again in the case of exponential weights, estimates along these lines have been attained which achieve the correct upper and lower tail exponents of $3/2$ and $3$. The first result in this direction was proved by Sepp\"al\"ainen, who obtained an upper bound for the upper tail (with the correct leading exponent coefficient $4/3$) in \cite[page 622]{seppalainen1998coupling} via a coupling with the totally asymmetric simple exclusion process and an evaluation and expansion of the large deviation rate function. The large deviation bound yields a finite $r$ estimate using superadditivity properties of the upper tail probabilities (see \eqref{e.johansson super-additivity} ahead for a discussion). But this strategy does not give a lower bound or bounds for the lower tail, and these bounds were proven using connections to random matrix theory. In more detail, Johansson proved in \cite[Remark~1.5]{johansson2000shape} via representation theoretic techniques that $X_r$ is equal in distribution to the top eigenvalue of the Laguerre Unitary Ensemble, and upper bounds on the upper and lower tails on this eigenvalue were proved in \cite[Theorem~2]{ledoux2010}; it is possible, though it does not seem to be written down, that the lower tail estimate also follows from a similar Riemann-Hilbert analysis as performed in the setting of geometric LPP in \cite{baik2001optimal}. \cite{ledoux2010} remarks, but does not prove, that a lower bound on the upper tail should be achievable by methods in the paper, but not a lower bound on the lower tail; the latter was proved very recently in \cite[Theorem~2]{basu2019lower}. This discussion may be summarized as the following theorem.

\begin{theorem}[\cite{seppalainen1998coupling,johansson2000shape,ledoux2010,basu2019lower}]\label{t.tails for exp LPP}
Let $\{\xi_v:v\in\Z^2\}$ be \iid exponential random variables. There exist positive finite constants $c_1$, $c_2$, $c_3$, $\theta_0$, and $r_0$ such that, for $r>r_0$ and $\theta_0 < \theta < r^{2/3}$,
\begin{align*}
\P\left(X_r > 4r + \theta r^{1/3}\right) &\leq \exp\left(-c_1\theta^{3/2}\right) \qquad\text{and}\\
\exp\left(-c_2\theta^3\right)\leq \P\left(X_r < 4r - \theta r^{1/3}\right) &\leq \exp\left(-c_3\theta^{3}\right).
\end{align*}
\end{theorem}

\begin{remark}
In fact, the missing lower bound on the upper tail is a straightforward consequence of one of our results (Theorem~\ref{t.upper tail lower bound}) along with the distributional convergence in Theorem~\ref{t.tracy-widom convergence} and an application of the Portmanteau theorem.
\end{remark}

That the above bounds hold only for $\theta \leq r^{2/3}$ is an important fact because one should not expect universality beyond this threshold. The lower tail is trivially zero for $\theta > 4r^{2/3}$ since the vertex weights are non-negative; for the upper tail,  beyond this level, we enter the large deviation regime, where the tail behavior is dictated by the individual vertex distribution. Thus in the case of exponential LPP, the upper tail decays exponentially in $\theta r^{1/3}$ for $\theta>r^{2/3}$.

Similar bounds as Theorem~\ref{t.tails for exp LPP} are available in only a handful of other LPP models; these are when the vertex weights are geometric \cite{johansson2000shape,baik2001optimal}, and the related models of Poissonian LPP \cite{seppalainen1998large, lowe-moderate-upper,lowe-moderate-lower} and Brownian LPP \cite{o2002representation,ledoux2010}. While \cite{seppalainen1998large} relies on coupling Poissonian LPP to Hammersley's process (a continuous version of the exclusion process), the remaining arguments use powerful identities with random matrix theory and connections to representation theory, combined with precise analysis of the resulting formulas.
 
However,  the conjectured universality of KPZ behavior suggests that similar bounds should hold under rather minimal assumptions on the vertex weight distribution, i.e., even when special connections to random matrix theory and representation theory are unavailable. 
Thus it is an important goal to develop more robust methods of investigation that may apply to a wider class of models, an objective that has driven a significant amount of work in this field, with the eventual aim to go beyond integrability. 


Nonetheless, despite various attempts, so far only a few results are known to be true in a universal sense. These include the existence of a limiting geodesic weight profile (i.e., the expected geodesic weight as the endpoint varies) and its concavity under mild moment assumptions on the vertex weights \cite{martin2006last}. This is a relatively straightforward consequence of super-additivity properties exhibited by the geodesic weights, as we elaborate on later. This and certain general concentration estimates based on martingale methods were first developed in Kesten's seminal work on first passage percolation (FPP) \cite{kesten1986aspects}; FPP is a notoriously difficult to analyze and canonical non-integrable model in the KPZ class where the setting is the same as that of LPP, but one instead minimizes the weight among all paths between two points, without any orientation constraint. Similar arguments extend to the case of general LPP models. Note that while the precise limiting profile is expected to depend on the model, properties such as concavity as well as local fluctuation behavior are predicted to be universal.

Following Kesten's work, there has been significant progress in FPP in providing  rigorous proofs assuming certain natural conditions, such as strong curvature properties of limit shapes and the existence of critical exponents dictating fluctuations. Thus an important broad goal is to extract the minimal set of key properties of such models that govern other more refined behavior. 
The recent work of the authors with Riddhipratim Basu and Alan Hammond in \cite{watermelon}, as well as the present work, are guided by the same philosophy. We will revisit this discussion in more detail after the statements of our main results.

To initiate the geometric perspective of the present paper, 
we point out the disparity in the upper and lower tail exponents in Theorem \ref{t.tails for exp LPP}. This is not surprising since, while the upper tail event enforces the existence of a \emph{single} path of high weight, the lower tail event is  global and forces \emph{all} paths to have low weight. 

However, the precise exponents of $3/2$ and $3$ might appear mysterious, and it is natural to seek a geometric explanation for them. This is the goal of this work.  More precisely, we establish bounds with optimal exponents in the nature of Theorem~\ref{t.tails for exp LPP}, starting from certain much weaker tail bounds as well as local strong concavity assumptions on the limit shape (Theorems~\ref{t.upper tail bootstrapping}--\ref{t.lower tail lower bound}). In particular, we do not make use of any algebraic techniques in our arguments; indeed, the nature of our assumptions do not allow such techniques to be applicable.
 Instead, our methods are strongly informed by an understanding of the geometry of geodesics and other weight maximising path ensembles in last passage percolation. 

We also mention that, while our main result is known in integrable models such as exponential LPP in view of Theorem~\ref{t.tails for exp LPP}, our techniques also obtain new tail exponents for a related LPP problem, namely the lower tail of the maximum weight over all paths \emph{constrained} to lie inside a strip of given width (Theorem~\ref{t.constrained lower tail bounds} ahead); the precise exponent depends on this width. Estimates of these probabilities have played important roles in previous geometric investigations \cite{slow-bond,watermelon}, but sharp forms had not been proven even in integrable models, and do not seem amenable to exactly solvable analysis. The form of the exponent we prove in Theorem~\ref{t.constrained lower tail bounds} is also suggestive of the anticipated answer to the question of typical transverse fluctuations of the geodesic when conditioned on having low weight in the moderate deviation regime; we elaborate on this slightly following Theorem~\ref{t.constrained lower tail bounds}. The large deviation version of the same question was investigated in \cite{basu2019delocalization}, and the related upper tail large deviation version in FPP in \cite{basu2017upper}.

%

We next set up precisely the framework of last passage percolation on $\Z^2$, describe our assumptions, and state our main results.


\subsection{Model and notation}
We denote the set $\{1,2, \ldots\}$ by $\N$, and, for $i,j\in\Z$, we denote the integer interval $\{i,i+1, \ldots, j\}$ by $\intint{i, j}$.

We start with a random field $\left\{\xi_v : v\in \Z^2\right\}$ of \iid random variables following a distribution $\nu$ supported on $[0,\infty)$. We consider up-right nearest neighbor paths, which we will refer to as \emph{directed} paths. For a directed path $\gamma$, the associated \emph{weight} is denoted $\ell(\gamma)$ and is defined by
$$\ell(\gamma) := \sum_{u\in \gamma} \xi_u.$$
For $u, v\in \Z^2_+$, with $u\preceq v$ in the natural partial order mentioned earlier, we denote by $X_{u,v}$ the \emph{last passage value} or \emph{weight} from $u$ to $v$, i.e.,
$$X_{u,v} := \max_{\gamma: u\to v}\, \ell(\gamma),$$
where the maximization is over all directed paths from $u$ to $v$; for definiteness, when $u$ and $v$ are not ordered in this way and there is no directed path from $u$ to $v$, we define $X_{u,v} = -\infty$. Now for ease of notation, for sets $A, B\subseteq \Z^2$, we also adopt the intuitive shorthand 
$$X_{A,B} := \sup_{u\in A, v\in B} X_{u,v}.$$
For $v\in \Z^2_+$, $X_v$ will denote $X_{(1,1), v}$, and for $r,z\in \Z$, we will denote $X_{(1,1),(r-z,r+z)}$ by $X_r^z$. We will also denote the case of $z=0$ by $X_r$, as above. Notational confusion between $X_v$ and $X_r$ is avoided in practice in this usage as $v$ will always be represented by a pair of coordinates, while $r$ is a scalar. Recall that a path (which may not be unique) which achieves the last passage value is called a {geodesic}.

For an up-right path $\gamma$ from $(1,1)$ to $(r-z, r+z)$, we define the \emph{transversal fluctuation} of $\gamma$ by
$$\tf(\gamma) := \min\left\{w : \gamma\subseteq U_{r,w,z}\right\},$$
where $U_{r,w,z}$ is the strip of width $w$ around the interpolating line, i.e., the set of vertices $v \in \Z^2$ such that $v + t\cdot(-1,1)$ lies on the line $y=\tfrac{r+z}{r-z}\cdot x$ for some $t\in \R$ with $|t|\leq w/2$.

\subsection{Assumptions}

The general form of our assumptions is quite similar to the ones in the recent work \cite{watermelon} devoted to the study of geodesic watermelons, a path ensemble generalizing the geodesic. 
We start by recalling that $\nu$ is the distribution of the vertex weights and has support contained in $[0,\infty)$. %
The limit shape at scale one is the map $\mu^{\mathrm{shape}}: [-1,1] \to \R$ given by $w\mapsto \lim_{s\to\infty} s^{-1}\E[X_{s}^{ws}]$, where by scale one we mean that the map takes unit order values; we will consider scale $r$ shortly. It follows from standard super-additivity arguments that this limit exists (though possibly infinite if the upper tail of $\nu$ is too heavy) for each $w\in[-1,1]$ and that this map is \emph{concave} \cite[Proposition 2.1]{martin2006last}. Let
$$\mu = \mu^{\mathrm{shape}}(0)= \lim_{r\to\infty} r^{-1}\E[X_{r}]$$
be this map evaluated at zero. Also note from Theorems~\ref{t.tracy-widom convergence} and \ref{t.tails for exp LPP} that the fluctuations of $X_{r}$ around $\mu r$ can be expected to be on scale $r^{1/3}$. Finally, we point out that the limit shape map at scale $r$ (i.e., we evaluate the map at $w=z/r$ for $z\in[-r,r]$ and multiply it by $r$) in the exactly solvable models of Exponential, Geometric, Brownian, and Poissonian LPP is, up to translation by a constant times $r$ and scaling by constants,
\begin{equation}\label{seriesexpansion}
\sqrt{r^2-z^2}=r-\frac{z^{2}}{2r}-O\left(\frac{z^4}{r^{3}}\right);
\end{equation}
for instance, the limit shape map in Exponential LPP is $(\sqrt{1+w} + \sqrt{1-w})^2 = 2+2\sqrt{1-w^2}$, so the limit shape map at scale $r$ is $2r+2\sqrt{r^2-z^2}$. This will be relevant in motivating the form of our second assumption, which compares $\E[X^z_r]$ with the limit shape at scale $r$. Note that the first term of the right hand side of \eqref{seriesexpansion} denotes the expected linear growth of the model, while the second encodes a form of strong concavity of the limit shape. (We remark that the left-hand side in \eqref{seriesexpansion} should not be expected to generalize to LPP models with other vertex-weight distributions; indeed, that it takes the above form in the four mentioned integrable models is simply a consequence of the fact that they can all be seen as appropriate degenerations of geometric LPP. However, the right-hand side, at least the first two terms, should hold in much greater generality (up to different constant pre-factors) by curvature considerations of the limit shape: indeed, $X_{r}^{z} \approx r\mu^{\mathrm{shape}}(z/r) \approx \mu r - C(z/r)^2r$. See also Remark~\ref{r.tail exponent}.)

 Finally, we remark that the non-random fluctuation, i.e., how much the mean of $X_r^z$ falls below \eqref{seriesexpansion}, is expected to be $\Theta(r^{1/3})$, which is known in the aforementioned exactly solvable models. 

Given the setting, we state our assumptions; not all the assumptions are required for each of the main results, and we will specify which ones are in force in each case. We will elaborate more on the content of each assumption following their statements.

\begin{enumerate}
	\item \textbf{Limit shape existence:} The vertex weight distribution $\nu$ is such that $\mu < \infty$. \label{a.passage time continuity}

\item \textbf{Strong concavity of limit shape and non-random fluctuations:} There exist positive finite constants $\rho$, $G$, $H$, $g_1$, and $g_2$ such that, for large enough $r$ and $z\in[-\rho r, \rho r]$, \label{a.limit shape assumption}
	\begin{equation*}
	\E[X_r^z] \in \mu r -G\frac{z^2}{r} + \left[-H\frac{z^4}{r^3}, 0\right] + \left[-g_1r^{1/3}, -g_2 r^{1/3}\right].
	\end{equation*}
	The first three terms on the right hand side encode the limit shape and its strong concavity as in \eqref{seriesexpansion}, while the final interval captures the non-random fluctuation.

	\item \textbf{Upper bound on moderate deviation probabilities, uniform in direction:} There exists $\alpha > 0$ such that the following hold. Fix any $\varepsilon>0$. Then, there exist positive finite constants $c$, $\theta_0$, and $r_0$ (all depending on $\varepsilon$) such that, for $r>r_0$, $|z|\in [0,(1-\varepsilon)r]$, and $\theta>\theta_0$, \label{a.one point assumption}
	\begin{enumerate}
		\item $\displaystyle\P\left(X_{r}^z -\E[X_{r}^z] > \theta r^{1/3}\right) \leq \exp(-c\theta^{\alpha})$, \label{a.one point assumption upper}

		\item $\displaystyle\P\left(X_{r}^z -\E[X_{r}^z] < -\theta r^{1/3}\right) \leq \exp(-c\theta^{\alpha}).$ \label{a.one point assumption lower}
	\end{enumerate}

	\item \textbf{Lower bound on diagonal moderate deviation probabilities:} There exist positive constants $\delta$, $C$, $r_0$ such that, for $r>r_0$,
	
	\begin{enumerate}
		\item $\displaystyle\P\left(X_r - \mu r > Cr^{1/3}\right) \geq \delta$, \label{a.lower bound upper tail}
		\item $\displaystyle\P\left(X_r -\mu r < - Cr^{1/3}\right) \geq \delta.$ \label{a.lower tail lower bound}
	\end{enumerate}	
		\end{enumerate}
	These will be respectively referred to as Assumptions 1--4 in this paper. Assumption~\ref{a.passage time continuity}, which is known to be true under mild moment conditions on $\nu$, is stated to avoid any pathologies and will be in force throughout the rest of the paper without us explicitly mentioning it further.
	Assumption~\ref{a.one point assumption} is the a priori tail assumption that our work seeks to  improve on.  We will refer to the tail bounds as \emph{stretched exponential} though this term usually refers to $0<\alpha\leq 1,$ (which is the case of primary interest for us). 
	

	Assumption~\ref{a.passage time continuity} in fact follows from Assumption~\ref{a.one point assumption upper}, for the latter implies that $\nu([\theta,\infty)) \leq \exp(-c\theta^\alpha)$, for a possibly smaller $c$ and sufficiently large $\theta$ (see Remark~\ref{r.tail extends to nu}).

Observe that Assumption~\ref{a.limit shape assumption} is a mild relaxation of the form of the weight profile in all known integrable models, as we do not impose a lower order term of order $-z^4/r^3$ in the upper bound. Our arguments would also work if we replaced the third term $[-Hz^4/r^3, 0]$ of Assumption~\ref{a.limit shape assumption} with $[-Hz^4/r^3,\, Hz^4/r^3]$, but we have not included this relaxation so as to not introduce further complexity.

 The additional translation by $-\Theta(r^{1/3})$ in Assumption~\ref{a.limit shape assumption} for the non-random fluctuation will be a crucial ingredient (note that $\E[X_r]\le \mu r$ by super-additivity). As the reader might already be aware, non-random fluctuations are an important object of study and this will be further evident from their role in the arguments in this paper (in particular that they are the same scale as the random fluctuations) as well as in past work: see, for example, \cite{watermelon,BHS18}. For applications in FPP, see \cite{chatterjee2013universal} and \cite{auffinger2014simplified}; a  general theory to control such objects for general sub-additive sequences was developed in \cite{alexander1997approximation}.

	
We end this discussion by pointing out that Assumption~\ref{a.lower tail lower bound} follows from Assumptions~\ref{a.limit shape assumption} and \ref{a.one point assumption lower}; see Lemma~\ref{l.lower tail lower bound}.	This is essentially because by assumption $\mu r > \E[X_r] + \Theta(r^{1/3})$ and we have assumed deviation bounds from the expectation.
However, this style of argument does not work to derive Assumption~\ref{a.lower bound upper tail} from Assumptions~\ref{a.limit shape assumption} and \ref{a.one point assumption}, and this task seems more difficult.



\subsection{Main results}

The main contribution of this paper is to obtain the optimal upper and lower tail exponents for $X_r$ in terms of upper and lower bounds, starting from a selection of the assumptions just stated. Notice that all the assumptions except the first involve the weight fluctuations occurring on scale $r^{1/3}$, and our results essentially connect this fluctuation exponent of $1/3$ to the two tail exponents. Here are the precise statements.

\begin{maintheorem}[Upper-tail upper bound]\label{t.upper tail bootstrapping}
Under Assumptions~\ref{a.limit shape assumption} and \ref{a.one point assumption upper},
there exist positive constants $c$, $\zeta \in (0,\frac{2}{25}]$, $r_0$, and $\theta_0$ (all depending on $\alpha$) such that, for $r>r_0$ and $\theta_0 \leq \theta \leq r^{\zeta}$,
$$\P\left(X_r -\E[X_r] \geq \theta r^{1/3}\right)\leq \exp\left(-c\theta^{3/2}(\log \theta)^{-1/2}\right).$$
Further, $\zeta(\alpha)\to 0$ as $\alpha\to 0$, and $\zeta(\alpha) = \frac{2}{25}$ if  $\alpha\geq 1$.
\end{maintheorem}

\begin{maintheorem}[Upper-tail lower bound]\label{t.upper tail lower bound}
Under Assumptions~\ref{a.limit shape assumption} and \ref{a.lower bound upper tail} (the former only at $z=0$), there exist positive constants $c$, $\eta$ and $r_0$ such that, for $r>r_0$ and $\theta_0 < \theta < \eta r^{2/3}$,
$$\P\left(X_r - \E[X_r] \geq  \theta r^{1/3}\right) \geq \exp\big(-c\theta^{3/2}\big).$$
\end{maintheorem}

\begin{maintheorem}[Lower-tail upper bound]\label{t.optimal lower tail upper bound}
Under Assumptions \ref{a.limit shape assumption} and \ref{a.one point assumption}, there exist positive constants $c$, $r_0$, and $\theta_0$ (all depending on $\alpha$) such that, for $r> r_0$ and $\theta>\theta_0$,
$$\P\left(X_r - \E[X_r] \leq - \theta r^{1/3}\right) \leq \exp\left(-c\theta^{3}\right).$$
\end{maintheorem}

\begin{maintheorem}[Lower-tail lower bound]\label{t.lower tail lower bound}
Under Assumptions~\ref{a.limit shape assumption}, \ref{a.one point assumption}, and \ref{a.lower tail lower bound}, there exist positive constants $c$, $\eta$, $\theta_0$, and $r_0$ (all depending on $\alpha$) such that, for $r>r_0$ and $\theta_0 < \theta < \eta r^{2/3}$,
$$\P\left(X_r - \E[X_r]\leq  - \theta r^{1/3}\right) \geq \exp\left(-c\theta^{3}\right).$$
\end{maintheorem}

The constants $\theta_0$ and $r_0$ in the theorems should not be confused with the ones appearing in the assumptions. 

The aforementioned result on upper and lower bounds for last passage values when paths are constrained to lie inside a strip of given width will be stated ahead as Theorem~\ref{t.constrained lower tail bounds} after Section~\ref{s.discussion of techniques}, which elucidates the main arguments of Theorems~\ref{t.upper tail bootstrapping}--\ref{t.lower tail lower bound}.

As the reader might anticipate, one might be able to relax some of the assumptions, e.g., the precise form of Assumption ~\ref{a.limit shape assumption} should not be essential, and we expect our arguments to go through  under reasonable relaxations. For example, a polynomial lower order term in \eqref{seriesexpansion}, say of the form $|z|^{2+\delta}/r^{1+\delta}$ for some $\delta>0$, or the related assumption of local uniform strong concavity of the limit shape may be sufficient.
However we have not pursued this as we have sought to achieve the cleanest presentation to highlight the key geometric insights underlying the arguments.

Next we make some remarks and observations on the results, focusing mainly on aspects of Theorem~\ref{t.upper tail bootstrapping}.

\begin{remark}[Relation of tail exponents to fluctuation exponents]\label{r.tail exponent}
We have assumed that weight fluctuations occur on the scale $r^{1/3}$, and this is because this is thought to be the scale of fluctuations for LPP in two dimensions for all vertex weight distributions with sufficiently fast tail decay (for example, it is expected that having the fifth moment finite suffices \cite{gueudre2014revisiting}; as mentioned earlier and explained in Remark~\ref{r.tail extends to nu}, Assumption~\ref{a.one point assumption upper} ensures a stretched exponential decay on the vertex weight distribution, and so, in the setting of this paper, all moments are finite). The basis for this is the following heuristic. Let $\chi$ be the scale of weight fluctuations, and $\xi$ the scale of transversal fluctuations (i.e., the scale of the width of the smallest rectangle containing the geodesic), also called the wandering exponent. In any dimension, these exponents are expected to satisfy the KPZ relations $\chi = 2\xi -1$; this has been proven in FPP in \cite{chatterjee2013universal} contingent on the existence of a particular precise definition of the exponents. (The assumption of the existence of these exponents is non-trivial and it is an important open problem to prove it.) In two dimensions, the weight profile is additionally expected to exhibit Brownian fluctuations, which suggests $\chi = \xi/2$. These combine to imply $\chi = 1/3$ and $\xi = 2/3$.

Our arguments do not rely on $\chi = \frac{1}{3}$ in any crucial way, in the sense that if we made our assumptions with respect to a fluctuation scale of $r^\chi$ instead of $\smash{r^{1/3}}$, we expect that the same results would be obtained with different tail exponents which are explicit functions of $\chi$; more precisely, the algebra of our arguments yields that $3/2$ would be replaced by $1/(1-\chi)$ and $3$ by $2/(1-\chi) = 1/(1-\xi)$ (using the KPZ relation).

For LPP in higher dimensions, the KPZ relation is still expected to hold, but not the Brownian nature of the weight profile (as it is no longer a one-dimensional function). Thus it is natural to ask what the tail exponents would be in this case. As we elaborate in Section~\ref{s.future directions}, the algebra of our arguments suggests that for general dimension the upper and lower tail exponents should again respectively be $1/(1-\chi)$ and $2/(1-\chi) = 1/(1-\xi)$.

Even in two dimensions, the exponent $\chi$ need not be $1/3$ if the noise field is not \iid \cite{brito2019geodesic}. As seen in \cite{brito2019geodesic}, the KPZ relation need not hold in this setting, and there is no reason to expect Brownian fluctuations for the weight profile. But it is interesting to ask what relation may exist between the fluctuation exponents and the tail exponents. In fact, the argument for Theorem~\ref{t.upper tail lower bound} (lower bound on the upper tail), which we discuss in Section~\ref{s.discussion of techniques}, should apply quite generally, i.e., as long as correlation inequalities hold, and suggests that at least in models enjoying positive association the upper tail exponent may be $1/(1-\chi)$.

In a more classical setting, it is a nice exercise to use that the fluctuations of random walk of size $n$ are of order $n^{1/2}$ to conclude that the tail exponents in that case should be $1/(1-\frac{1}{2}) = 2$, again via the arguments for the upper tail ahead. For the lower bound, the argument does not make use of concentration of measure estimates, and thus provides a simple geometric indication of the source of the Gaussian distribution's tail exponent that we were not previously aware of. (The above prediction of a higher exponent for the lower tail does not apply since this is not a model of last passage percolation.)
\end{remark}

\begin{remark}[Suboptimal log factor in Theorem~\ref{t.upper tail bootstrapping}]\label{r.suboptimal upper tail}
The reader would have noticed that the tail in Theorem~\ref{t.upper tail bootstrapping} is not optimal, due to the appearance of $(\log \theta)^{-1/2}$. This arises due to the lack of sub-additivity of the sequence $\{X_r\}_{r\in\N}$ (which is super-additive instead), which necessitates considering a certain union bound; coping with the entropy from the union bound leads to the introduction of the factor of $(\log \theta)^{-1/2}$ in the exponent. We discuss this further in Section~\ref{s.discussion of techniques}.\end{remark}

\begin{remark}[$\zeta(\alpha)\to 0$ as $\alpha\to 0$]\label{r.tail extends to nu}
The tail exponent claimed in Theorem~\ref{t.upper tail bootstrapping} holds only for $\theta\leq r^{\zeta}$ for a positive $\zeta = \zeta(\alpha)$ with $\lim_{\alpha\to 0}\zeta(\alpha) = 0$, and as we will see now, this is indeed necessary. First note that Assumption~\ref{a.one point assumption} implies that the vertex weight distribution's upper tail decays with exponent at least $\alpha$; to see this, observe that $\smash{\P(X_{(r-1,r)}-\E[X_{(r-1,r)}] > - 0.5tr^{1/3})} > 1/2$ for all large enough $t$ by Assumption~\ref{a.one point assumption lower}, and so
\begin{align*}
\frac{1}{2}\cdot \P\left(\xi_{(r,r)} \geq t r^{1/3}\right)  &\leq \P\left(X_{(r-1,r)} -\E[X_{(r-1,r)}]> - 0.5tr^{1/3}, \xi_{(r,r)} \geq t r^{1/3}\right)\\
&\leq \P\left(X_{r}-\E[X_{r}] \geq 0.25 tr^{1/3}\right) \leq \exp(-ct^{\alpha}),
\end{align*}
using Assumption~\ref{a.one point assumption upper} in the last inequality, and bounding \smash{$\E[X_r]-\E[X_{(r-1,r)}]$} by $0.25 tr^{1/3}$. This holds for all $r$ and $t$ large enough; taking $r=r_0$ large enough for the bound to hold and letting $\xi_{(r,r)}$ be any random variable distributed according to $\nu$ shows that, for all large enough $t$,
$$\P\big(\xi_{(r,r)} \geq tr_0^{1/3}\big) \leq \exp(-ct^\alpha) \implies \P\big(\xi_{(r,r)} \geq t\big) \leq \exp(-\tilde ct^{\alpha}).$$
Conversely, assuming that  $\P\big(\xi \geq t\big) \geq \exp(-\tilde ct^{\alpha}),$ it follows that  Assumption~\ref{a.one point assumption upper} cannot hold with any power $\beta>\alpha$ for the entire tail.
Now recall, as mentioned after Theorem~\ref{t.tails for exp LPP}, that after a certain point the behavior of individual vertex weights is expected to govern the tail of point-to-point weights. So under the aforementioned assumption on $\xi_{(r,r)}$, an upper bound for $\zeta(\alpha)$ could be obtained by considering the value of $\zeta$ which solves 
$$\exp(-c\theta^{3/2}) = \exp(-c(\theta r^{1/3})^{\alpha})$$
for $\theta= r^{\zeta}$, which is $\zeta = 2\alpha/(9-6\alpha)$. This goes to zero as $\alpha\to 0$, as in Theorem~\ref{t.upper tail bootstrapping}.
\end{remark}

\begin{remark}[Intermediate regimes for upper tail]\label{r.other regimes}
While Theorem~\ref{t.upper tail bootstrapping} asserts the $3/2$ tail exponent up till $r^{\zeta}$, its proof will also show the existence of a number of ranges of $\theta$ in the interval $[r^{\zeta}, \infty)$ in which the tail exponent transitions from $3/2$ to $\alpha$. More precisely, there exists a finite $n$ and numbers $\alpha=\beta_1 < \beta_2<  \ldots < \beta_n = 3/2$ and $\infty = \zeta_1 > \zeta_2 >  \ldots  >\zeta_n = \zeta$ such that, for $j\in\intint{1,n-1}$ and $\theta\in [r^{\zeta_{j+1}}, r^{\zeta_j}]$,
$$\P\left(X_r -\E[X_r] \geq \theta r^{1/3}\right)\leq \exp\left(-c\theta^{\beta_{j}}\right).$$
Recursive expressions are also derived for the $\beta_j$ and $\zeta_j$ quantities; see Remark~\ref{r.precise other regimes}.

However, we believe that these intermediate regimes are an artifact of our proof, and that the true behavior is that the tail $\exp(-c\theta^{3/2})$ holds for $\theta$ till $r^{2\alpha/(9-6\alpha)}$, and $\exp(-c(\theta r^{1/3})^{\alpha})$ after (as in Remark~\ref{r.tail extends to nu}). Note also that for $\alpha=1$, $r^{2\alpha/(9-6\alpha)} = r^{2/3}$, matching Theorem~\ref{t.tails for exp LPP}.
\end{remark}

From the point of view of applications in models where some initial estimates resembling the ones in our assumptions can be obtained, the most useful part of the framework developed here is the argument for Theorem~\ref{t.optimal lower tail upper bound}, the upper bound on the lower tail. This is because in most integrable models bounds on the upper tail are technically much easier than bounds on the lower tail. Indeed, for example, upper tail bounds for the free energy of the log-gamma polymer model were developed in \cite{barraquand2021fluctuations}, but lower tail bounds are still not available. In fact, the only polymer model where lower tail bonds are available is the O'Connell-Yor semi-discrete polymer model, for which \cite{landon2022tail} actually obtains lower tail bounds for the free energy with the correct tail exponent of $3$ by adapting the argument presented here (after obtaining an initial bound with tail exponent $3/2$ from an explicit formula for the moment generating function in a stationary version of the model).

But while our main results are stated only for the $(1,1)$ direction, i.e., $z=0$, since the basic ideas are conveyable in the most elegant fashion in this case, often in applications one needs the bounds to hold in all directions. We remark on this next and how one might extend our results.

\begin{remark}[Extending to other values of $z$]\label{r.other directions}
We have stated our results for the last passage value to $(r,r)$, but some also extend to $(r-z, r+z)$ for certain ranges of $z$. For the upper tail the argument of Theorem~\ref{t.upper tail bootstrapping} also applies for $|z|= O(r^{2/3})$, while Theorem~\ref{t.upper tail lower bound} extends to all $|z| = O(r^{5/6})$; as mentioned after the assumptions, the source of the $5/6$ is that for $z$ of this order, the upper and lower bounds of Assumption~\ref{a.limit shape assumption} differ by the weight fluctuation order, i.e., $r^{1/3}$.
Regarding the upper bound on the lower tail, the argument for Theorem~\ref{t.optimal lower tail upper bound} does not conceptually rely on $z=0$, but formally uses a result from \cite{watermelon} which is not proven  for $z\neq 0$. The latter result can be extended to any $z$ without much difficulty (see below). Finally, the argument for Theorem~\ref{t.lower tail lower bound} applies for $|z|\leq r^{5/6}$.

To obtain the bounds in directions corresponding to $z\gg r^{5/6}$, eg. $z=\Theta(r)$, one needs to assume something like Assumption~\ref{a.limit shape assumption} to hold in the corresponding other directions. We will give one possible form of the modification to Assumption~\ref{a.limit shape assumption} in Section~\ref{s.construction in other directions}, as well as a fairly detailed sketch of the modifications to the argument needed to obtain the upper bound on the lower tail,  which, as noted above, for applications is the bound most useful to have. While the arguments for the other directions should go through as well, we do not provide further commentary on them.
\end{remark}

The set of assumptions we adopt bears similarities to the ones that have appeared in the past literature on FPP, some of which we discussed in Remark~\ref{r.tail exponent}.  The most prominent of these include the work of Newman and coauthors (see e.g. \cite{newman1995divergence,auffinger201750,newman1995surface}) which investigated the effect of limit shape curvature assumptions on the geometry of geodesics and the fluctuation exponents.  More recently, the previously mentioned work \cite{chatterjee2013universal} of Chatterjee assumed a strong form of existence of the exponents governing geometric and weight fluctuations of the geodesics and verified the KPZ relation between them; see also \cite{auffinger2014simplified}. Subsequently \cite{damron2014busemann,alexander2020geodesics,gangopadhyay2020fluctuations} studied geodesics and bi-geodesics under related assumptions. 

Inspired by this, recently, results in the exactly solvable cases of LPP have been obtained, relying merely on inputs analogous to the ones stated in the assumptions. See, for example, the very recent work \cite{watermelon} which develops the theory of \emph{geodesic watermelons} under a similar set of assumptions to deduce properties of all known integrable lattice models.  Other examples include \cite{BHS18,slow-bond,FO17}, which work in the specific case of LPP with exponential weights; and \cite{brownianLPPtransversal,hammond2017modulus,hammond2017patchwork}, in which geometric questions in the semi-discrete model of Brownian LPP are studied.


An intriguing and novel aspect of our arguments is the use of the concentration of measure phenomena for sums of independent stretched exponential random variables, which is in fact at the heart of this paper. General concentration results have, of course, been widely investigated in recent times \cite{boucheron2013concentration}, but they have not previously played a central role in studies of LPP. On the other hand, concentration of measure has played a more significant role in FPP. We mention here \cite{damron2014subdiffusive} which proves exponential concentration of the passage time on a subdiffusive scale and the related line of work bounding the variance \cite{kesten1986aspects,benjamini2003first,benaim2008exponential,damron2015sublinear}. Also related is \cite{chatterjee2013central} which proves a central limit theorem for certain constrained first passage times. We point the reader to \cite[Section 3]{auffinger201750} for a more in depth survey.

A common theme in concentration of measure is that sums of independent random variables have behavior which transitions, as we extend further into the tail, from being sub-Gaussian to being governed by the tail decay of the individual variables. When the variables have stretched exponential tails, a precise form of this is a bound that is a generalization of Bernstein's inequality for subexponential random variables. Though such results are not unexpected, the recent article \cite{stretched-exp-concentration} explicitly records many extensions of concentration results for sums of sub-Gaussian or subexponential random variables to the stretched exponential case with a high dimensional statistics motivation, in a form particularly convenient for our application. 

We next move on to an outline of the key ideas driving our proofs.
\subsection{A brief discussion of the arguments}
\label{s.discussion of techniques}

Before turning to the ideas underlying our arguments, we deal with some matters of convention. We will use the words ``width'' and ``height'' to refer to measurements made along the antidiagonal and diagonal respectively. So, for example, the set of $(x,y)\in\Z^2$ such that $2\leq x+y\leq 2r$ and $|x-y|\leq \ell r^{2/3}$ is a parallelogram of height $r$ and width $\ell r^{2/3}$. This usage will continue throughout the article.

In the overview we will at certain moments make use of a few refined tools, which have appeared previously in \cite{watermelon}, and whose content is explained informally in this section; their precise statements are gathered in Section~\ref{s.tools} ahead.

Now we turn to the mathematical discussion. The flavors of our arguments are different for the upper and lower bounds on the two tails. Super-additivity, in various guises, plays a recurring role in all except the upper bound on the lower tail. In all the arguments a parameter $k$ appears which plays different roles, but is essentially always finally set to be a multiple of $\theta^{3/2}$, where $\theta$ measures the depth into the tail we are considering. The reader should keep in mind this value of $k$ in the discussion. Also, we assume without loss of generality that $\alpha\leq 1$ in this section.

We briefly give a version of a common theme which underlies the different arguments, namely of looking at smaller scales, which further explains why we take $k=\Theta(\theta^{3/2})$. Consider a geodesic path from $(1,1)$ to $(r,r)$ which attains a weight of $\mu r + \theta r^{1/3}$ for large $\theta$ (the following also makes sense for $-\theta$). If we look at a given $1/k$-fraction of the geodesic, that fraction's weight should be close to $\mu r/k + \theta r^{1/3}/k$ if the geodesic gains weight roughly uniformly across its journey; but on the other hand, KPZ fluctuation dictates that the fraction's weight should typically be $\mu r/k + \Theta(r/k)^{1/3}$. So we look for a scale at which the typical behavior is not in tension with the notion of the geodesic's weight being spread close to uniformly over much of its journey. This means finding $k$ such that $\theta r^{1/3}/k$ and $(r/k)^{1/3}$ are of the same order, which occurs if $k=\Theta(\theta^{3/2})$. 

Now we come to the detailed descriptions.
\subsubsection*{Upper bound on upper tail.} We start by discussing a simplified argument for the upper tail of the upper bound to illustrate the idea of bootstrapping. The starting point is a concentration of measure phenomenon for stretched exponential random variables alluded to before. 
More precisely, sums of independent stretched exponential random variables have the same qualitative tail decay deep in the tail as that of a single one (see Proposition~\ref{p.stretched exponential concentration} ahead). Not so deep in the tail lies a regime of Gaussian decay, but we will never be in this regime in our arguments.

Let $X_{r/k,i}$ be the last passage value from $i(r/k, r/k)+(1,0)$ to $(i+1)(r/k,r/k)$. Suppose, for purposes of illustration, that we actually had that $X_r$ are \emph{sub}-additive rather than super-additive, i.e., we had that $X_r\leq \sum_{i=1}^k X_{r/k,i}$. Each $X_{r/k,i}$ fluctuates at scale $(r/k)^{1/3}$, and 
\begin{equation}\label{e.mean difference}
\Big|\sum_{i=1}^k\E[X_{r/k,i}] - \E[X_r]\Big| \leq k\cdot C(r/k)^{1/3} = Ck^{2/3}r^{1/3},
\end{equation}
using Assumption~\ref{a.limit shape assumption}. So under this illustrative sub-additive assumption we would have
\begin{align}
\P\left(X_r - \E[X_r] > \theta r^{1/3}\right) 
&\leq \P\left(\sum_{i=1}^k (X_{r/k,i} - \E[X_{r/k,i}]) > \theta r^{1/3} - Ck^{2/3}r^{1/3}\right) \nonumber\\
&\leq \P\left(\sum_{i=1}^k (X_{r/k,i} - \E[X_{r/k,i}]) > \frac{1}{2}\theta k^{1/3}(r/k)^{1/3}\right), \label{e.boostrap exposition}
\end{align}
where the last inequality holds for $k\leq (2C)^{-3/2}\theta^{3/2}$, which dictates our choice of $k$. Now by Assumption~\ref{a.one point assumption upper} we know that 
\begin{align*}
\P\left(X_{r/k,i} - \E[X_{r/k, i}] > \theta (r/k)^{1/3}\right) &\leq \exp(-c\theta^\alpha)\\
\implies \P\left(X_{r/k,i} - \E[X_{r/k, i}] > \theta k^{1/3} (r/k)^{1/3}\right) &\leq \exp(-c\theta^\alpha k^{\alpha/3}).
\end{align*}
Because sums of stretched exponentials have the same deep tail decay as a single one, \eqref{e.boostrap exposition} shows that the probability that $X_r - \E[X_r]$ is greater than $\theta r^{1/3}$ is essentially like that of $X_{r/k} - \E[X_{r/k}]$ being greater than $\smash{\theta k^{1/3} (r/k)^{1/3}}$, which is at most $\smash{\exp(-c\theta^\alpha k^{\alpha/3})}$. 
This gives an improved tail exponent of $3\alpha/2$ for the point-to-point's upper tail, compared to the input of $\alpha$, since $k$ can be at most $O(\theta^{3/2})$.

We can now repeat this argument, with the improved exponent as the input, and obtain an output exponent which is greater by a factor of $3/2$, and we can continue doing so as long as the input exponent is at most 1. If we perform the argument one last time with the input exponent as 1, we obtain the optimal exponent of $3/2$. 

The reason we require the input exponent to be at most 1 is that, beyond this point, the concentration behavior changes: for $\alpha\leq 1$ the deep tail behavior of a sum of independent stretched exponentials is governed by the event that a single variable is large, while for $\alpha>1$ the behavior is governed by the event that the deviation is roughly equidistributed among all the variables. This is a result of the change of the function $x^{\alpha}$ from being concave to convex as $\alpha$ increases beyond 1.
%
 More precisely, suppose $\alpha \in (1,3/2]$ is the point-to-point tail exponent and let us accept the equidistributed characterization of the deep tail (as is proved in \cite[Theorem~3.1]{stretched-exp-concentration}). Then the probability \eqref{e.boostrap exposition} would be at most the probability that each of the $k$ variables \smash{$X_{r/k,i} - \E[X_{r/k,i}]$} is at least $(\theta k^{1/3}/k) (r/k)^{1/3} = \theta k^{-2/3}(r/k)^{1/3}$, which is in turn bounded by 
$$\exp\left(-ck\cdot(\theta k^{-2/3})^{\alpha}\right) = \exp\left(-c\theta^{\alpha} k^{1-2\alpha/3}\right).$$
By taking $k = \eta \theta^{3/2}$, which, as mentioned earlier, is the largest possible value we can take, we see that this final expression is $\exp(-c\theta^{3/2})$. In other words, the exponent of $3/2$ is a natural fixed point for the bootstrapping procedure.

Now we turn to addressing the simplifications we made in the above discussion. Handling them correctly makes the argument significantly more complicated and technical, and reduces the tail from $\theta^{3/2}$ to $\theta^{3/2}(\log\theta)^{-1/2}$.

One simplification we skipped over is that the improvement in the tail bound after one iteration only holds for $\theta\leq r^{2/3}$ and not the entire tail (since $k$, the number of parts that the geodesic to $(r,r)$ is divided into, can be at most $r$, and $k = \Theta(\theta^{3/2})$), which is a slight issue for the next round of the iteration. This is handled by a simple truncation.

But the main difficulty is that the $X_r$ are super-additive, not sub-additive. To handle this, we consider a grid of height $r$ and width $\mathrm{poly}(\theta)\cdot r^{2/3}$, where $\mathrm{poly}(\theta)$ is a term which grows at most as a power of $\theta$. This width is set such that, with probability at most $\exp(-c\theta^{3/2})$, the geodesic exits the grid, using the bound recorded in Proposition~\ref{p.tf} ahead on the transversal fluctuation; this allows us to restrict to the event that the geodesic stays within the grid. Intervals in the grid have width $(r/k)^{2/3}$ and are separated by a height of $r/k$.

\begin{figure}
\centering
\includegraphics[width=0.4\linewidth]{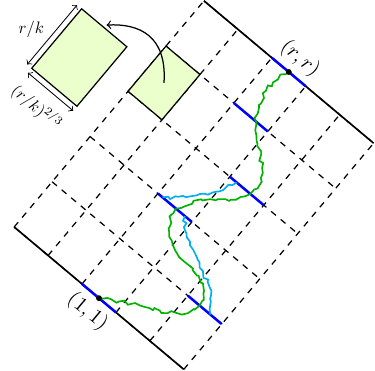}
\caption{In green is depicted the heaviest path which passes through the selection of intervals in blue. The cyan curve between the second and third (similarly the third and fourth) blue intervals is the heaviest path with endpoints on those intervals. Because these consecutive cyan paths do not need to share endpoints, the weight of the green path is at most the sum of the interval-to-interval weights defined by the blue intervals, which provides the substitute sub-additive relation.}
\label{f.grid proof overview}
\end{figure}

The utility of the grid is that $X_r$ can be bounded by a sum of interval-to-interval weights in terms of the intervals of the grid that the geodesic passes through; this bound can play the role of a sub-additive relation. See Figure~\ref{f.grid proof overview}. Then, just as we had a tail bound above for \smash{$X_{r/k,i}$} to bootstrap, a requisite step is to obtain an upper bound on the upper tail of the interval-to-interval weight, using only the point-to-point estimate available. We do this in Lemma~\ref{l.Z tail bound} with the basic idea that the interval-to-interval weight being high will cause a point-to-point weight, from ``backed up'' points (i.e., points taken to be further behind the first interval and further ahead than the second interval), to also be high; see Figure~\ref{f.from point to interval} for a more detailed illustration of the argument (such an argument of backing up has previously been implemented in \cite{slow-bond,watermelon}).

With the interval-to-interval tail bound, we discretize the geodesic by considering which sequence of intervals it passes through, and bound the highest weight through a given sequence by the sum of interval-to-interval weights. This uses the bootstrapping idea and yields an improved tail estimate for the highest weight through a given sequence. Later we will take a union bound over all possible sequences of intervals; this union bound is what leads to the appearance of the suboptimal $(\log\theta)^{-1/2}$ in the bound as mentioned in Remark~\ref{r.suboptimal upper tail}.

This strategy requires handling paths which are extremely ``zig-zaggy''; to show that these paths are not competitive, we need upper bounds on upper tails of point-to-point weights, i.e. $X_r^z$, in a large number of directions indexed by $z$, though we are only ultimately proving a bound for paths ending at $(r,r)$. (Recall that $X_r^z$ is the weight to $(r-z,r+z)$ from $(1,1)$.) Further, in order to repeat the iterations of the bootstrap, the bounds in other directions must also be improving with each iteration. To achieve this, we in fact bound the deviations not from $\E[X_r^z]$ (which to second order is $\mu r -Gz^2/r$) in the $j^{th}$ round of iteration, but from the bigger $\mu r - \lambda_jGz^2/r$, for a $\lambda_j\leq 1$ which decreases with the iteration number $j$. By adopting this relaxation we are able to obtain the improvement in the tail for all the required $z$ with each iteration, which appears to be difficult if one insists that $\lambda_j = 1$ for all $j$.  

A similar grid construction has been used previously, for example to obtain certain tail bounds in \cite{watermelon}, to bound the number of disjoint geodesics in a parallelogram in \cite{BHS18}, and to study coalescence of geodesics in \cite{coalescence}. 

\subsubsection*{Lower bound on upper tail}
This is the easiest of the four arguments. Recall that we have $C$ and $\delta$ from Assumption~\ref{a.lower bound upper tail} such that $\P(X_{r/k} > \mu r/k + C(r/k)^{1/3}) \geq \delta$, and let \smash{$X_{r/k,i}$} be as in \eqref{e.mean difference}. By the super-additivity that the $X_r$ genuinely enjoy, for any $k$ it holds that $X_r \geq \sum_{i=1}^k \smash{X_{r/k,i}}$. Choosing $k$ to be an appropriate multiple of \smash{$\theta^{3/2}$}, we obtain
\begin{align*}
\P\left(X_r  > \mu r + \theta r^{1/3}\right) 
&\geq \prod_{i=1}^k \P\left(X_{r/k,i} > \mu r/k+ C(r/k)^{1/3} \right)
\geq \delta^{k} = \exp(-c\theta^{3/2}).
\end{align*}
Here we used the independence of $\smash{X_{r/k,i}}$, but note that it would have sufficed for our purposes to have that they are positively associated, by the FKG inequality. Replacing $\mu r$ by $\E[X_r]$ is a simple application of Assumption~\ref{a.limit shape assumption}.

\subsubsection*{Upper bound on lower tail}
The illustrative argument using sub-additivity given above for the upper bound on the upper tail is actually correct for the upper bound on the lower tail, as the super-additivity of $X_r$ is in the favorable direction in this case. But, as we saw there, the approach can only bring the tail exponent up to $3/2$, and not $3$. This is essentially because that argument focuses on the weight of a single path, while the exponent of $3$ for the lower tail is a result of \emph{all} paths having low weight. Thus our strategy to prove the stronger bound is to construct $\theta^{3/2}$ disjoint paths moving through independent parts of the space, each suffering a weight loss of $\theta r^{1/3}$. By the discussion above and independence, the probability of each of them being small can be bounded by $\exp(-c\theta^{3/2}\cdot\theta^{3/2})=\exp(-c\theta^3)$.

\begin{figure}
\centering
\includegraphics[width=0.35\linewidth]{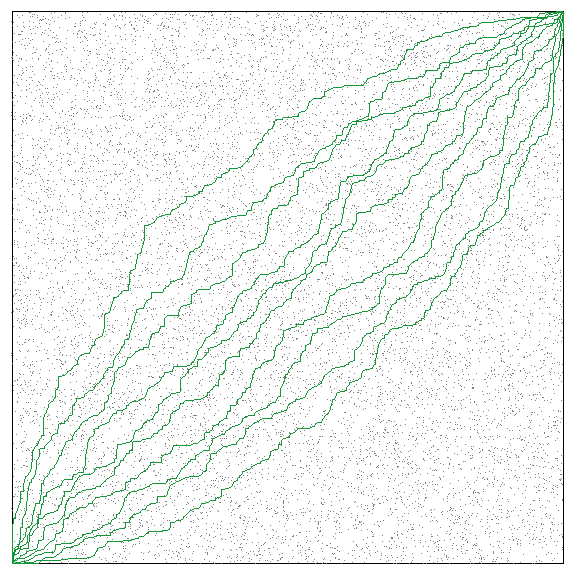}
\caption{A simulation of the $k$-geodesic watermelon in the related model of Poissonian last passage percolation for $k=10$.}
\label{f.watermelon}
\end{figure}

To do this formally, we rely on an important ingredient from \cite{watermelon}, which studies the weight and geometry of maximal weight collections of $k$ disjoint paths in $\intint{1,r}^2$, called \emph{$k$-geodesic watermelons}. See Figure~\ref{f.watermelon}. It is shown that these $k$ paths typically are each of weight $\mu r - Ck^{2/3}r^{1/3}$, and that they have a collective transversal fluctuation of order $k^{1/3} r^{2/3}$. In fact, the following quantitative bound on the weight $W_r^k$ of the $k$-geodesic watermelon is proved there via a direct multi-scale construction of disjoint paths with correct order collective weight, which we formally state ahead as Theorem~\ref{t.flexible construction}: 
\begin{equation}\label{e.melon weight bound}
\P\left(W_r^k \leq \mu kr - Ck^{5/3}r^{1/3}\right) \leq \exp(-ck^2),
\end{equation}
for all $k \leq \eta r$ for a small constant $\eta>0$. 
 We give a brief overview of the construction in Section~\ref{s.melon construction overview} due to its conceptual importance in the argument for the lower tail bound.

For our purposes we observe that, for any $k\in \N$,
$$\P\left(X_r < \mu r - \theta r^{1/3}\right) \leq \P\left(W_r^k < \mu k r - \theta k r^{1/3}\right).$$
Taking $k=\eta\theta^{3/2}$ and noting that then $k\theta$ is of order $k^{5/3}$ and that $\theta < r^{2/3}\implies k< \eta r$ shows that
$$\P\left(X_r < \mu r - \theta r^{1/3}\right) \leq \exp(-ck^2) = \exp(-c\theta^3),$$
and it is a simple matter to replace $\mu r$ by $\E[X_r]$ by possibly reducing the constant $c$.

However the framework in \cite{watermelon} works with strong tail bounds on the one point weight of the kind we are seeking to prove in this paper.
So in order to access the bound \eqref{e.melon weight bound} from \cite{watermelon} we need to deliver the required inputs starting from our assumptions. 
%
There are three inputs required. The first is the following: 
\begin{enumerate}
\item Limit shape bounds, which we have by Assumption~\ref{a.limit shape assumption}. 
\end{enumerate}
The next two inputs concern the maximum weight over all midpoint-to-midpoint paths constrained to lie in a given parallelogram $U= U_{r,\ell, z}$ of height $r$, width \smash{$\ell r^{2/3}$}, and opposite side midpoints $(1,1)$ and $(r-z,r+z)$. We will call such weights ``constrained weights''. 
\begin{enumerate}\setcounter{enumi}{1}
\item  An exponential upper bound on the constrained weight's lower tail, which we will arrive at by bootstrapping. To elaborate, by using Assumption~\ref{a.one point assumption lower} and the previously mentioned Proposition~\ref{p.tf} on the transversal fluctuation of the unconstrained geodesic, we can obtain an initial stretched exponential upper bound (\eqref{e.constrained lower tail} of Proposition~\ref{p.constrained statements} ahead) on the constrained weight's lower tail. Then, via a bootstrapping argument as above, we can upgrade this to a tail with exponent $3/2$ (see Proposition~\ref{p.bootstrap}).
\item A lower bound on the mean of constrained weights using the above tail, provided by \eqref{e.constrained mean} of Proposition~\ref{p.constrained statements}. 
\end{enumerate}


\subsubsection*{Lower bound on lower tail}

A detail about the construction described, which is captured in its formal statement Theorem~\ref{t.flexible construction}, is that it fits inside a strip of width $4k^{1/3}r^{2/3}$ around the diagonal. To lower bound the lower tail probability, this suggests that we need to focus on paths which remain in the strip of this width (again we will be setting $k$ to be a constant times $\theta^{3/2}$). Essentially this is because a consequence of the parabolic weight loss of Assumption~\ref{a.limit shape assumption} is that, with high probability, \emph{any} path (not just the geodesic) which exits the strip of width $k^{1/3}r^{2/3}$ suffers a loss of $(k^{1/3}r^{2/3})^2/r = k^{2/3}r^{1/3}$, which is of order $\theta r^{1/3}$. This is captured more precisely in Theorem~\ref{t.tf general} ahead.

Similar to the argument for the upper bound on the upper tail, we consider a grid where each cell has height $r/k$ and width $(r/k)^{2/3}$, but with overall width $k^{1/3}r^{2/3}$. This gives $k$ cells in each column and in each row, for a total of $k^2$ cells. See Figure~\ref{f.lower tail lower bound}.

Now consider the event that, for each interval in the grid, the maximum weight from that interval to the next row of intervals is less than $\mu r/k - C(r/k)^{1/3}$, and that the maximum weight of a path which exits the grid is at most $\mu r - Ck^{2/3}r^{1/3}$. This is an intersection of decreasing events, and on this event $X_r$ is at most $\mu r - Ck^{2/3}r^{1/3}$: if it exits the grid it suffers a loss of $Ck^{2/3}r^{1/3}$ and if it stays in the grid it undergoes a loss of at least $C(r/k)^{1/3}$ for each of the $k$ rows. Now if we know that there is a constant order probability (say $\delta>0$) lower bound on the event that a single interval-to-line weight is low, the FKG inequality (along with Theorem~\ref{t.tf general} to lower bound the probability of parabolic weight loss when exiting the grid) provides a lower bound of order \smash{$\delta^{k^2}$} on the described event's probability; setting $k$ to be a multiple of $\theta^{3/2}$ will complete the proof.

\begin{figure}
\centering
\includegraphics[width=0.4\textwidth]{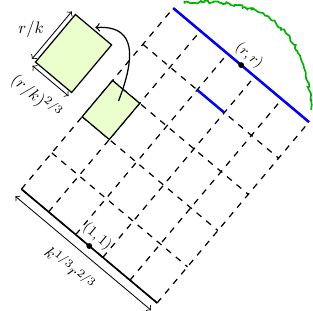}
\caption{The grid of $k^2$ intervals for the lower bound of the lower tail. An interval and the following row of intervals are blue: consider the event that the heaviest path from the former to the latter is at most $\mu r/k - C(r/k)^{1/3}$. To prove that this has positive probability, we make use of parabolic curvature of the weight profile (shown in green) to argue that if the endpoint on the row is too extreme, it will typically suffer the loss we want; a separate backing up argument is employed for when the endpoint is near the center where the parabolic weight loss is not significant.}
\label{f.lower tail lower bound}
\end{figure}

To implement this we need a lower bound on the probability that the interval-to-line weight is small using the point-to-point lower bound of Assumption~\ref{a.lower tail lower bound}. This is Lemma~\ref{l.int to int lower bound}. The proof proceeds in two steps. First, a stepping back strategy as earlier (i.e., obtaining interval-to-interval estimates by stepping back to consider a point-to-point weight between backed up points) gives a constant lower bound on the interval-to-interval weight's lower tail for intervals of size $\varepsilon r^{2/3}$, for some small $\varepsilon>0$
(as will be clear from the precise argument, the smallness of $\e$ is crucial for this).
By the FKG inequality, this is upgraded to a bound for intervals of length $r^{2/3}$; essentially, if each of the intervals are divided into $\varepsilon^{-1}$ intervals of size $\varepsilon r^{2/3}$, and all $\varepsilon^{-2}$ pairs of intervals have small weight (which is an intersection of decreasing events), then so must the original intervals. To get from this to an interval-to-line bound we again argue based on FKG. We divide the line into $r^{1/3}$ many intervals of size $r^{2/3}$ each. We can ensure that the weight is low whenever the destination interval is one of a constant number near $(r,r)$ using the previous bound, and for the rest the parabolic curvature ensures that it is so likely to be low that the FKG inequality gives a positive lower bound independent of $r$, in spite of considering an intersection of $r^{1/3}$ many events; see Figure~\ref{f.lower tail lower bound}.



\subsection{Tails for constrained weight}\label{s.constrained optimal tails}


The ideas described in the previous section can be applied slightly more generally to yield the following theorem on the lower tail of the constrained weight $Y_r^U$ of the best path from $(1,1)$ to $(r,r)$ constrained to stay inside a parallelogram $U$. (Quantities labeled with $Y$ will throughout this paper be LPP values under some constraint.)

Recall that $U = U_{r,\ell, z}$ denotes a parallelogram of height $r$, width $\ell r^{2/3}$, and opposite side midpoints $(1,1)$ and $(r-z,r+z)$, defined to be the set of vertices $v = (v_x,v_y) \in \Z^2$ such that $v + t(-1,1)$ lies on the line $y=\tfrac{r+z}{r-z}\cdot x$ for some $t\in \R$ with $|t|\leq \ell \smash{r^{2/3}/2}$, and $2\leq v_x+v_y\leq 2r$.
%
%
Let $\smash{Y_r^U}$ be the maximum weight among all paths from $(1,1)$ to $(r-z,r+z)$ constrained to be inside $U$.  The notation $U_{r,\ell,z}$ will be used for parallelograms throughout this article. In the next result we take $z=0$.

Estimates on constrained weights have been crucial in several recent advances, see \cite{slow-bond}. The following theorem proves a sharp estimate on the tail as a function of the aspect ratio of the parallelogram, measured on the characteristic KPZ scale.
\begin{maintheorem}\label{t.constrained lower tail bounds}
 Under Assumptions~\ref{a.limit shape assumption}, \ref{a.one point assumption}, and \ref{a.lower tail lower bound}, there exist finite positive constants $c_1$, $c_2$, $\eta$, $C$, $\theta_0$, and $r_0$ (all independent of $\ell$) such that, for $z=0$, ranges of $\theta$ to be specified, $r>r_0$, and $C\theta^{-1}\leq \ell \leq 2r^{1/3}$,
$$\exp\left(-c_1\min(\ell\theta^{5/2}, \theta^3)\right) \leq \P\left(Y_r^U - \mu r \leq - \theta r^{1/3}\right) \leq \exp\left(-c_2 \min(\ell\theta^{5/2}, \theta^3)\right);$$
the second inequality holds for $\theta>\theta_0$ while the first holds for $\theta_0<\theta < \eta r^{2/3}$.
If $\ell$ is bounded below by a constant $\varepsilon>0$ independent of $\theta$, we can replace $\mu r$ by $\E[Y_r^U]$ for $r>\tilde r_0(\varepsilon)$ and with $c_1$ depending on $\varepsilon$.
\end{maintheorem}

We note that Theorem~\ref{t.optimal lower tail upper bound} and Theorem~\ref{t.lower tail lower bound} are implied by Theorem~\ref{t.constrained lower tail bounds} by taking $\ell = 2r^{1/3}$, as then the width of $U$ is $2r$, i.e., there is no constraint since the entire square $\intint{1,r}^2$ may be used by the geodesic; in other words, $Y_r^U = X_r$ deterministically.

We also remark that the transition from $\ell \theta^{5/2}$ to $\theta^3$ occurs when $\ell$ becomes of order $\theta^{1/2}$; this matches the belief (which comes from the parabolic curvature) that the geodesic, conditioned on its weight being less than $\mu r - \theta r^{1/3}$, will have typical transversal fluctuations of order $\theta^{1/2}r^{2/3}$.

The proof idea of Theorem~\ref{t.constrained lower tail bounds} is a refinement of those of Theorems~\ref{t.optimal lower tail upper bound} and \ref{t.lower tail lower bound} described above, by picking the number of paths of average separation $k^{-2/3} = \theta^{-1}$ to be packed inside $U$, which turns out to be $\min(\ell k^{2/3}, k)$ (rather than $k$ as before). We omit further outline to avoid repetition.

\subsection{Related work}\label{s.related work}

The main tools we use in our arguments are the super-additivity of the $X_r$ (i.e., $X_{r+j} \geq X_r + X_{(r+1,r),(r+j,r+j)}$), geodesic watermelons, and concentration of measure results for sums of independent stretched exponential random variables. We have discussed aspects of the latter two that have appeared in various works, and here we briefly overview the first, i.e., super-additivity.


Not surprisingly, super-additivity of the weight has been an important tool in other investigations of non-integrable models; for example, the proof of the almost sure existence of a deterministic limit for $X_r/r$ as $r\to\infty$ under a wide class of vertex distributions goes via Kingman's sub-additive theorem.  Super-additivity was also crucial in \cite{ledoux2018law}, where a law of iterated logarithm for $X_r$ was proved. More precisely, for exponential weights, $\limsup_{r\to\infty} (X_r - 4r)/(r^{1/3}(\log\log r)^{2/3})$ was shown to almost surely exist and be a finite, positive, deterministic constant. Super-additivity only aids in proving a result for the $\limsup$, and so the result on the $\liminf$ in \cite{ledoux2018law} is weaker. This was addressed in \cite{basu2019lower}, where the lack of sub-additivity was handled by shifting perspective to also consider point-to-line passage times, which, as we have outlined, we will do in the present article as well.
A usage of super-additivity was also made by Sepp\"al\"ainen in \cite{seppalainen1998coupling} and shortly after by Johansson in \cite{johansson2000shape}, where, from a limiting large deviation theorem for the upper tail, it was pointed out that the same gives an explicit bound for finite $r$ by a super-additivity argument. Briefly, and again in the context of Exponential LPP, the observation is that for every $r$ and every $N\geq 1$,
\begin{equation}\label{e.johansson super-additivity}
\P\big(X_r > \theta\big)^N \leq \P\big(X_{Nr} > N\theta\big) \implies \P\big(X_r > \theta\big) \leq \lim_{N\to\infty} \left[\P\big(X_{Nr} > N\theta\big)\right]^{1/N},
\end{equation}
and the latter limit was shown to exist and explicitly identified in \cite{seppalainen1998coupling}. 
In a sense our arguments are dual to that of \eqref{e.johansson super-additivity}; while \eqref{e.johansson super-additivity} uses super-additivity to go to \emph{larger} $r$ in order to obtain a bound, our arguments use super-additivity to reason about \emph{smaller} $r$ to obtain a bound. 

Finally, we mention the recent work \cite{emrah2020right} which proves a sharp upper bound (i.e., with the correct coefficient of $4/3$ as in \eqref{e.TW tails}) on the right tail of $X_r$ (centred by $\mu r= 4r$ and appropriately scaled) in exponential LPP via more probabilistic arguments, rather than precise analysis of integrable formulas. The technique utilizes calculations in an increment-stationary version of exponential LPP (where the vertex weight on the boundaries of $\Z_{\geq 0}^2$ differ in distribution from the rest) and a moment generating function identity specific to this model---features absent in the general setting under consideration in this article.



\subsection{Future directions}\label{s.future directions}
 This work leads to several research directions, some of which we outline below.
A natural question is whether the stretched exponential tails of Assumption~\ref{a.one point assumption} may be weakened, for example to polynomially decaying tails. An important aspect of our arguments is that we are able to reach the exponent of $3/2$ in finitely many iterations of the bootstrap. While it should not be difficult to show via a bootstrapping argument that, starting from a polynomial decay of a given degree, one can reach a polynomial decay of any given higher degree after finitely many rounds, it appears to us non-trivial to bridge the gap and reach a superpolynomial decay, such as the stretched exponential decay of Assumption~\ref{a.one point assumption}, in finitely many steps.

Another interesting direction of inquiry is whether these methods can be used to study directed last passage percolation model tails in higher dimensions. This appears to us doable in principle, given suitable assumptions on fluctuation scales and the limit shape.

It is worthwhile to point out that the basic super-additive argument for the lower bound on the upper tail explained above does not have any dependencies on the limit shape or dimension, and, assuming a weight fluctuation exponent of $\chi$, yields a lower bound of 
$\exp(-c\theta^{1/(1-\chi)})$ for the probability that $X_r$ is at least $\mu r + \theta r^{\chi}$ under an analogous assumption to Assumption~\ref{a.lower bound upper tail}. It is an interesting question whether there is a matching upper bound and what is the exponent of the lower tail; as alluded to in Remark~\ref{r.tail exponent}, it is natural to guess that a matching upper bound holds, and that the lower tail exponent should be $2/(1-\chi)$. The latter is obtained by considering disjoint paths packed optimally. Note also that, by the KPZ scaling relation $\chi = 2\xi - 1$, the predicted lower tail exponent $2/(1-\chi)$ is the same as $1/(1-\xi)$ ($\xi$ being the wandering exponent, i.e., the exponent of transversal fluctuations, which, for example, should be $2/3$ in two dimensions).

Finally, we comment on the possibility of applying these techniques to first passage percolation, perhaps the most canonical non-integrable model expected to be in the KPZ class (and hence has weight and transversal fluctuation exponents of $1/3$ and $2/3$ in two dimensions). 
In principle many of our arguments should apply, as FPP enjoys a natural sub-additive structure analogous to the super-additive structure of LPP. But one technical difference that arises is that the paths in FPP are not directed and can backtrack, and this would require changes in the grid based discretizations employed in this paper for several of the main results. This will be pursued in future work.


\subsection{A few important tools}\label{s.tools}

In this section we collect some refined tools for last passage percolation which we will use for our arguments as outlined in Section~\ref{s.discussion of techniques}. 
There are four statements: the first asserts that it is typical for a path to suffer a weight loss which is quadratic in its transversal fluctuation, measured in the characteristic scalings of $r^{1/3}$ and $r^{2/3}$; the second is a related transversal fluctuation bound, but for paths with endpoint $(r-z,r+z)$ for $|z| \leq r^{5/6}$; the third is a high probability construction of a given number of disjoint paths which achieve a good collective weight; and the fourth provides bounds on the lower tail and mean of constrained weights.

%
We will import the proof ideas from \cite{watermelon}
 where similar statements have appeared. Our proofs are essentially the same but adapted suitably to work under the weaker tail exponent $\alpha$ assumed here; 
 for this reason, we only explain the modifications that need to be made for the first, second and fourth tools in Appendix~\ref{s.appendix}. The proof of the third tool is discussed in Section~\ref{s.melon construction overview} in slightly more detail.
%
%
%
%

\subsubsection{Parabolic weight loss for paths with large transversal fluctuation}

The following is the precise statement of the first tool.

\begin{theorem}[Refined transversal fluctuation loss]\label{t.tf general}
Let $t\leq s$ and let $Y_r^{s,t}$ be the maximum weight over all paths $\Gamma$ from the line segment joining $(-tr^{2/3}, tr^{2/3})$ and $(tr^{2/3}, -tr^{2/3})$ to the line segment joining 
$(r-tr^{2/3},r+tr^{2/3})$ and $(r+tr^{2/3},r-tr^{2/3})$
such that $\tf(\Gamma) > (s+t)r^{2/3}$.
Under Assumptions~\ref{a.limit shape assumption} and \ref{a.one point assumption upper}, there exist absolute positive constants
$r_0$, $s_0$, $c$ and $c_2$
such that, for $s>s_0$ and $r>r_0$,
$$\P\left(Y_r^{s,t} > \mu r -c_2s^2r^{1/3}\right) < \exp\left(-cs^{2\alpha}\right).$$
\end{theorem}

The proof of this follows that of \cite[Theorem~3.3]{watermelon}. We explain the necessary modifications in the appendix.

An important feature of Theorem~\ref{t.tf general} is that it bounds the probability of a \emph{decreasing} event, which is useful as it allows the application of the FKG inequality.

\subsubsection{Transversal fluctuation bound for $|z| \leq r^{5/6}$}
The second tool is a result on the transversal fluctuation of geodesics to $(r-z,r+z)$ (note that Theorem~\ref{t.tf general} is related but only for $z=0$), which is the following. We note in passing that the event of the geodesic having large transversal fluctuation is neither  increasing nor decreasing.


 \begin{proposition}[Transversal fluctuations]\label{p.tf}
For given $z$, let $\Gamma^z_r$ be a geodesic from $(1,1)$ to $(r-z,r+z)$ with maximum transversal fluctuation. Under Assumptions~\ref{a.limit shape assumption} and \ref{a.one point assumption}, there exist positive constants $c$, $s_0$, and $r_0$ such that, for $r>r_0$, $s>s_0$, and $|z|\leq r^{5/6}$,
$$\P\left(\tf(\Gamma^z_r) > sr^{2/3}\right) \leq \exp\left(-cs^{2\alpha}\right).$$
\end{proposition}

The proof of this is similar to that of \cite[Theorem 11.1]{slow-bond} and appears in the appendix.

\subsubsection{A high probability construction of disjoint paths with good collective weight}

Here is the statement of our third tool.

\begin{theorem}[Theorem~3.1 of \cite{watermelon}]\label{t.flexible construction}
Under Assumptions~\ref{a.limit shape assumption} and \ref{a.one point assumption}, there exist positive constants $c$, $C_1$, $k_0\in \N$ and $\eta$ such that for all $k_0\leq k \leq \eta r$ and $m\in\intint{1,k}$, with probability $1-e^{-ckm}$, there exist $m$ disjoint paths $\gamma_1, \ldots, \gamma_{m}$ in the square $\llbracket 1,r \rrbracket^2$, with $\gamma_i$ from $(1,m-i+1)$ to $(r,r-i+1)$ and $\max_i \tf(\gamma_i) \leq 2mk^{-2/3}r^{2/3}$, such that 
$$\sum_{i=1}^m\ell(\gamma_i) \geq \mu rm- C_1mk^{2/3}r^{1/3}.$$
\end{theorem}

The proof of Theorem~\ref{t.flexible construction} will be discussed in some detail in Section~\ref{s.melon construction overview} and will require as input our fourth tool on bounds for the lower tail and mean of constrained weights.

\subsubsection{Bounds for constrained weights}\label{s.constrained weights}
To state our fourth and final tool, recall from Section~\ref{s.constrained optimal tails} the notation for parallelograms $U_{r,\ell,z}$ of height $r$, width $\ell r^{2/3}$ and opposite midpoints $(1,1)$ and $(r-z,r+z)$ as well as that for maximum weight of paths constrained inside $U$, $Y_r^U$.

\begin{proposition}[Lower tail \& mean of constrained point-to-point, Proposition~3.7 of \cite{watermelon}]\label{p.constrained statements}
\mbox{}
Let positive constants $L_1$, $L_2$, and $K$ be fixed. Let $z$ and $\ell$ be such that $|z|\leq Kr^{2/3}$ and $L_1\leq \ell\leq L_2$, and let $U = U_{r,\ell, z}$. There exist positive constants $r_0 = r_0(K, L_1, L_2)$ and $\theta_0 = \theta_0(K, L_1, L_2)$, and an absolute positive constant $c$, such that, for $r>r_{0}$ and $\theta>\theta_{0}$, 
\begin{equation}\label{e.constrained lower tail}
\P\left(Y_{r}^{U} \leq \mu r-\theta r^{1 / 3}\right) \leq \exp\left(-c\ell^{2\alpha/3} \theta^{2\alpha/3}\right).
\end{equation}
As a consequence, there exists $C = C(K, L_1, L_2)$ such that, for $r>r_0$,
\begin{equation}\label{e.constrained mean}
\E[Y^U_{r}] \geq \mu r - Gz^2/r - Cr^{1/3}.
\end{equation}
\end{proposition}

To be consistent with previous expressions we have included the parabolic term $-Gz^2/r$ in the previous, but note that for the ranges of $z$ mentioned we can absorb it into the $Cr^{1/3}$ term.


\subsection{Organization of the article}

In Section~\ref{s.concentration and bootstrap} we collect the concentration statements for stretched exponential random variables and prove an abstracted version of the bootstrap. In Section~\ref{s.upper tail bounds} we prove Theorems~\ref{t.upper tail bootstrapping} and \ref{t.upper tail lower bound} which respectively concern upper and lower bounds on the upper tail. In Section~\ref{s.lower tail bounds} we address Theorems~\ref{t.optimal lower tail upper bound} and \ref{t.lower tail lower bound} on the corresponding bounds for the lower tail, as well as Theorem~\ref{t.constrained lower tail bounds} on the lower tails of the constrained geodesic weight. Finally in Appendix~\ref{s.appendix} we explain how the proofs of the first and fourth tools (Theorem~\ref{t.tf general} and Proposition~\ref{p.constrained statements}) of Section~\ref{s.tools} follow from the proofs of analogous results in \cite{watermelon} by replacing the use of tail bounds with exponent $3/2$ with the stretched exponential tails assumed here; provide the proofs of Lemmas~\ref{l.Z tail bound}, \ref{l.Z tail bound 2}, and \ref{l.lower tail lower bound} from the main text; and prove the second tool of Section~\ref{s.tools}, Proposition~\ref{p.tf}.


\subsection*{Acknowledgments} We thank the referees for their detailed comments. 
SG is partially supported by NSF grant DMS-1855688, NSF CAREER  Award DMS-1945172, and a Sloan Research Fellowship.  MH is supported by a summer grant and the Richman Fellowship of the UC Berkeley Mathematics department, and by NSF grant DMS-1855550.


\section{Concentration tools and the bootstrap}
\label{s.concentration and bootstrap}

In this section we collect the concentration inequality for stretched exponential random variables from \cite{stretched-exp-concentration} and prove a slightly more flexible version which is more suitable for our applications. We then move to stating a general version of one iteration of the bootstrap, which will both illustrate the basic mechanism and be used later in Section~\ref{s.lower tail bounds}.

To set the stage, let $\alpha\in (0,1]$ and suppose $Y_i$ are independent mean zero random variables which satisfy, for some $L, M< \infty,$
\begin{equation}\label{e.finite orlicz norm}
\inf\left\{\eta > 0 : \E\left[ g_{\alpha, L} \left(\frac{|Y_i|}{\eta}\right)\right] \leq 1 \right\} \leq M,
\end{equation}
where $g_{\alpha, L}(x) = \exp\left(\min\{x^2, (x/L)^\alpha\}\right) -1$.
The above condition is equivalent to the finiteness of a certain Orlicz norm introduced in \cite{stretched-exp-concentration}; see Definition~2.3 and Proposition~A.1 therein. The use of Orlicz norms to prove concentration inequalities is well known; see for example \cite{vershynin2018high,wainwright-concentration}.
The reader not familiar with this notion can keep in mind mean zero random variables $Y_i$ with the property that, for some $c>0$ and $C$, and all $t\geq 0$,
\begin{equation}\label{twosided}
\P(|Y_i|\ge t)\le C\exp(-ct^{\alpha}),
\end{equation}
which are known to satisfy \eqref{e.finite orlicz norm}.

\begin{proposition}\label{p.stretched exponential concentration} Given the above setting,
 there exists $c = c(M,L)>0$ such that for all $t \geq 0$ and all $k\in \N$,
$$\P\left(\left|\sum_{i=1}^k Y_i\right| \geq t \right) \leq 
\begin{cases}
\displaystyle 2\exp\left(-\frac{ct^2}{k}\right) & 0 \leq t \leq k^{1/(2-\alpha)}\\
2\exp\left(-ct^\alpha\right) & t \geq k^{1/(2-\alpha)}.
\end{cases}$$
\end{proposition}

These two regimes capture the transition from the Gaussian behavior in the immediate tail to stretched exponential behavior deep into the tail.

\begin{proof}[Proof of Proposition~\ref{p.stretched exponential concentration}]
\cite[Theorem 3.1]{stretched-exp-concentration} implies that, for some constants $C$ and $c>0$ (depending on $M$ and $L$), for all $t\ge 0,$
$$\P\left(\left|\sum_{i=1}^k Y_i\right| \geq C(\sqrt{kt} + t^{1/\alpha}) \right) \leq 2\exp(-ct).$$
Evaluating the transition point where $\sqrt{kt} = t^{1/\alpha}$ yields the statement of Proposition~\ref{p.stretched exponential concentration} by modifying the value of $c$ in the previous display.
\end{proof}

In our applications, we will only have an upper tail bound and hence not a direct verification of the hypothesis \eqref{e.finite orlicz norm} which needs two sided bounds as in \eqref{twosided}. It will also at times be convenient to center the variables not by their expectation but by some other constant for which a tail bound is available. These two aspects are handled in the next lemma.

\begin{lemma}\label{l.stretched exp conc with pseudomean}
Suppose $k\in \N$, $\{Y_i: i\in\intint{1,k}\}$ are independent, and there exist constants $\nu_i\in\R$, $\alpha \in (0,1]$, and $c>0$ such that, for $t > t_0$, and $i\in\intint{1,k}$,
\begin{equation}\label{hypoth10}
\P\left(Y_i-\nu_i \geq t \right) \leq \exp(-ct^\alpha).
\end{equation}
Then there exist positive constants $c_1 = c_1(c, \alpha, t_0)$ and $c' = c'(c,\alpha)$ such that, for $t \geq 0$ and all $k\in \N$,
$$\P\left(\sum_{i=1}^k (Y_i-\nu_i) > t + kc_1\right) \leq 
\begin{cases}
2\displaystyle\exp\left(-\frac{c't^2}{k}\right) & 0 \leq t \leq k^{1/(2-\alpha)}\\
2\exp\left(-c't^\alpha\right) & t \geq k^{1/(2-\alpha)}.
\end{cases} $$
\end{lemma}

\begin{proof}
Let $W_i$ be independent positive random variables whose distribution is defined by $\P(W_i > t) = \exp(-ct^\alpha)$ for $t\geq 0$. Then the hypothesis on $Y_i$ implies that $Y_i-\nu_i$ is stochastically dominated by $W_i + t_0,$ and hence there is a coupling of the $Y_i$ and $W_i$ over all $i$ simultaneously such that
$$Y_i-\nu_i \leq W_i + t_0,$$
by standard coupling arguments.
It is a calculation that $\E[W_i] = \alpha c^{-1/\alpha}\Gamma(\alpha)$, where $\Gamma(z) = \int_0^\infty x^{z-1} e^{-x}\, \mathrm dx$ is the gamma function. Thus we get
\begin{align*}
P\left(\sum_{i=1}^k (Y_i-\nu_i) > t + kc_1\right) \leq \P\left(\sum_{i=1}^k(W_i -\E[W_i]) > t + k(c_1-t_0 - \alpha c^{-1/\alpha}\Gamma(\alpha))\right).
\end{align*}
Setting $c_1 = t_0 + \alpha c^{-1/\alpha}\Gamma(\alpha)$ and applying Proposition~\ref{p.stretched exponential concentration} completes the proof of Lemma~\ref{l.stretched exp conc with pseudomean}, under the condition that $W_i-\E[W_i]$ satisfies \eqref{e.finite orlicz norm} for some $L$ and $M$ depending only on $\alpha$ and $c$. We verify this next. \cite[Proposition A.3]{stretched-exp-concentration} asserts that for any random variable $Y$ satisfying, for all $t\geq 0$,
\begin{equation}\label{e.orlicz norm tail condition}
\P(|Y|\geq t) \leq 2\exp(-\tilde ct^\alpha),
\end{equation}
there exist $M$ and $L$, depending on $\alpha$ and $\tilde c$, such that \eqref{e.finite orlicz norm} holds with $Y$ in place of $Y_i$. Therefore it is sufficient to verify \eqref{e.orlicz norm tail condition} for $Y = W_i - \E[W_i]$ for some $\tilde c$ depending on $\alpha$ and $c$. Since $W_i$ is positive for each $i$, we have the bound
$$\P\Big(|W_i - \E[W_i]| \geq t\Big) \leq \begin{cases}
1 & 0\leq t\leq \E[W_i]\\
\exp(-ct^\alpha) & t>\E[W_i]
\end{cases},$$
which implies that \eqref{e.orlicz norm tail condition} holds with $\tilde c = \min(c, \log 2)\cdot(\E[W_i])^{-\alpha}$, since $2\exp(-\tilde ct^\alpha) \geq 1$ for $0\leq t \leq \E[W_i]$. Note that $\tilde c$ depends on only $\alpha$ and $c$. This completes the proof of Lemma~\ref{l.stretched exp conc with pseudomean}.
\end{proof}

With the concentration tool Lemma~\ref{l.stretched exp conc with pseudomean} in hand, we next present the driving step of the bootstrapping argument. It is the formal statement and proof of one step of the iteration under a \emph{sub}-additive assumption. As indicated in the outline of proof section, since $X_r$ are \emph{super}-additive, this will not be of use for the upper bound on the upper tail; but it will find application in the upper bound on the lower tail, where super-additivity is the favourable direction.

\begin{proposition}\label{p.bootstrap}
Suppose that for each $r,k\in \N$ with $k\leq r$, $\{Y_{r,i}^{(k)}: i\in\intint{1,k}\}$ is a collection of independent random variables. Suppose also that there exist $\alpha\in(0,1]$ and positive constants $c$, $r_0$, and $\theta_0$ such that, for $r\in \N$, $k\in\N$, $i\in\intint{1,k}$, and $\theta\in \R$ such that $r/k>r_0$ and $\theta > \theta_0$,
\begin{equation}\label{hypoth12}
\P\left(Y_{r,i}^{(k)} > \theta (r/k)^{1/3}\right) \leq \exp\left(-c\theta^{\alpha}\right).
\end{equation}
Finally, let $Y_r$ be a random variable such that \smash{$Y_r\leq \sum_{i=1}^k Y_{r,i}^{(k)}$} for any $k\in \N$ satisfying $r/k>r_0$.
%
 Then there exist positive constants $\widetilde\theta_0 = \widetilde \theta_0(c,\alpha,\theta_0, r_0)$ and $c' = c'(c, \alpha, \theta_0, r_0)$ such that, for $r>r_0$ and $\widetilde\theta_0 < \theta < r^{2/3}$,
$$\P\left(Y_r  > \theta r^{1/3}\right) \leq \exp\left(-c' \theta^{3\alpha/2}\right).$$
\end{proposition}

Proposition~\ref{p.bootstrap} is written in a slightly more general way, without explicit reference to the LPP context it will be applied, to highlight the features of LPP that are relevant. In its application $Y_r$ will be the weight of the heaviest path constrained to be in a certain parallelogram of height $r$, centred by $\mu r$, and $\smash{Y_{r,i}^{(k)}}$ will be weights when constrained to be in disjoint subparallelograms of height $r/k$, centred by $\mu r/k$.

Finally, we mention a rounding convention we will adopt for the rest of the paper: the quantities $k$ and $r/k$ should always be integers and, when expressed as real numbers, will be rounded down without comment. The discrepancies of $\pm 1$ which so arise will be absorbed into universal constants.

\begin{proof}[Proof of Proposition~\ref{p.bootstrap}]

By the bound $Y_r\leq \sum_{i=1}^k Y_{r,i}^{(k)}$, for every $r,k\in\N$ with $k\leq r$,
\begin{equation}\label{e.bootstrap domination}
\P\left(Y_r > \theta r^{1/3}\right) \leq \P\left(\sum_{i=1}^k Y_{r,i}^{(k)}  > \theta r^{1/3}\right).
\end{equation}
We will choose $k = \eta \theta^{3/2}$ for some $\eta \in (0,1)$, a form which is guided by our desire to apply the concentration bound Lemma~\ref{l.stretched exp conc with pseudomean} with its input bound \eqref{hypoth10} provided by the hypothesis \eqref{hypoth12} of Proposition~\ref{p.bootstrap}; we also need $k\geq 1$. The first consideration will determine an acceptable value for $\eta$ via its development as the following two constraints:

\begin{enumerate}
	\item Lemma~\ref{l.stretched exp conc with pseudomean} introduces a linear term $kc_1$, which, when multiplied by the scale $(r/k)^{1/3}$ of the \smash{$Y_{r,i}^{(k)}$} indicated by \eqref{hypoth12}, is $c_1k^{2/3}r^{1/3}$; we want this to be smaller than a constant, say $\frac{1}{2}$, times $\theta r^{1/3}$. Note that $c_1$ depends on $\alpha$, $c$, and $\theta_0$.

	\item We require $r/k > r_0$ to apply the hypothesis of Proposition~\ref{p.bootstrap}. 
\end{enumerate}

These two constraints, and that $\theta < r^{2/3}$ by hypothesis, force $\eta$ to be smaller than $r_0^{-1}$ and \smash{$2^{-3/2}c_1^{-3/2}$}. We pick an $\eta$ which satisfies these inequalities; thus $\eta$ depends on $c_1$ and $r_0$. Set $\tilde\theta_0 = \eta^{-2/3}$; then $\theta\geq \tilde\theta_0$ implies $k\geq 1$. We will apply Lemma~\ref{l.stretched exp conc with pseudomean} with $Y_i = Y_{r,i}^{(k)}(r/k)^{-1/3}$, $\nu_i = 0$,
 and $t=\frac{1}{2}\theta k^{1/3}$. For $\theta\geq\tilde\theta_0$ and for a $\tilde c$ depending on only $c$ and $\alpha$,
\begin{align*}
\P\left(Y_r >  \theta r^{1/3}\right) &\leq \P\left(\sum_{i=1}^k Y_{r,i}^{(k)} > \frac{1}{2}\theta k^{1/3} \left(\frac{r}{k}\right)^{1/3} + kc_1 \left(\frac{r}{k}\right)^{1/3}\right)\\
&\leq \begin{cases}
2\exp\left(-\tilde c\theta^2k^{-1/3}\right) & \tilde\theta_0k^{1/3}\leq \theta k^{1/3} \leq k^{1/(2-\alpha)}\\
2\exp\left(-\tilde c\theta^\alpha k^{\alpha/3}\right) & \theta k^{1/3} \geq k^{1/(2-\alpha)}
\end{cases} \tag*{(applying Lemma~\ref{l.stretched exp conc with pseudomean})}\\
&\leq 2\exp\left(-\tilde c \eta^{\alpha/3}\theta^{3\alpha/2}\right) ;
\end{align*}
in the final line we have taken the second case of the preceding line. This is because $\alpha\leq 1$ implies $k^{1/(2-\alpha)} \leq k$, and the choice of $k$ (and that $\eta<1$) ensures that $\theta k^{1/3}\geq k$; so the second case holds.

The proof of Proposition~\ref{p.bootstrap} is complete by absorbing the factor of $2$ in the final display into the exponential, which we do by setting $c'$ to $\tilde c\eta^{\alpha/3}/2$ and increasing $\tilde \theta$ (if needed), depending on $c'$, so that $\exp(-c'(\tilde\theta_0)^{3\alpha/2}) \leq 1/2$.
\end{proof}


\section{Upper tail bounds}\label{s.upper tail bounds}

In this section we prove Theorems~\ref{t.upper tail bootstrapping} and \ref{t.upper tail lower bound}, respectively the upper and lower bounds on the upper tail.

\subsection{Upper bound on upper tail}
As mentioned in Section~\ref{s.discussion of techniques}, for the argument for the upper bound on the upper tail, we need a sub-additive relation, instead of the natural super-additive properties that point-to-point weights exhibit. To bypass this issue, we discretize the geodesic and bound the weights of the discretizations by interval-to-interval weights, which do have a sub-additive relation with the point-to-point weight; this allows us to appeal to a form of the basic bootstrapping argument outlined around \eqref{e.mean difference}; Then performing a union bound over all possible discretizations will complete the proof.

We next state a version of one iteration of the bootstrap for the upper bound on the upper tail. There are a number of parameters which we will provide more context for after the statement.

\begin{proposition}\label{p.upper tail bootstrap iteration}
Let $\lambda_j = \frac{1}{2}+\frac{1}{2^j}$. Suppose there exist $\alpha\in(0,1]$, $\beta\in [\alpha,1]$, $\zeta\in (0, \infty]$, $j\in\N$ and positive constants $c$, $\theta_0$, and $r_0 \geq 2$ such that, for $\theta>\theta_0$, $r>r_0$, and $|z|\leq r^{5/6}$,
\begin{equation}\label{e.upper tail bootstrap hypothesis}
\P\left(X^z_r \geq \mu r - \lambda_j\frac{Gz^2}{r} + \theta r^{1/3}\right) \leq 
\begin{cases}
\exp\left(-c\theta^\beta\right) & \theta_0<\theta< r^{\zeta}\\
\exp\left(-c\theta^\alpha\right) & \theta \geq r^{\zeta}.
\end{cases}
\end{equation}
Let $\zeta'= \min\big(\frac{\alpha\zeta}{1+\alpha\zeta}\cdot\frac{3-\beta}{3\beta},\,\, \frac{2\alpha}{9+16\alpha}\big)$, with $\frac{\alpha\zeta}{1+\alpha\zeta}$ interpreted as $1$ if $\zeta=\infty$.
There exist positive constants $c' = c'(c,\alpha,\beta,j)>0$, $\theta_0' = \theta_0'(\theta_0,c,\alpha,\beta,j)$, and $r_0' = r_0'(\alpha, j, r_0)$ such that, for $\theta > \theta_0'$, $r>r_0'$, and $|z|\leq r^{5/6}$,
$$\P\left(X^z_r \geq \mu r - \lambda_{j+1}\frac{Gz^2}{r}+ \theta r^{1/3}\right) \leq
\begin{cases}
\exp\left(-c'\theta^{\frac{3\beta}{3-\beta}}(\log \theta)^{-\frac{\beta}{3-\beta}}\right) & \theta_0 < \theta < r^{\zeta'}\\
\exp(-c'\theta^{\alpha}) & \theta\geq r^{\zeta'},
\end{cases}$$
%
In particular, the input \eqref{e.upper tail bootstrap hypothesis} with parameters $(\alpha,\beta,\zeta,j)$ gives as output  the same inequality with parameters $(\alpha,\beta',\zeta',j+1)$, where $\beta'>\beta$ may be taken to be $\frac{3-\beta/2}{3-\beta}\cdot\beta$ in order to  absorb the logarithmic factor.
%
\end{proposition}

(The condition $r_0\geq 2$ is arbitrary to ensure that $r^\infty = \infty$; in our applications $r_0$ will always be a large number and can be increased if needed.)

We first explain in words the content of the above result and describe the role of the various quantifiers appearing in the statement.

\paragraph{\emph{The range of $z$}} Though Theorem~\ref{t.upper tail bootstrapping} is stated only for $z=0$, the discretization of the geodesic we adopt demands that we have the bootstrap improve the tail bound in a number of directions, defined by $|z|\leq r^{5/6}$, in order to handle the potential  ``zig-zaggy'' nature of the geodesic. Here we choose to consider $|z|$ till $r^{5/6}$ as till this level the lowest order term $Hz^4/r^3$ in Assumption~\ref{a.limit shape assumption} is at most of the order of fluctuations, namely $r^{1/3}$.

\paragraph{\emph{The role of $\lambda_j$}}
One may expect to be able to obtain an improved tail for deviation from the expectation, which is $\mu r - Gz^2/r$ up to smaller order terms. However, for technical reasons, this proves to be difficult; we say a little more about this in the caption of Figure~\ref{f.from point to interval}. Instead, Proposition~\ref{p.upper tail bootstrap iteration} proves a bound for the deviation only from a point away from the expectation, reflected by the factor $\lambda_j$ in front of the parabolic term, which decreases as $j$ increases. Nonetheless, this weaker bound suffices for our application:
the relaxation has no effect for the $z=0$ direction asserted by Theorem~\ref{t.upper tail bootstrapping} since the parabolic term is always zero in that case. 

\paragraph{\emph{The role of $\zeta$}} Notice that in the hypothesis \eqref{e.upper tail bootstrap hypothesis} we allow two tail behaviors (with tail exponents $\alpha$ and $\beta$) for $X_r^z$ in different regimes, with boundary at $r^\zeta$. This is to allow the use of the conclusion of Proposition~\ref{p.upper tail bootstrap iteration}, which only improves the tail exponent for $\theta$ up to $r^{\zeta'}$, as input for subsequent applications of the same proposition. Theorem~\ref{t.upper tail bootstrapping} will be obtained by applying Proposition~\ref{p.upper tail bootstrap iteration} a finite number of times, with the output bound (with an increased exponent) of one application being the input for the next, till the exponent is raised from the initial value of $\beta=\alpha$ to a value greater than one for $\theta$ in the appropriate range of the tail. Then the same proposition will be applied one final time with $\beta =1$; at this value of $\beta$, $$\theta^{3\beta/(3-\beta)}(\log \theta)^{-\beta/(3-\beta)} = \theta^{3/2}(\log \theta)^{-1/2},$$ which will yield Theorem~\ref{t.upper tail bootstrapping}. The quantity  
$$\zeta'= \min\left(\frac{\alpha\zeta}{1+\alpha\zeta}\cdot \frac{3-\beta}{3\beta},\,\, \frac{2\alpha}{9+16\alpha}\right)$$
measures how far into the tail each improved exponent holds via our arguments.  The above explicit expression we obtain is perhaps hard to parse and is not of great importance for our conclusions. Nonetheless, we point out two basic properties of $\zeta'$: (i) it is smaller than $\zeta$, as may be seen by algebraic manipulations of the first of the two expressions being minimized in its definition (along with $\beta\geq \alpha$); and (ii) it decays to zero as $\alpha\to0$ linearly.

We next prove Theorem~\ref{t.upper tail bootstrapping} given Proposition~\ref{p.upper tail bootstrap iteration}, before turning to the proof of Proposition~\ref{p.upper tail bootstrap iteration}.

\begin{proof}[Proof of Theorem~\ref{t.upper tail bootstrapping}]
First, if $\alpha\geq 1$, we apply Proposition~\ref{p.upper tail bootstrap iteration} with $\alpha=\beta =1$, $\zeta = \infty$, $j=1$, and the hypothesis \eqref{e.upper tail bootstrap hypothesis} provided by Assumption~\ref{a.one point assumption upper}. This yields Theorem~\ref{t.upper tail bootstrapping} by taking $z=0$.

If $\alpha\in(0,1)$, we will apply Proposition~\ref{p.upper tail bootstrap iteration} iteratively finitely many times. Let $\alpha_j$, $\beta_j$, and $\zeta_j$ be values which we will specify shortly. We will select these values such that the hypothesis \eqref{e.upper tail bootstrap hypothesis} of Proposition~\ref{p.upper tail bootstrap iteration} holds with parameters $(\alpha_1,\beta_1,\zeta_1,1)$ for all $|z|\leq r^{5/6}$, and, knowing that \eqref{e.upper tail bootstrap hypothesis} holds with parameters $(\alpha_j,\beta_j,\zeta_j,j)$ for all $|z|\leq r^{5/6}$ and applying Proposition~\ref{p.upper tail bootstrap iteration} will imply that \eqref{e.upper tail bootstrap hypothesis} holds with parameters $(\alpha_{j+1}, \beta_{j+1},\zeta_{j+1}, j+1)$ for all $|z|\leq r^{5/6}$.

We set $\alpha_j = \alpha$ for all $j$, and adopt the initial settings $\beta_1 = \alpha$ and $\zeta_1 = \infty$; so again \eqref{e.upper tail bootstrap hypothesis} is provided by Assumption~\ref{a.one point assumption upper} when $j=1$. The subsequent values are read off of Proposition~\ref{p.upper tail bootstrap iteration} as follows for $j \geq 2$:
\begin{equation}\label{e.beta_j zeta_j values}
\beta_{j} = \min\left(\frac{3-\frac{1}{2}\beta_{j-1}}{3-\beta_{j-1}}\cdot \beta_{j-1},\,\, 1\right)\quad\text{and}\quad \zeta_{j} = \min\left(\frac{\alpha\zeta_{j-1}}{1+\alpha\zeta_{j-1}}\cdot \frac{3-\beta_{j-1}}{3\beta_{j-1}},\,\, \frac{2\alpha}{9+16\alpha}\right),
\end{equation}
where $\alpha \zeta_{j-1}/(1+\alpha \zeta_{j-1})$ in the definition of $\zeta_j$ is interpreted as $1$ when $\zeta_{j-1} = \infty$. We adopt the previous expression for $\beta_j$ instead of the one given by Proposition~\ref{p.upper tail bootstrap iteration} in order to absorb the $\log$ factor in the denominator of the exponent furnished by that proposition. Observe that $\beta_j > \beta_{j-1}$ whenever $\beta_{j-1}<1$.

We define $n\in\N$ by
\begin{equation}\label{e.upper tail n value}
n:=\min\Big\{j : \beta_j= 1\Big\};
\end{equation}
it can be checked that $n$ is finite since, if $\beta_{j} < 1$,
$$\frac{\beta_j}{\beta_{j-1}} = \frac{3-\frac{1}{2}\beta_{j-1}}{3-\beta_{j-1}} = 1 + \frac{\beta_{j-1}}{2(3-\beta_{j-1})} \geq 1 + \frac{\alpha}{2(3-\alpha)},$$
as $\beta_{j-1} > \beta_{j-2} > \ldots > \beta_1 = \alpha$.

By the previous discussion, we know that \eqref{e.upper tail bootstrap hypothesis} holds with parameters $(\alpha_{n}, \beta_{n}=1,\zeta_{n}, n)$. Applying Proposition~\ref{p.upper tail bootstrap iteration} with these parameters and taking $z=0$ gives the statement of Theorem~\ref{t.upper tail bootstrapping} with $\zeta = \zeta_{n+1} = \min(\frac{2}{3}\cdot \frac{\alpha\zeta_{n}}{1+\alpha\zeta_{n}}, \frac{2\alpha}{9+16\alpha})$. It is clear from this expression that $\zeta\to 0$ as $\alpha\to 0$, and, since $2\alpha/(9+16\alpha)$ achieves a maximum value of $2/25$ for all $\alpha\in(0,1]$, that $\zeta\in(0,2/25]$.
\end{proof}

\begin{remark}\label{r.precise other regimes}
We can now specify more precisely the regimes of $\theta$ provided by the proof of Theorem~\ref{t.upper tail bootstrapping} where the tail exponent transitions from $3/2$ to $\alpha$ as one goes further into the tail, as mentioned in Remark~\ref{r.other regimes}. That is, for $j=\intint{1,n}$ with $n$ as in \eqref{e.upper tail n value} and $\beta_j$ and $\zeta_j$ as in \eqref{e.beta_j zeta_j values}, it holds for $\theta\in [r^{\zeta_{j+1}}, r^{\zeta_j}]$ that
\begin{align*}
\P\left(X_r -\E[X_r] \geq \theta r^{1/3}\right)\leq \exp\left(-c\theta^{\beta_{j}}\right) \quad \text{for} \quad \theta\in [r^{\zeta_{j+1}}, r^{\zeta_{j}}).
\end{align*}
\end{remark}

It remains to prove Proposition~\ref{p.upper tail bootstrap iteration}. A roadmap for the proof is as follows. 

\begin{enumerate}
	\item As indicated immediately before the statement of the proposition, to achieve a stochastic domination of the geodesic weight by a sum, we specify a grid-based discretization of the geodesic, and Lemma~\ref{l.discretizations cardinality} bounds the cardinality of the number of possible discretizations.

	\item Lemma~\ref{l.P(X_L) main bound} provides an improved tail bound (i.e., a larger coefficient depending on $k$, which we will later set to be power of $t$, in the exponent, compared to the constant coefficient in hypothesis \eqref{e.upper tail bootstrap hypothesis}) for the weight of a given discretization, using the bootstrapping idea of looking at smaller scales. This makes use of Lemma~\ref{l.Z tail bound}, which takes the point-to-point tail available from \eqref{e.upper tail bootstrap hypothesis} and gives an interval-to-interval bound with the same tail.

	\item  When Lemmas~\ref{l.discretizations cardinality} and \ref{l.P(X_L) main bound} are in hand, the proof of Proposition~\ref{p.upper tail bootstrap iteration} will be completed by taking a union bound.
\end{enumerate}

We address each of the above three steps in turn in the next three subsections.

\subsubsection{Step 1: The discretization scheme}\label{s.upper tail grid specification}
We will define a grid $\mathbb G^z$ of intervals through which any geodesic from $(1,1)$ to $(r-z,r+z)$, on the event that it is typical, must necessarily pass through; see Figure~\ref{f.grid}.

We recall from Section~\ref{s.discussion of techniques} that ``width" refers to measurement along the anti-diagonal and ``height'' to measurement along the diagonal.
For $k\in \N$ to be set, the width of a cell in the grid will be $(r/k)^{2/3}$, and the height $r/k$. The number of cells in a column of the grid is $k$, and the number of cells in a row  $M$ will be, up to rounding, $2 \theta^{3/(4\alpha)}k^{2/3}$ as we want the width of $\mathbb G^z$ to be $2\theta^{3/(4\alpha)}r^{2/3}$. The width of $\mathbb G^z$ is set to this value because, by Proposition~\ref{p.tf} on the probability of any geodesic having large transversal fluctuations, $\P(\tf(\Gamma^z_r) > \theta^{3/(4\alpha)}r^{2/3}) \leq \exp(-c\theta^{3/2})$; note that this is smaller than the bound we are aiming to prove in Proposition~\ref{p.upper tail bootstrap iteration} and so we may essentially ignore the event that any geodesic exits the grid.

\begin{figure}
\centering
\includegraphics[width=0.5\textwidth]{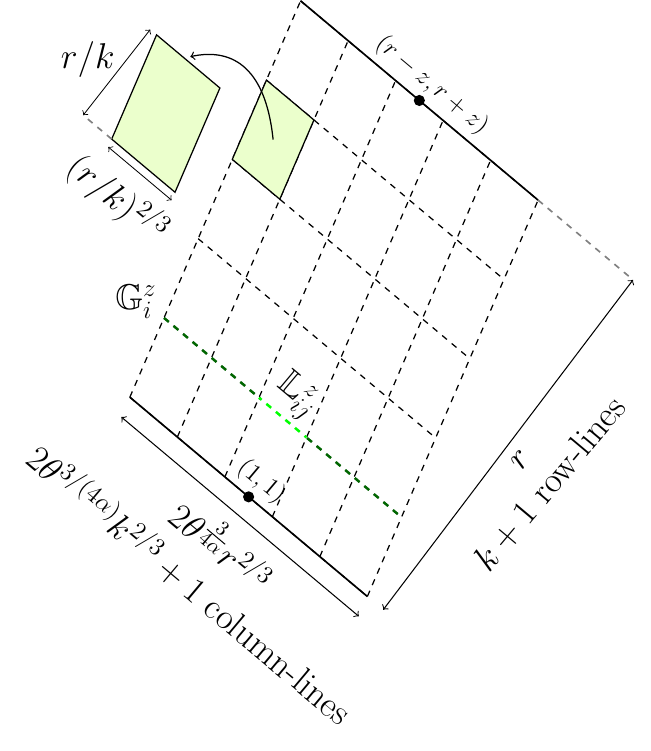}
\caption{The grid utilized for the discretization in Step 1 of the proof of Proposition~\ref{p.upper tail bootstrap iteration}. Note that measurements are made along the antidiagonal and diagonal only, with the diagonal chosen over the line with the slope of the left or right boundary of the grid. The lower boundary of the grid $\mathbb G^z$ is centered at $(1,1)$ and the upper boundary at $(r-z,r+z)$. From each grid line $\mathbb G_i^z$, one interval $L_i$ is picked to form a discretization $\mc L^z = (L_0, \ldots, L_k)$ with the constraint that $L_0$ is fixed to be the interval on $\mathbb G_0^z$ whose midpoint is $(1,1)$ and $L_k$ to be the interval on $\mathbb G_k^z$ whose midpoint is $(r-z,r+z)$. On the high probability event that all geodesics passes through the grid, its weight is upper bounded by the maximum, over all discretizations $\mc L^z$, of the sum of interval-to-interval weights of the intervals in $\mc L^z$. These weights are independent and have fluctuations of scale $(r/k)^{1/3}$, which allows us to use the idea of bootstrapping.}
\label{f.grid}
\end{figure}

We now move to the formal definition. We assume $k$ is small enough that $(r/k)^{2/3}\geq 1$, i.e., $k \leq r$ (as the minimum separation of points in $\Z^2$ is $1$). The grid $\mathbb G^z$ consists of intervals $\L_{ij}^z$ as follows:
$$\mathbb G^z=\big\{\L_{ij}^z: i\in\intint{0,k}, j\in\intint{0,M}\big\},$$
where $M$ is precisely defined as
\begin{equation}
M = 2\cdot\lceil \theta^{\frac{3}{4\alpha}}k^{2/3} \rceil -1. \label{e.N and M value}
\end{equation}
Let $v_i = \lfloor ir/k\rfloor$ and $h_{i,j}^z = \lfloor iz/k+(\theta^{\frac{3}{4\alpha}}-jk^{-2/3})r^{2/3}\rfloor$. For $i\in\intint{0,k}$ and $j\in\intint{0,M}$, the line segment $\L_{ij}^z$ will connect the points
$$\left(v_i-h_{i,j}^z, v_i+h_{i,j}^z\right) \quad \text{and}\quad \left(v_i -h_{i,j+1}^z, v_i+h_{i,j+1}^z\right).$$
In words, the grid $\mathbb G^z$ is contained in the parallelogram $\{|y-\frac{r+z}{r-z}\cdot x| \leq \theta^{\frac{3}{4\alpha}}r^{2/3}, 0\leq x+y\leq 2r\}$. Grid lines along the anti-diagonal will be called $\mathbb G_i^z$, i.e., for $i\in\intint{0,k}$,
$$\mathbb G_i^z = \big\{\L_{ij}^z : j\in\intint{0,M}\big\}.$$
%
%
%
We call $\mc L^z= (L_0, \ldots, L_k)$ a discretization, where $L_i \in \mathbb G_i^z$ is an interval on the $i^{\text{th}}$ grid line. We impose that $L_0$ and $L_k$ are the intervals whose midpoints are $(1,1)$ and $(r-z, r+z)$ respectively.

\begin{lemma}\label{l.discretizations cardinality}
The set of discretizations has size at most $\exp\left\{k\left(\log k + \frac{3}{4\alpha}\log \theta + \log 2\right)\right\}$.
\end{lemma}

\begin{proof}
This follows from the observation that there are $M = 2\theta^{\frac{3}{4\alpha}}k^{2/3} \leq 2\theta^{\frac{3}{4\alpha}}k$ intervals on each grid line $\mathbb G_i^z$, and there are $k-1$ grid lines in total where there is a choice of interval (as the intervals from $\mathbb G_0^z$ and $\mathbb G_{k}^z$ are fixed), giving $\smash{(2\theta^{\frac{3}{4\alpha}}k)^{k-1}}$ discretizations.
\end{proof}

For a given discretization $\mc L^z = (L_0, \ldots, L_k)$, let $X_{\mc L^z}$ be the maximum weight of all paths which pass through all intervals of $\mc L^z$. The discretization described above implies that, on the event that $\tf(\Gamma^z_r) \leq \theta^{\frac{3}{4\alpha}}r^{2/3}$, 
$$X_r^z \leq \max_{\mc L^z} X_{\mc L^z},$$
where the maximization is over all discretizations $\mc L^z$. So to prove Proposition~\ref{p.upper tail bootstrap iteration}, we need a tail bound on $X_{\mc L^z}$ for a fixed discretization $\mc L^z$; this is Step 2 and is done in the next subsection, where the hypothesis \eqref{e.upper tail bootstrap hypothesis} and bootstrapping are used to provide an improved tail bound on $X_{\mc L^z}$.

\subsubsection{Step 2: An improved tail bound on $X_{\mc L^z}$}

 Because $\theta$ is a global parameter which affects the set of discretizations, we will use the symbol $t$ as in \eqref{e.fixed discretization tail bound} ahead to denote the scaled deviation when considering the weight associated to a fixed discretization, though we will eventually set $t=\theta$. The following lemma uses the idea of moving to lower scales to obtain an improved tail bound for $X_{\mc L^z}$ for a fixed discretization $\mc L^z$.

\begin{lemma}\label{l.P(X_L) main bound}
Under the hypotheses of Proposition~\ref{p.upper tail bootstrap iteration} there exist positive constants $c' = c'(c,\alpha,\beta,j)>0$, $\delta = \delta(c,\beta,j,\theta_0)$, and $t_0 = t_0(c,\beta,j)$ such that the following holds. Let $t>t_0$, $r>r_0$, $2^6 \leq k \leq \min(\delta t^{3/2}, r_0^{-1}r)$, $\theta \geq \theta_0 $, and $z \in [-r,r]$ be such that $|z|\leq r^{5/6}$ and $(r/k)^{5/6} > 4\theta^{3/(4\alpha)}r^{2/3}$. Let $\mc L^z = (L_0, \ldots,  L_k)$ be a fixed discretization. Then
\begin{equation}\label{e.fixed discretization tail bound}
\P\left(X_{\mc L^z} > \mu r - \lambda_{j+1}\frac{Gz^2}{r} + t r^{1/3}\right) \leq \exp\left(-c't^\beta k^{\beta/3}\right) + k\cdot \exp\left(-c'(r/k)^{\alpha\zeta}\right),
\end{equation}
with the second term interpreted as zero if $\zeta=\infty$.
\end{lemma}

The basic tool in the proof of Lemma~\ref{l.P(X_L) main bound} is to bound $X_{\mc L^z}$ by the sum of the interval-to-interval weights defined by the intervals in $\mc L^z$. So given a point-to-point upper tail bound, as in the hypothesis of Proposition~\ref{p.upper tail bootstrap iteration}, we will first need to obtain an upper tail bound for interval-to-interval weights. 

We define the relevant intervals to state the interval-to-interval bound next.
For $r$ fixed, and $|w|\leq r^{5/6}$, let $\linelower$ be the line segment joining $(-r^{2/3}, r^{2/3})$ and $(r^{2/3}, -r^{2/3})$ and let $\lineupper$ be the line segment joining $(r-w-r^{2/3}, r+w+r^{2/3})$ and $(r-w+r^{2/3}, r+w- r^{2/3})$. Thus $w$ is the midpoint displacement of the intervals, and note that their height difference is $r$.  Define $Z$ by
$$Z= X_{\linelower, \lineupper}.$$
The content of the next lemma is a tail bound on $Z$.

\begin{lemma}\label{l.Z tail bound}
Suppose \eqref{e.upper tail bootstrap hypothesis} holds as in Proposition~\ref{p.upper tail bootstrap iteration}. Then there exist positive $\tilde c = \tilde c(c,j)$, $\tilde t_0 = \tilde t_0(\theta_0, j)$, and $\tilde r_0 = \tilde r_0(r_0, j)$ such that, for $r>\tilde r_0$, $|w|\leq r^{5/6}$, and $t>\tilde t_0$
\begin{equation}\label{e.Z tail bound}
\P\left(Z > \mu r -\lambda_{j+1}\frac{Gw^2}{r}+ t r^{1/3}\right) \leq \begin{cases}
\exp\left(-\tilde ct^\beta\right) & \tilde t_0< t < r^{\zeta}\\
\exp\left(-\tilde ct^\alpha\right) & t \geq r^{\zeta}.
\end{cases}
\end{equation}
\end{lemma}

We note that the hypothesis \eqref{e.upper tail bootstrap hypothesis} of Proposition~\ref{p.upper tail bootstrap iteration} is a point-to-point tail bound from $\mu r - \lambda_j Gz^2/r$, whereas the conclusion of Lemma~\ref{l.Z tail bound} has the weaker $\lambda_{j+1}$ in place of $\lambda_j$ (recall $\lambda_j = 1/2+2^{-j}$). This reduction in the coefficient of the parabolic term is the previously mentioned relaxation which allows the bootstrap to proceed to the next iteration. 

The proof of Lemma~\ref{l.Z tail bound} relies on the geometric idea of stepping back from the two intervals and considering a proxy point-to-point weight. Similar arguments have appeared in the literature previously (see e.g., \cite{slow-bond}), but for completeness we give a self-contained proof of Lemma~\ref{l.Z tail bound} in Appendix~\ref{s.appendix}. However, we highlight the main idea  in Figure~\ref{f.from point to interval}, where we also say a few words on why it is difficult to avoid the relaxation in the parabolic loss.

\begin{figure}
\centering
\includegraphics[width=0.37\textwidth]{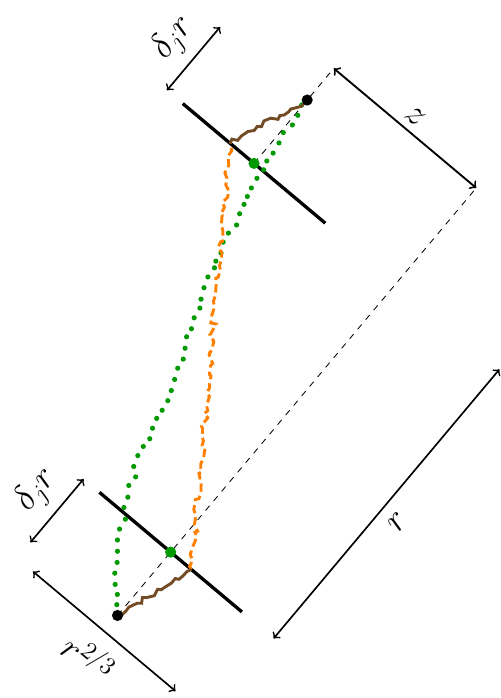}
\caption{The argument for Lemma~\ref{l.Z tail bound}. The two black intervals have midpoint separation of $z$ in the antidiagonal direction. The orange path (dashed) is the heaviest path between the two intervals (so has weight $Z$), and the brown paths (solid) are geodesics connecting the black points to the endpoints of the red path. The green path (dotted) is a geodesic between the two black points. With positive probability the two brown paths each have weight greater than $\mu \delta_j r - \frac{1}{3}\theta r^{1/3}$, and so, on the intersection of those events with $\{Z > \mu r - \lambda_{j+1}Gz^2/r + \theta r^{1/3}\}$, it holds that the green path has weight at least $\mu(1+2\delta_j)r - \lambda_{j+1}Gz^2/r + \frac{1}{3}\theta r^{1/3}$. We choose $\delta_j$ such that the parabolic term in this expression is $\lambda_j Gz^2/(1+2\delta_j)r$ and apply the point-to-point bound we have. It is because the antidiagonal separation between each pair of black and green points is zero that we have a decrease in the parabolic term. If we instead make this separation proportional to $z$, then there is no decrease in the parabolic term, but for large $z$ the gradient of the limit shape from Assumption~\ref{a.limit shape assumption} causes issues. This can be more carefully handled if we instead consider the supremum of \emph{fluctuations} of point-to-point weights from their expectation, and we will have need to do this on one occasion in the appendix. We also note an inaccuracy in the figure which we have retained to not distract from the main point: in truth, the brown and green paths will have some amount of overlap around their starting and ending points, as a general phenomenon of geodesic coalescence.}
\label{f.from point to interval}
\end{figure}

\begin{proof}[Proof of Lemma~\ref{l.P(X_L) main bound}]
Observe the following stochastic domination $$X_{\mc L^z}\preceq \sum_{i=1}^{k} Z_i,$$ where $Z_i$ are independent random variables distributed as the weight of the best path from $ L_{i-1}$ to $L_i$. Apart from possible rounding, because $Z_i$ and $Z_{i-1}$ are independent versions of weights which overlap on the interval $L_{i-1}$, it is possible that the linear term in $Z_i$ is $\mu r/k +O(1)$ rather than $\mu r/k$. We handle this discrepancy by absorbing it into the term $tr^{1/3}$ of Lemma~\ref{l.Z tail bound}, which is the only situation where it arises, without further comment.



We note that the diagonal separation between the sides of $Z_i$ is $r/k$, instead of $r$ as in the definition of $Z$. We denote the anti-diagonal displacement of the midpoints of the corresponding intervals of $Z_i$ by $z_i$. We want to eventually apply Lemma~\ref{l.stretched exp conc with pseudomean} to $\sum Z_i$, appropriately centred, with its input tail bound \eqref{hypoth10} provided by Lemma~\ref{l.Z tail bound}. To reach a form of the probability where Lemma~\ref{l.stretched exp conc with pseudomean} is applicable, we observe that
%
%
\begin{align}
\P\left(X_{\mc L^z} > \mu r -\lambda_{j+1}\frac{Gz^2}{r} + t r^{1/3}\right)
&\leq \P\left(\sum_{i=1}^k Z_i \geq \mu r - \lambda_{j+1}\frac{Gz^2}{r} + t r^{1/3}\right)\nonumber\\
&= \P\left(\sum_{i=1}^k (Z_i - \nu_i) \geq \mu r - \lambda_{j+1}\frac{Gz^2}{r} - \sum_{i=1}^k \nu_i + t r^{1/3}\right)\nonumber\\
&\leq \P\left(\sum_{i=1}^k (Z_i-\nu_i) \geq tr^{1/3}\right),
\label{e.interval to interval bound}
\end{align}
where $\nu_i = \mu r/k - \lambda_{j+1}Gz_i^2k/r$.
The choice of $\nu_i$ is dictated by the desire to apply \eqref{e.Z tail bound} with $r$ replaced by $r/k.$  All the steps before the last inequality are straightforward consequences of definitions. To see the last inequality, note that since $\sum z_i = z$, the Cauchy-Schwarz inequality implies that $\sum \nu_i$ is smaller than $\mu r - \lambda_{j+1}Gz^2/r$. 



We will soon apply Lemma~\ref{l.Z tail bound}, which will yield a tail bound for $Z_i-\nu_i$; the tail bound is \eqref{e.Z tail bound} with $r$ replaced by $r/k$. However, this tail bound has two regimes with different exponents, $\alpha$ and $\beta$, while the basic concentration result we seek to apply, i.e., Lemma~\ref{l.stretched exp conc with pseudomean}, assumes the same tail exponent throughout.

Thus to have variables that have the larger exponent $\beta$ in the entire tail, we will apply a simple truncation on $Z_i$: define
$$\overline Z_i = \begin{cases}
Z_i & \text{if } Z_i-\nu_i \leq \left(\frac{r}{k}\right)^{\zeta+1/3}\\
\nu_i & \text{if } Z_i-\nu_i > \left(\frac{r}{k}\right)^{\zeta+1/3}.
\end{cases}$$
Now following \eqref{e.interval to interval bound}, we get
\begin{equation}\label{e.P(X_L) bound 1}
\begin{split}
 \P\left(X_{\mc L^z} > \mu r - \lambda_{j+1}\frac{Gz^2}{r} + t r^{1/3}\right)
 &\leq \P\left(\sum_{i=1}^k (\overline Z_i - \nu_i) \geq t r^{1/3}\right)\\
 &\qquad\qquad + \P\bigg(\bigcup_{i=1}^k\left\{Z_{i} - \nu_i > (r/k)^{\zeta+1/3}\right\}\bigg).
 \end{split}
\end{equation}
We will apply the concentration bound Lemma~\ref{l.stretched exp conc with pseudomean} to bound the first term. We first want to apply Lemma~\ref{l.Z tail bound}, with $r/k$ in place of $r$, in order to get the tail bound on each individual $Z_i-\nu_i$, which will act as input for Lemma~\ref{l.stretched exp conc with pseudomean}. Two hypotheses of Lemma~\ref{l.Z tail bound}, namely \eqref{e.upper tail bootstrap hypothesis} and that $r/k>r_0$, are available here by the hypotheses of Lemma~\ref{l.P(X_L) main bound}. 

But Lemma~\ref{l.Z tail bound} has the additional hypothesis that the anti-diagonal displacement $|w|$ is at most $(r/k)^{5/6}$, which must also be checked. The verification of this follows from the hypothesis in Lemma~\ref{l.P(X_L) main bound} that $(r/k)^{5/6} > 4\theta^{3/(4\alpha)}r^{2/3}$, as the maximum anti-diagonal displacement possible in a single row of the grid is at most $2\theta^{3/(4\alpha)}r^{2/3}+|z|/k$, where the first term is the grid width $2\theta^{3/(4\alpha)}r^{2/3},$ and the second term is the shift caused by the overall slope of the grid. Now since $|z|\le r^{5/6}$ and $k\geq 2^6$, we see that $|z|/k$ is at most $\frac{1}{2}(r/k)^{5/6}$, and some simple algebra completes the verification. 

Thus, applying Lemma~\ref{l.Z tail bound} with $r/k$ in place of $r$, we use the first case of \eqref{e.Z tail bound} (since $Z_i$ has been appropriately truncated to give $\overline Z_i$) as the input tail bound with exponent $\beta$ on $\overline Z_i-\nu_i$ required for Lemma~\ref{l.stretched exp conc with pseudomean}. Finally, with $c_1 = c_1(c,\beta,\theta_0, j)$ as in the statement of the latter, let $\delta = \min(1, (2c_1)^{-3/2})$ where recall we have the hypothesis that $k\leq \min(\delta t^{3/2}, r_0^{-1}r)$; $\delta$ depends on $c,\beta, j$ and $\theta_0$. With this preparation, we see
\begin{align*}
\P\left(\sum_{i=1}^k (\overline Z_i - \nu_i) \geq t r^{1/3}\right)%
&= \P\left(\sum_{i=1}^k (\overline Z_i - \nu_i)(r/k)^{-1/3} \geq (t k^{1/3} - kc_1) + kc_1\right)\\
&\leq \P\left(\sum_{i=1}^k (\overline Z_i - \nu_i)(r/k)^{-1/3} \geq \frac{1}{2}t k^{1/3} + kc_1\right) \tag*{[since by hypothesis $k\leq (2c_1)^{-3/2}t^{3/2}$]}\\
&\leq
\begin{cases}
2\exp\left(-\tilde ct^2k^{-1/3}\right) & 0\leq t k^{1/3} < k^{1/(2-\beta)}\\
2\exp\left(-\tilde ct^\beta k^{\beta/3}\right) & t k^{1/3} \geq k^{1/(2-\beta)};
\end{cases} \tag*{[by Lemma~\ref{l.stretched exp conc with pseudomean}]}
\end{align*}
we have applied Lemma~\ref{l.stretched exp conc with pseudomean} with $t k^{1/3}$ in place of $t$ and $\alpha = \beta$. Here $\tilde c$ is a function of $c$ (as given in the hypothesis \eqref{e.upper tail bootstrap hypothesis}), $\beta$, and $j$.
We now claim that the second case of the last display dictates the fluctuation behavior under our hypotheses. To see this,  note that since $\beta \leq 1$, $k^{1/(2-\beta)} \leq k$. Thus the first case in the last display holds only if $k> t^{3/2}$ while by hypothesis $k\leq t^{3/2}$ since $\delta\leq 1$. Further, since $k\geq 1$, we may set the lower bound $t_0$ on $t$ high enough that $\exp(-\tfrac{1}{2}\tilde c t^{\beta}) \leq 1/2$ so as to absorb the pre-factor of 2 in the last display; $t_0$ depends on $\tilde c$ and $\beta$. We have hence bounded the first term of \eqref{e.P(X_L) bound 1}. 

To bound the second term when $\zeta<\infty$, we take a union bound and apply Lemma~\ref{l.Z tail bound}, where the latter's hypotheses are satisfied by the same reasoning as used above in the application for the first term. This yields that the second term of \eqref{e.P(X_L) bound 1} is bounded by $k\cdot \exp\left(-\tilde c(r/k)^{\alpha\zeta}\right)$, using the second case of \eqref{e.Z tail bound} with $r/k$ in place of $r$. Here $\tilde c$ is a function of $c$, $\alpha$, and $j$. When $\zeta=\infty$, the second term of \eqref{e.P(X_L) bound 1} is clearly zero.

Returning to \eqref{e.P(X_L) bound 1} with these two bounds completes the proof of Lemma~\ref{l.P(X_L) main bound}, taking $c'=\tilde c$.
\end{proof}

\subsubsection{Step 3: Handling all the discretizations}

With the improved tail bound for a fixed discretization provided by Lemma~\ref{l.P(X_L) main bound}, we can implement Step 3 and complete the proof of Proposition~\ref{p.upper tail bootstrap iteration}, essentially via a union bound.

\begin{proof}[Proof of Proposition~\ref{p.upper tail bootstrap iteration}]
Recall that $\theta_0'$ is the lower bound on $\theta$ under which the conclusions of Proposition~\ref{p.upper tail bootstrap iteration} must be shown to hold, and that we have the freedom to set it. We will increase its value as needed as the proof proceeds. We will be explicit about the dependencies $\theta_0'$ takes on at each such time. We start with $\theta_0' = e$ so that $\log \theta \geq 1$. Also, in this proof, $c$ is reserved for the constant in the point-to-point tail hypothesis \eqref{e.upper tail bootstrap hypothesis}.

Lemma~\ref{l.discretizations cardinality} says that the entropy from the union bound we will soon perform will be $\exp\{\Theta(k\log k+k\log \theta)\}$, which needs to be counteracted by the bound from Lemma~\ref{l.P(X_L) main bound}. Anticipating this we take, in Lemma~\ref{l.P(X_L) main bound},
\begin{equation} \label{e.upper tail upper bound k value}
t = \theta \qquad\text{and}\qquad k = \varepsilon\cdot\theta^{\frac{3\beta}{3-\beta}}(\log \theta)^{-\frac{3}{3-\beta}},
\end{equation}
for $\varepsilon = \varepsilon(c,\alpha,\beta,j)\in (0,1)$ a sufficiently small constant, to be set shortly. 
At this point we will ensure that the hypotheses of Lemma~\ref{l.P(X_L) main bound} hold. We set $\theta_0'$ larger if needed so that it is at least $t_0$ as in Lemma~\ref{l.P(X_L) main bound}, so that the value of $t$ above satisfies $t>t_0$. Additionally we have to verify that, with $\delta$ as provided by Lemma~\ref{l.P(X_L) main bound},

\begin{itemize}
\item
$4\theta^{3/(4\alpha)}r^{2/3}<(r/k)^{5/6}$.
\item $k\in\intint{2^6, \min(\delta t^{3/2}, r_0^{-1}r)}$
\end{itemize}

For the first condition, the fact that $k\leq \theta^{3/2}$ (since $\varepsilon,\beta \leq 1$ and $\log\theta\geq 1$), and some algebraic manipulation, implies that it is sufficient if $\theta \leq \frac{1}{4}r^{(2\alpha)/(9+15\alpha)}$; to avoid carrying forward the factor of $4$, we instead reduce the exponent of $r$ to absorb it and impose that
\begin{equation}\label{e.theta bound first}
\theta \leq r^{\frac{2\alpha}{9+16\alpha}};
\end{equation}
this implies $\theta \leq \frac{1}{4}r^{(2\alpha)/(9+15\alpha)}$ (and hence the first condition above) when $r_0'$, which is the lower bound on $r$ that we are free to set, is large enough. The value of $r_0'$ depends only on $\alpha$.

For the second condition, note that $2\alpha/(9+16\alpha) < 2/3$, and that $\beta\leq 1$ implies $3\beta/(3-\beta) \leq 3/2$.  Combining this latter inequality with the value \eqref{e.upper tail upper bound k value} of $k$, and that $\theta\leq r^{2/3}$ from \eqref{e.theta bound first}, ensures that $k\in\intint{2^6, \min(\delta\theta^{3/2}, r_0^{-1}r)}$ by setting $\theta_0'$ large enough, depending on $\beta$, $\delta$, and $\varepsilon$; so the second condition holds. 

 Thus applying Lemma~\ref{l.P(X_L) main bound} with values of $t$ and $k$ as in \eqref{e.upper tail upper bound k value} we obtain that, for $\theta_0'<\theta < r^{2\alpha/(9+16\alpha)}$,
\begin{equation}\label{e.X_L^z bound}
\begin{split}
\MoveEqLeft[10]
\P\left(X_{\mc L^z} > \mu r -\lambda_{j+1}\frac{Gz^2}{r} + \theta r^{1/3}\right)\\
%
%
%
&\leq \exp\left(-c'\cdot\varepsilon^{\beta/3}\theta^{\frac{3\beta}{3-\beta}}(\log \theta)^{-\frac{\beta}{3-\beta}}\right)
+ \theta^{\frac{3\beta}{3-\beta}}\exp\left(-c'r^{\alpha\zeta}\theta^{-\frac{3\alpha\beta\zeta}{3-\beta}}\right),
\end{split}
\end{equation}
with the second term equal to zero if $\zeta=\infty$, and with $c'$ as in Lemma~\ref{l.P(X_L) main bound}; thus $c'$ depends on $c,\alpha,\beta$, and $j$. In substituting $k$ in the second term of \eqref{e.fixed discretization tail bound} we have used that $k\leq \theta^{3\beta/(3-\beta)}$ since $\varepsilon<1$ and $\log \theta \geq 1$. When $\zeta<\infty$, we would like the exponential factor of the second term of \eqref{e.X_L^z bound} to be smaller than the first term; i.e., it is sufficient if
$$r^{\alpha\zeta}\theta^{-\frac{3\alpha\beta\zeta}{3-\beta}} \geq \theta^{\frac{3\beta}{3-\beta}}.$$
(We will soon absorb the polynomial-in-$\theta$ factor in the second term of \eqref{e.X_L^z bound} by reducing the constant $c'$.)
Simple algebraic manipulations show that the inequality of the last display is implied by the condition
\begin{equation}\label{e.zeta bar definition}
\theta \leq r^{\bar \zeta} \quad \text{with}\quad \bar\zeta = \frac{\alpha\zeta}{1+\alpha\zeta}\cdot \frac{3-\beta}{3\beta}.
\end{equation}
Strictly speaking we need this condition on $\theta$ only when $\zeta<\infty$, since, when $\zeta=\infty$, the second term of \eqref{e.X_L^z bound} is zero and so certainly smaller than the first term (still under the overall condition that $\theta \leq r^{2\alpha/(9+16\alpha)}$). But for simplicity, we also impose the  condition \eqref{e.zeta bar definition} on $\theta$ when $\zeta=\infty$; and in this $\zeta=\infty$ case, we interpret the first factor in the definition of $\bar\zeta$ to be one, i.e., $\bar\zeta = (3-\beta)/(3\beta)$.

To handle both the condition in the last display and \eqref{e.theta bound first}, we impose $\theta_0' < \theta < r^{\zeta'}$, with
$$\zeta' = \min\left(\bar\zeta, \frac{2\alpha}{9+16\alpha}\right).$$
So far we have shown that, for $r>r_0'$ and $\theta_0'<\theta< r^{\zeta'}$,
\begin{equation}\label{e.P(X_L) bound final}
\P\left(X_{\mc L^z} > \mu r -\lambda_{j+1}\frac{Gz^2}{r} + \theta r^{1/3}\right) \leq 2\exp\left(-\tfrac{1}{2}c'\cdot\varepsilon^{\beta/3}\theta^{\frac{3\beta}{3-\beta}}(\log \theta)^{-\frac{\beta}{3-\beta}}\right);
\end{equation}
where, for all $\theta>\theta_0'$, the $\theta^{\frac{3\beta}{3-\beta}}$ polynomial factor coming from the second term of \eqref{e.X_L^z bound} has been absorbed by the reduction of $c'$ to $c'/2$. To do this we may also need to increase the value of $\theta_0'$; this choice of $\theta_0'$ can be made depending only on $c'$ since we only need $\theta^{3\beta/(3-\beta)}\exp(-c'\theta^{3\beta/(3-\beta)}) \leq \exp(-0.5c'\theta^{3\beta/(3-\beta)})$ and the same function of $\theta$ is in the exponent and as the polynomial-factor.
%

Now we observe that on the event that any geodesic stays within the grid $\mathbb G^z$, $X_r^z$ is dominated by $\max_{\mc L^z} X_{\mc L^z}$. This yields
\begin{equation}\label{e.breakup into max over discr and TF}
\begin{split}
\MoveEqLeft[6]
\P\left(X^z_r > \mu r -\lambda_{j+1}\frac{Gz^2}{r}+ \theta r^{1/3}\right)\\
&\leq \P\left(\max_{\mc L^z} X_{\mc L^z} > \mu r -\lambda_{j+1}\frac{Gz^2}{r} + \theta r^{1/3}\right) + \P\left(\tf(\Gamma_r^z) > \theta^{\frac{3}{4\alpha}}r^{2/3}\right).
\end{split}
\end{equation}
The second term is bounded by $\exp(-c'\theta^{3/2})$ by Proposition~\ref{p.tf} for all $\theta$ such that $\theta^{3/(4\alpha)} > s_0$, with $s_0$ an absolute constant as given in the statement of the corollary. We increase $\theta'$ if needed to meet this condition; this increase can be done in a way that depends only on $s_0$ as since $\alpha\leq 1$, it is sufficient if $\theta_0\geq \smash{s_0^{4/3}}$. 

We want to bound the first term of\eqref{e.breakup into max over discr and TF} by a union bound over all discretizations $\mc L^z$. First we bound the cardinality of the set of discretizations using Lemma~\ref{l.discretizations cardinality}. Note that the definition of $k$ in \eqref{e.upper tail upper bound k value} implies that $\log k \leq \frac{3\beta}{3-\beta}\log\theta$ as $\varepsilon<1$. Lemma~\ref{l.discretizations cardinality} asserts that the set of discretizations has cardinality at most $\exp\{k(\log k + \frac{3}{4\alpha}\log \theta + \log 2)\}$. The just mentioned bound on $\log k$ and the value of $k$ from \eqref{e.upper tail upper bound k value} shows that this cardinality is at most
$$\exp\left(\tilde c\varepsilon \theta^{\frac{3\beta}{3-\beta}}(\log \theta)^{-\frac{3}{3-\beta}+1}\right) = \exp\left(\tilde c\varepsilon \theta^{\frac{3\beta}{3-\beta}}(\log \theta)^{-\frac{\beta}{3-\beta}}\right),$$
with $\tilde c$ a constant which depends on only $\alpha$ and $\beta$.
Given this and the bound in \eqref{e.P(X_L) bound final}, we apply a union bound. This yields that, for $\theta_0' < \theta < r^{\zeta'}$, the first term of \eqref{e.breakup into max over discr and TF} is at most
%
\begin{align*}
2\exp\left(-\tfrac{1}{2}c'\cdot\varepsilon^{\beta/3}\theta^{\frac{3\beta}{3-\beta}}(\log \theta)^{-\frac{\beta}{3-\beta}} + \tilde c\cdot\varepsilon\theta^{\frac{3\beta}{3-\beta}}(\log \theta)^{-\frac{\beta}{3-\beta}}\right).
\end{align*}
Now since $\beta\leq 1$, for sufficiently small $\varepsilon$ it holds that $\tilde c \varepsilon \leq \frac{1}{4}c'\cdot\varepsilon^{\beta/3}$, and we fix $\varepsilon$ to such a value; note that $\varepsilon$ does not depend on $\theta$ and only on $c$, $\alpha$, $\beta$, and $j$. This can be seen since $\varepsilon$ depends on $\tilde c$ and $c'$, which respectively depend on $\alpha$ and $\beta$ only, and $c$, $\alpha$, $\beta$, and $j$. 

For this value of $\varepsilon$ and for $\theta_0' < \theta < r^{\zeta'}$, the previous display is bounded above by
$$\exp\left(-\tfrac{1}{4}c'\cdot\varepsilon^{\beta/3}\theta^{\frac{3\beta}{3-\beta}}(\log \theta)^{-\frac{\beta}{3-\beta}}\right).$$
Putting this bound into \eqref{e.breakup into max over discr and TF}  completes the proof of Proposition~\ref{p.upper tail bootstrap iteration} for $\theta_0'<\theta<r^{\zeta'}$ after relabeling $c'$ in its statement by $\frac{1}{4} c'\cdot\varepsilon^{\beta/3}$. For when $\theta> r^{\zeta'}$, the hypothesis \eqref{e.upper tail bootstrap hypothesis} provides the bound when $r>r_0$, which we ensure by raising  $r_0'$ (if necessary) to be at least $r_0$. This completes the proof of Proposition~\ref{p.upper tail bootstrap iteration} by relabeling $c'$ in its statement to be less than $c$ if needed.
\end{proof}


\subsection{Lower bound on upper tail}

We prove the lower bound on the upper tail, i.e., Theorem~\ref{t.upper tail lower bound}.

\begin{proof}[Proof of Theorem~\ref{t.upper tail lower bound}]
Assumption~\ref{a.limit shape assumption} implies that $\P(X_r \geq \mu r + \theta r^{1/3}) \leq \P\left(X_r \geq \E[X_r] + \theta r^{1/3}\right)$, and we prove the stronger bound that $\P(X_r \geq \mu r + \theta r^{1/3}) \geq \exp(-c\theta^{3/2})$ for appropriate $\theta$.

Observe that $X_r \geq \sum_{i=0}^n X_{r/k,i}$ where $X_{r/k,i} = X_{i(r/k, r/k)+(1,0), (i+1)(r/k, r/k)}$. Now by Assumption~\ref{a.lower bound upper tail} we have that
$$\P\left(X_{r/k,i} \geq \mu r/k + C(r/k)^{1/3}\right) \geq \delta$$
for each $i\in\intint{0,k-1}$, as long as $r/k > r_0$. Since 
$$\left\{\sum_{i=0}^k X_{r/k,i} \geq \mu r + Ck^{2/3} r^{1/3}\right\} \supseteq \bigcap_{i=0}^{k-1}\left\{X_{r/k,i} \geq \mu r/k + C(r/k)^{1/3}\right\},$$
we have
$$\P\left(X_r \geq \mu r + Ck^{2/3}r^{1/3}\right) \geq \delta^{k},$$
using the independence of the $X_{r/k,i}$ across $i$. Now we set $k=C^{-3/2}\theta^{3/2}$, giving
$$\P\left(X_r \geq \mu r + \theta r^{1/3}\right) \geq \exp(-c \theta^{3/2})$$
for some $c>0$ and for all $\theta$ satisfying $1\leq C^{-3/2} \theta^{3/2} \leq r/r_0$, which is equivalent to $C\leq \theta \leq Cr_0^{-2/3} \times r^{2/3}$. Thus the proof of Theorem~\ref{t.upper tail lower bound} is completed by setting $\theta_0 = C$ and $\eta = Cr_0^{-2/3}$.
\end{proof}

\section{Lower tail and constrained lower tail bounds}\label{s.lower tail bounds}

In this section we prove Theorems~\ref{t.optimal lower tail upper bound}, \ref{t.lower tail lower bound}, and \ref{t.constrained lower tail bounds}. In fact Theorem~\ref{t.constrained lower tail bounds} implies both of the other two, but we prove Theorem~\ref{t.optimal lower tail upper bound} first separately to aid in exposition.

\subsection{Upper bound on lower tail}
Note that the abstracted bootstrap statement Proposition~\ref{p.bootstrap} is applicable with $Y_r = -(X_r-\mu r)$ and \smash{$Y_{r,i}^{(k)} = -(X_{r/k, i}-\mu r/k)$}, where \smash{$X_{r/k,i}$} is the last passage value from $(i-1)/k\cdot(r,r) + (1,0)$ to $i/k\cdot(r,r)$ for $i\in\intint{1,k}$. As we have noted earlier, iterating this would yield a lower tail exponent of $3/2$ (a similar argument for the upper tail under \emph{sub}-additivity was outlined in the beginning of Section~\ref{s.discussion of techniques}) but will not be able to reach the optimal exponent of~$3$. 


Recall from Section~\ref{s.discussion of techniques} that our argument relies on a high-probability construction from \cite{watermelon} of $k$ disjoint paths with good collective weight, here Theorem~\ref{t.flexible construction}. Thus the probability the construction fails is an upper bound on the probability that many disjoint curves have small weight, which in turn bounds the probability that the geodesic has small weight, as we seek. 

As outlined before, the construction relies on three inputs: the first is the parabolic curvature on the limit shape, provided by Assumption~\ref{a.limit shape assumption}; the second is an exponential upper bound on the lower tail of the maximum weight among all paths constrained to stay within a given parallelogram; and the third is a lower bound on the mean of such weights. Recall that we call such weights ``constrained weights''. Like the first input, the third input is available to us already, and is the content of \eqref{e.constrained mean} of Proposition~\ref{p.constrained statements}. So only the second input needs to be attained via bootstrapping.

From here on the argument has two broad steps.
\begin{enumerate}
	\item Use our assumptions to obtain the exponential bound (in fact, we obtain an exponent of $3/2$) on the constrained weight's lower tail that can be used as an input for the construction in \cite{watermelon}.  This is Proposition~\ref{p.lower tail for constrained weight}. The argument uses bootstrapping as in Proposition~\ref{p.bootstrap}, and applies that proposition iteratively. 

	\item Relate the lower tail event of $X_r$ to the event of the existence of $k$ disjoint paths constructed in \cite{watermelon} (Theorem~\ref{t.flexible construction} here).
\end{enumerate}

We will implement these two steps in turn next, and then, in Section~\ref{s.melon construction overview}, we provide an overview of the main ideas of the construction from \cite{watermelon} that we are invoking. We start by specifying some notation for constrained weights.

Recall the notation for parallelograms introduced in Section~\ref{s.constrained weights}, where $U = U_{r,\ell, z}$ is a parallelogram of height $r$, width $\ell r^{2/3}$, and opposite side midpoints $(1,1)$ and $(r-z,r+z)$. Recall also that $Y_r^{U}$ is the maximum weight over all paths from $(1,1)$ to $(r-z,r+z)$ which are constrained to stay in $U$.

Proposition~\ref{p.constrained statements} provides a stretched exponential lower tail for $Y_r^U$ from our assumptions. The following upgraded tail  obtained via bootstrapping will suffice for our purpose; note that the bound is still not the optimal one stated in Theorem \ref{t.constrained lower tail bounds},  which we prove later.

\begin{proposition}\label{p.lower tail for constrained weight}
Let $L_1, L_2$, and $K$ be such that $L_1 < \ell < L_2$ and $|z|\leq Kr^{2/3}$. Under Assumptions~\ref{a.limit shape assumption} and \ref{a.one point assumption}, there exist positive constants $r_0$, $\theta_0$, and $c$, all depending on only $L_1$, $L_2$, $K$, and $\alpha$, such that, for $r>r_0$ and $\theta>\theta_0$,
$$\P\left(Y^U_r \leq \mu r - \theta r^{1/3}\right) \leq \exp(-c\theta^{3/2}).$$
\end{proposition}

This is the first step outlined above. The proof is similar to that outlined at the beginning of this section for $X_r$, and involves using bootstrapping for a number of iterations, with the exponent increasing by the end of each iteration to $3/2$ times its value at the start of it. Once the exponent passes 1, a final iteration brings it to $3/2$.

\begin{proof}[Proof of Proposition~\ref{p.lower tail for constrained weight}]
Consider the $k$ subparallelograms $U_i$, $i\in\intint{1,k}$, where $U_i$ is defined as the parallelogram with height $r/k$, width $\min(\ell r^{2/3}, (r/k)^{2/3})$, and opposite side midpoints $(r-z,r+z)\cdot (i-1)/k+(1,0)$ and $(r-z,r+z)\cdot i/k$. Let $Y_r =-(Y_r^U-\mu r)$ and $Y_{r,i}^{(k)} = -(Y_{r/k}^{U_i}-\mu r/k)$.

We want to apply Proposition~\ref{p.bootstrap} to these variables. By the definition of $Y_r^U$ and \smash{$Y^{U_i}_{r/k}$}, we have that $Y_r\leq \sum_{i=1}^k \smash{Y_{r,i}^{(k)}}$ for all $k\leq r$. The variables $\{\smash{Y_{r,i}^{(k)}}: i\in\intint{1,k}\}$ are independent for each $k$ as they are defined by the randomness in disjoint parts of the environment, and \eqref{e.constrained lower tail} of Proposition~\ref{p.constrained statements} provides a stretched exponential tail (of exponent $\alpha' = 2\alpha/3$) for each $\smash{Y_{r,i}^{(k)}}$. 
Since the constants $c$, $\theta_0$, and $r_0$ of Proposition~\ref{p.constrained statements} depend on only $K, L_1, L_2$, we obtain from Proposition~\ref{p.bootstrap} that there exist $\tilde c$, $\tilde r_0$, and $\tilde\theta_0$ (all depending on $K$, $L_1$, $L_2$, and $\alpha$) such that, for $r>\tilde r_0$ and $\tilde\theta_0 < \theta < r^{2/3}$,
$$\P\left(Y^U_r \leq \mu r - \theta r^{1/3}\right) \leq \exp(-\tilde c \theta^{3\alpha'/2}).$$
Note that  if $\mu >1$, the constraint  $\theta<r^{2/3}$ can be extended to $\theta< \mu r^{2/3}$ by reducing the constant $\tilde c$ if needed, in a way that depends on only $\alpha$ and $\mu$. Beyond $\mu r^{2/3}$, the probability on the left side of the last display is zero since the vertex weights are non-negative, and so the last displayed inequality actually holds for all $\theta>\tilde\theta_0$.

We may iterate the above argument, such that at the end of each iteration the tail exponent is $3/2$ times its value at the beginning, till the tail exponent exceeds 1. Then we may apply the above argument one last time with $\alpha'=1$, and this completes the proof. Since the finite number of iterations is only a function of $\alpha$, the proposition follows.
\end{proof}


We may now formally prove Theorem~\ref{t.flexible construction} under the weaker point-to-point tail assumptions of this paper (as compared to \cite{watermelon}) by detailing which statements of the latter paper need to be replaced by statements proved in this paper; an overview of the  construction given in \cite{watermelon} will be discussed shortly in Section~\ref{s.melon construction overview}.

\begin{proof}[Proof of Theorem~\ref{t.flexible construction}]
As mentioned, \cite[Theorem 3.1]{watermelon} has three inputs: (i) parabolic curvature of the weight profile; (ii) an exponential upper bound on the constrained weight lower tail; and (iii) a lower bound on the expected constrained weight. \cite[Assumption 2]{watermelon} provides (i), while (ii) and (iii) are provided by \cite[Proposition~3.7]{watermelon}. 

In this paper, Assumption~\ref{a.limit shape assumption} implies \cite[Assumption 2]{watermelon} and provides (i). The item (ii) is provided by Proposition~\ref{p.lower tail for constrained weight}, and (iii) by \eqref{e.constrained mean} of Proposition~\ref{p.constrained statements}. The proof of \cite[Theorem 3.1]{watermelon} applies verbatim after making these replacements.
\end{proof}

Next we prove Theorem~\ref{t.optimal lower tail upper bound} using Theorem~\ref{t.flexible construction}.




\begin{proof}[Proof of Theorem~\ref{t.optimal lower tail upper bound}]
Since $\E[X_r]\leq \mu r$ by Assumption~\ref{a.limit shape assumption} (this is also implied directly by the super-additivity of $\{X_r\}_{r\in\N}$), it is sufficient to upper bound the probability
$\P\left(X_r \leq \mu r- \theta r^{1/3}\right).$
Let $\eta$ be as given in Theorem~\ref{t.flexible construction}, and denote the event whose probability is lower bounded there by $E_{m,k,r}$, i.e., $E_{m,k,r}$ is the event that there exist $m$ disjoint paths $\gamma_1, \ldots, \gamma_m$ with prescribed endpoints, $\max_i \tf(\gamma_i) \leq 2mk^{-2/3}r^{2/3}$, and $\sum_{i=1}^m \ell(\gamma_i) \geq \mu rm - C_1mk^{2/3}r^{1/3}$. Observe that any of these paths $\gamma_i$ can be extended to a path from $(1,1)$ to $(r,r)$ without decreasing its weight. Now for $\theta\leq C_1\eta^{2/3}r^{2/3}$, set $m=k=C_1^{-3/2}\theta^{3/2}$, and observe that $C_1mk^{2/3} = m\theta$. Thus,
$$\P\left(X_r\leq \mu r - \theta r^{1/3}\right)\leq \P\left(E_{m,k,r}^c\right) \leq \exp(-cmk) = \exp(-c\theta^3),$$
the second inequality by Theorem~\ref{t.flexible construction} since the value of $k$ satisfies $k\leq \eta r$; this latter inequality is implied by the condition that $\theta\leq C_1\eta^{2/3} r^{2/3}$. The inequality for $\theta \in [C_1\eta^{2/3}r^{2/3},\mu r^{2/3}]$ can be handled be reducing the value of $c$ (if $C_1\eta^{2/3}<\mu$), and for $\theta > \mu r^{2/3}$, the probability being bounded is trivially zero. This completes the proof of Theorem~\ref{t.optimal lower tail upper bound}.
\end{proof}
We next present a brief outline of the construction from \cite[arXiv version 1]{watermelon} before going into the proof of Theorem~\ref{t.constrained lower tail bounds}, which then also implies Theorem~\ref{t.lower tail lower bound}.


\begin{figure}[h]
	\begin{subfigure}[t]{0.45\textwidth}
		\centering
		\includegraphics[width=\textwidth]{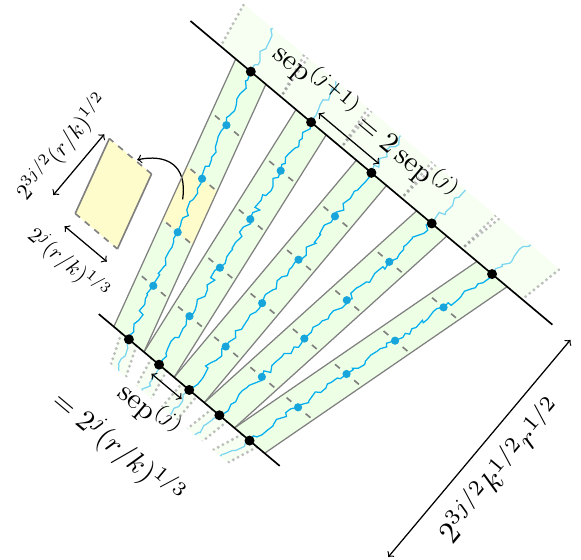}
		\caption{One scale of construction near corners}
		\label{f.construction segment 2}
	\end{subfigure} 
	\hspace{1.5cm}
	\begin{subfigure}[t]{0.35\textwidth}
		\centering
		\raisebox{6mm}{\includegraphics[width=\textwidth]{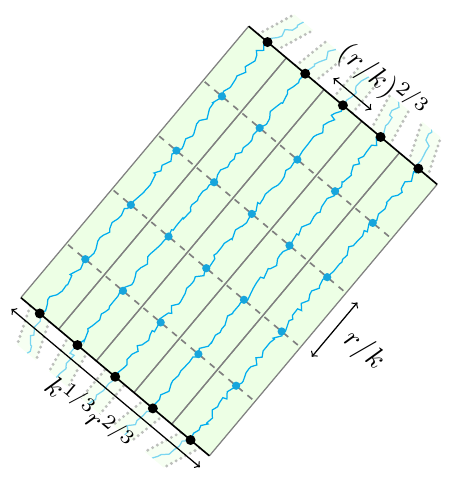}}
		\caption{Bulk phase of construction}
		\label{f.construction segment 3}
	\end{subfigure}
	\caption{Panel A is a depiction of one scale (indexed by $j$) of the construction near the corners of $\intint{1,r}^2$ when $m=k=5$. The paths which form the construction are in blue, and the separation on the $j$\textsuperscript{th} scale is denoted $\sep{j}$, which is also the width of the green parallelograms that the paths are constrained to pass through. Note that each individual green parallelogram has on-scale dimensions.
	Depicted in lower opacity is how the construction continues on the succeeding (larger) and preceding (smaller) scale. In panel B is the bulk phase of the construction, where the separation between curves is maintained for a distance of order $r$; thus there are order $k$ cells in a single column. Also depicted in lower opacity on either side are the largest scales of the second phase.}\label{f.construction}
\end{figure}

\subsection{An overview of the proof of Theorem~\ref{t.flexible construction}} \label{s.melon construction overview}

Here we give a brief overview of the high-probability construction that proves Theorem~\ref{t.flexible construction}, with the help of Figure~\ref{f.construction}. A detailed description and proof appears in \cite[Section~8]{watermelon}. These are for the $z=0$ case, and we will discuss a modification of Assumption~\ref{a.limit shape assumption} and of the construction which can handle $z\neq 0$ in the following Section~\ref{s.construction in other directions}.

Recall that we have to construct $m$ disjoint paths, each with weight loss at most of order $k^{2/3}r^{1/3}$, in a strip of width $4mk^{-2/3}r^{2/3}$. In the bulk of the environment, this is straightforward: for each curve, we set up order $k$ many parallelograms of width $(r/k)^{2/3}$ and height $r/k$ sequentially and consider the path obtained by concatenating together the heaviest midpoint-to-midpoint path constrained to remain in the corresponding parallelogram; see Figure~\ref{f.construction segment 3}. The weight loss in each parallelogram is on scale, i.e., of order $(r/k)^{1/3}$, and so the total loss across the $m$ curves is $m\cdot k\cdot(r/k)^{1/3} = mk^{2/3}r^{1/3}$. The total transversal fluctuation is of order $m(r/k)^{2/3}$, as required, and it is in this phase of the construction that the transversal fluctuation is maximum.

But the previous description is only possible in the bulk, and if the curves have already been brought to a separation of $(r/k)^{2/3}$. Since the curves start and end at a microscopic separation of $1$ at the corners of $\intint{1,r}^2$, the difficult part of the construction is there, where the curves must be coaxed apart while not sacrificing too much weight. Here the construction proceeds in a dyadic fashion, doubling the separation between curves as the scale increases, while ensuring that the antidiagonal displacement borne by the curves is not too high, so as to not incur a high weight loss due to parabolic curvature. Again the idea is to construct a sequence of parallelograms for each curve that it is constrained to remain within; see Figure~\ref{f.construction segment 2}. We require a curvature assumption such as Assumption~\ref{a.limit shape assumption} in order to estimate the weight loss in this phase of the construction, where the antidiagonal displacement is increasing, and a calculation shows that the weight loss is again of order $mk^{2/3}r^{1/3}$. The dyadic step-by-step separation performed at the bottom left corner is repeated in reverse at the top right corner to bring the curves back to a microscopic separation of~$1$.

The other two inputs, namely a lower bound on the means of constrained weights and exponential decay of the lower tails of the same, are needed to control the probability that the paths constructed by concatenating together these constrained paths have the requisite weight; in particular, the mean bound is used to show that the expected weight is correct, while the lower tail bound is used to control the deviation below the mean of the total construction weight after expressing it as a sum of independent subexponential variables and invoking a concentration inequality.

\subsection{The construction in other directions with a modified Assumption~\ref{a.limit shape assumption}}\label{s.construction in other directions}

Here we give a fairly detailed outline of the above construction in the case of other directions. This would then yield an upper bound on the lower tail also in other directions, thus expanding upon Remark~\ref{r.other directions}.

The basic difference between the $z=0$ case and the $z\neq 0$ cases is that the gradient of the limit shape is zero at $z=0$ (simply owing to symmetry, assuming the gradient exists), but is expected to be non-zero for other $z$ owing to strict convexity. Thus the modification to Assumption~\ref{a.limit shape assumption} needed for other directions is a precise encoding of this gradient term.

\subsubsection*{A form of a curvature assumption} One way to do this would be to encode it directly, i.e., modify Assumption~\ref{a.limit shape assumption} to have an extra linear term with the same $z$-dependent coefficient in both the implicit upper and lower bounds. A slightly cleaner approach is to consider the limit shape not for the LPP value from $(1,1)$ to $(r(1+z)+s,r(1-z)-s)$ (with $z$ fixed and as a function of $s$) but for the LPP value from $(1,1)$ to $(r(1+z)+s, r(1-z)-s)$ to $(2r(1+z), 2r(1-z))$, i.e., the best weight of a path from $(1,1)$ to $(2r(1+z),2r(1-z))$ which is constrained to go through $(r(1+z)+s, r(1-z)-s)$. Doing so, the gradient terms corresponding to $s$ on either side should cancel out, leaving a limit shape of the form $\mu_zr - G_z s^2/r$ to second order, as in the $z = 0$ case. This motivates a form of Assumption~\ref{a.limit shape assumption} in other directions.

\setcounter{countAssumption}{1}
\begin{assumption}\label{a.limit shape in general direction}
Fix $z\in(-1,1)$. There exist positive constants $\mu_z$, $G_z$, $H_z$, $g_{1,z}$, $g_{2,z}$, and $\rho_z$ such that, for all large enough $r$ and $|s|\leq \rho_z r$,
\begin{align*}
\E\left[X_{r}^{zr+s}+X_{r}^{zr-s}\right] \in \mu_z r - G_z \frac{s^2}{r} + \left[-H_z\frac{s^4}{r^3}, 0\right] + [-g_{1,z}r^{1/3}, g_{2,z} r^{1/3}].
\end{align*}

\end{assumption}

With this assumption, the above outlined construction can be performed with no conceptual changes, but we sketch out the details more now for the reader's convenience.

\subsubsection*{Antidiagonal displacement in $z\neq 0$} As we saw above, an important ingredient for the construction was control on the lower tail and mean of a geodesic constrained to lie in a on-scale parallelogram with antidiagonal displacement of its opposite sides also on-scale. In the $z=0$ case, the latter meant that the antidiagonal displacement was at most $Cr^{2/3}$ for some absolute constant $C$. In the $z\neq 0$ case, the antidiagonal displacement of a parallelogram in the construction (which has height $r/k$ and side lengths $(r/k)^{2/3}$) must be approximately $rz/k$, and the condition for on-scale anti-diagonal displacement is instead that the displacement is $rz/k\pm Cr^{2/3}$. (For brevity, while referring to antidiagonal displacement, in what follows we will ignore the ``linear" term of $rz/k$, and simply focus on the ``on-scale term" $Cr^{2/3}$. In these terms the condition for anti-diagonal displacement is again that it is at most $Cr^{2/3}$ in absolute value.) 

Thus to perform the construction in the $z\neq 0$ case, we first need to have control on the mean and lower tail of constrained geodesic weights from $(1,1)$ to $(r(1-z), r(1+z))$ for all large values of $r$. Analogous to the proof of Proposition~\ref{p.constrained statements} given in \cite{watermelon} in the $z=0$ case, these in turn follow from bounds on the transversal fluctuations of an unconstrained geodesic from $(1,1)$ to $(r(1-z), r(1+z))$. We give a brief outline of this argument in the $z\neq 0$ case next.

\subsubsection*{Transversal fluctuation control in $z\neq 0$} The transversal fluctuation bound argument follows that of Proposition~\ref{p.tf}. There the proof utilizes a dyadic decomposition, and so the main thing to be proved is that the probability that the geodesic from $(1,1)$ to $(r(1+z), r(1-z))$ is, at the \emph{midpoint}, greater than $sr^{2/3}$ from $(r/2(1+z), r/2(1-z))$ decays like $\exp(-cs^{2\alpha})$. In the case of $|z|\leq r^{-1/6}$ (i.e., $|rz| \leq r^{5/6}$), this is Proposition~\ref{p.midpoint tf}. For general $z$, the same basic argument given there applies, namely, one coarse grains the location of the geodesic at the midpoint into intervals of size of order $r^{2/3}$ and shows that paths corresponding to intervals at distance $jr^{2/3}$ from $(r/2(1+z), r/2(1-z))$ suffer a loss of order $j^2r^{1/3}$; the fundamental point is that, by Assumption~\ref{a.limit shape in general direction}, any path which goes from $(1,1)$ to $(r+z,r-z)$ via $(r/2+z/2+sr^{2/3}, r/2-z/2-sr^{2/3})$ suffers a weight loss in expectation of $s^2r^{1/3}$ compared to the one that does the same with $s=0$. Note also that the form of Assumption~\ref{a.limit shape in general direction} fits well with the above dyadic argument.


To summarize, using Assumption~\ref{a.limit shape in general direction}, we can prove the analogue of Proposition~\ref{p.midpoint tf} (transversal fluctuation at midpoint), which yields the analogue of Proposition~\ref{p.tf} (overall transversal fluctuation). This can then be used to obtain the analogue of Proposition~\ref{p.constrained statements} (statements about constrained weights).

\begin{figure}
\includegraphics[width=0.45\textwidth]{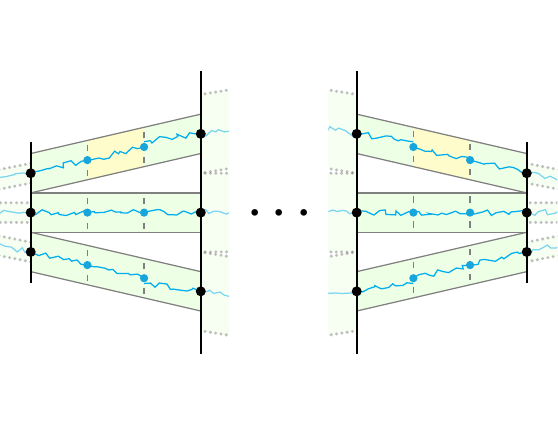}
\caption{An illustration of how to match parallelograms (marked in yellow) between the initial expansion and final compression parts of the construction (depicted in a rotated fashion) such that the parallelograms correspond to the same absolute value of deviation from the $(1,1)$ directions with opposite signs.} 
\label{f.matching parallelograms}
\end{figure}

\subsubsection*{Pairing of parallelograms to utilize Assumption~\ref*{a.limit shape assumption}$\,'$}
 So, if we have the input for the constrained weights whose argument was just outlined, we can apply it to each of the parallelograms of height $r/k$ (see Figures~\ref{f.construction} or \ref{f.matching parallelograms}) that are present in the construction. Observe, however, that the above arguments will yield a lower bound on the mean of the constrained weight between a pair of points in terms of the mean of the unconstrained weight between the same points, and similarly the upper bound on the lower tail of the constrained weight will be with respect to deviations from the same mean (which is expected to involve a gradient term). However we do not have any control on this expectation since Assumption~\ref*{a.limit shape in general direction} concerns only the sum of two such expectations, each corresponding to equal but opposite deviations from the $z$-parametrized direction. 

 For this reason, in the construction we instead consider pairs of constrained weights corresponding to parallelograms which have this opposite deviation property, as in Figure~\ref{f.matching parallelograms}. 
 Thus it is of importance that the construction has this symmetry to handle the gradient term in other directions.

With the construction in hand, the upper bound on the lower tail is immediate, exactly as in the proof of Theorem~\ref{t.optimal lower tail upper bound} above.

\subsection{Bounds on constrained lower tail} \label{s.constrained paths}

In this section we will prove Theorem~\ref{t.constrained lower tail bounds}; this will also imply Theorem~\ref{t.lower tail lower bound}. We start with the short proof of the upper bound of Theorem~\ref{t.constrained lower tail bounds}, which is a straightforward consequence of Theorem~\ref{t.flexible construction} and is a refinement of the argument for Theorem~\ref{t.optimal lower tail upper bound}.

\begin{proof}[Proof of upper bound of Theorem~\ref{t.constrained lower tail bounds}]
We prove a stronger bound with $\mu r$ in place of $\E[Y_r^U]$ (since $\E[Y_r^U]\leq \mu r$). 

Recall that the width of $U$ is $\ell r^{2/3}$. On the event that $Y_r^U \leq \mu r- \theta r^{1/3}$, it follows that any $m$ disjoint paths which lie inside $U$ must have total weight at most $\mu m r - m\theta r^{1/3}$. But for any $m$ which satisfies $2mk^{-2/3}r^{2/3}\leq \ell r^{2/3}$, with probability at least $1-\exp(-cm\theta^{3/2})$, there exist $m$ disjoint paths which lie inside $U$ with total weight at least $\mu m r - m\theta r^{1/3}$; this is the assertion of Theorem~\ref{t.flexible construction} with $k = C_1^{-3/2}\theta^{3/2}$.

Thus for any $m$ and $\ell$ satisfying $m\leq k$ and $m\leq \frac{1}{2}\ell k^{2/3}$ we have
$$\P\left(Y_r^U \leq \mu r- \theta r^{1/3}\right) \leq \exp(-cm\theta^{3/2}).$$
Taking $m = \min\left(\frac{1}{2}\ell k^{2/3}, k\right)$ with $k$ as above completes the proof if we verify that $m\geq 1$. This follows by setting $C = 2C_1$ in the assumed lower bound $\ell \geq C\theta^{-1}$ and setting $\theta_0>C_1$: the bound on $\ell$ and the value of $k$ implies that $\frac{1}{2}\ell k^{2/3} \geq 1$, and the bound on $\theta$ and the value of $k$ implies that $k\geq 1$. 
\end{proof}

The rest of this section is devoted to assembling the tools to prove, and then proving, the lower bound on the lower tail of Theorem~\ref{t.constrained lower tail bounds}. To start with, we need a constant lower bound on the lower tail of the point-to-point weight for a range of directions.  This is a straightforward consequence of the assumed mean behavior in Assumption \ref{a.limit shape assumption} and the lower tail bound in Assumptions~\ref{a.one point assumption lower}, and its proof is deferred to the appendix.
In fact, we only need $X_r^z < \mu r - Cr^{1/3}$ with positive probability in our application, but we prove a stronger statement with a parabolic loss.

\begin{lemma}\label{l.lower tail lower bound}
Let $\rho$ be as given in Assumption~\ref{a.limit shape assumption}. Under Assumptions~\ref{a.limit shape assumption} and \ref{a.one point assumption lower}, there exist positive constants $C$, $r_0$, and $\delta$ such that, for $r>r_0$ and $|z|\leq \rho r$,
$$\P\left(X_{r}^z < \mu r - \frac{Gz^2}{r} - Cr^{1/3}\right) \geq \delta.$$
\end{lemma}

The argument for the lower bound of Theorem~\ref{t.constrained lower tail bounds}, however, will require moving from the above lower bound on the point-to-point lower tail to a similar lower bound on the interval-to-line lower tail. Note that although we had previously encountered interval-to-interval weights, this is the first time in our arguments that we are seeking to bound interval-to-line weights.  This is the content of the next lemma.
The proof will entail a few steps which we will describe soon. For the precise statement recall that for two sets of vertices $A$ and $B$ in $\Z^2$, $X_{A,B}$ is the maximum weight of all up-right paths starting in $A$ and ending in $B$.

\begin{lemma}\label{l.int to line bound}
Let $I\subseteq \Z^2$ be the interval of lattice points connecting the coordinates $(-r^{2/3}, r^{2/3})$ and $(r^{2/3}, -r^{2/3})$ on the line $x+y=0$, and let $\L_r\subseteq \Z^2$ be the lattice points on the line $x+y=2r$. Under Assumptions~\ref{a.one point assumption lower} and \ref{a.lower tail lower bound}, there exist positive constants $C'$, $\delta'$, and $r_0'$ such that, for $r>r_0'$,
$$\P\left(X_{I, \L_r} \leq \mu r - C'r^{1/3}\right) \geq \delta'.$$
\end{lemma}

Before turning to the proof of Lemma~\ref{l.int to line bound}, we finish the proof of the  lower bound of Theorem~\ref{t.constrained lower tail bounds} and hence also the proof of Theorem~\ref{t.lower tail lower bound}.

\begin{proof}[Proof of Theorem~\ref{t.lower tail lower bound} and lower bound of Theorem~\ref{t.constrained lower tail bounds}]
Recall the lower bound statement of Theorem~\ref{t.constrained lower tail bounds} that, for $\theta_0\leq \theta \leq \eta r^{2/3}$,
$$\P\left(Y_r^U - \mu r \leq - \theta r^{1/3})\right) \geq \exp\left(-c_1\min(\ell\theta^{5/2}, \theta^3)\right);$$ 
note that Theorem~\ref{t.lower tail lower bound} is implied by the case that $\ell = r^{1/3}$ in Theorem~\ref{t.constrained lower tail bounds}, since by choosing $\eta< 1$ in the statement of the latter, we assume $\theta \leq r^{2/3}$, and hence $\min(\ell\theta^{5/2}, \theta^3)=\theta^3.$ 
We now proceed to proving the bound for general $\ell$.

Let $k$ and $m$ be positive integers whose values will be specified shortly.
We will define a grid similar to the one in Section~\ref{s.upper tail grid specification} that was depicted in Figure~\ref{f.grid}, but of width $mk^{-2/3}r^{2/3}$. For $i\in\intint{1,k}$ and $j\in\intint{1,m}$, let $v_{i, j}$ be the point
$$\left(i\frac{r}{k} - \frac{m}{2}\left(\frac{r}{k}\right)^{2/3},\,\, i\frac{r}{k} + \frac{m}{2}\left(\frac{r}{k}\right)^{2/3}\right) + j \left(\left(\frac{r}{k}\right)^{2/3},\,\, -\left(\frac{r}{k}\right)^{2/3}\right),$$
and let $I_{i,j}$ be the interval with endpoints $v_{i,j}$ and $v_{i,j+1}$ (of width $\left(\frac rk\right)^{2/3}$);. As in Section~\ref{s.upper tail grid specification}, we will collectively refer to these intervals as a grid, and so the rows of the grid are indexed by $i$ and the columns by $j$.
Note that, though the grid is shifted from the diagonal (since $j=0$ is excluded above), it lies inside $U$ if $m\leq \ell k^{2/3}$ and covers the breadth of $U$ if $m=\ell k^{2/3}$, since the total breadth of the grid is $mk^{-2/3}r^{2/3}$. The latter property is what we will be using.

This is what dictates the choice of $m,$ although for technical reasons, we set 
\begin{equation}\label{e.m value}
m=\min(\ell k^{2/3}, k),
\end{equation}
where our choice for $k$ later (of order $\theta^{3/2}$) will ensure that indeed for all interesting values of $\ell$ (i.e., $\ell=O(\theta^{1/2})$), we would have $m=\ell k^{2/3}$.

The idea now is to construct an event on which $Y_r^U \leq \mu r - \theta r^{1/3}$. Let $X_{I,\L}^{i,j}$ be the maximum weight among all paths with starting point on $I_{i,j}$ and ending point on the line $x+y = 2(i+1)\frac{r}{k}$. The event will be defined by forcing 
\begin{enumerate}
	\item all the $X_{I,\L}^{i,j}$ to be small, i.e., $X_{I,\L}^{i,j} \leq \mu(r/k) - C'(r/k)^{1/3}$ for a constant $C'$; and 

	\item any path which has transversal fluctuation greater than $k^{1/3}r^{2/3}$ to suffer a parabolic weight loss of order $k^{2/3}r^{1/3}$.
\end{enumerate}

Before proceeding, we let $Y_r^{k}$ be the maximum weight among all paths $\Gamma$ from $(1,1)$ to $(r,r)$ with transversal fluctuation satisfying $\tf(\Gamma) \geq k^{1/3}r^{2/3}$. Thus the second condition above says $Y_r^k$ falls below $\mu r$ by at least order $k^{2/3}r^{1/3}.$

 We claim that, on the event described, $Y_r^{U} \leq \mu r - \Omega(k^{2/3}r^{1/3})$, where $\Omega(k^{2/3}r^{1/3})$ is a quantity bounded below by a constant times $k^{2/3}r^{1/3}$; so, we are saying that $Y_r^U$ suffers a loss of order at least $k^{2/3}r^{1/3}$. This is due to the following. First, any path within the grid must pass through one of $\{I_{i,j}: j\in\intint{1,m}\}$ for every $i\in\intint{1,k}$ and so has weight at most $k\cdot \max_{i,j}X_{I,\L}^{i,j}\leq \mu r - C'k^{2/3}r^{1/3}$. Second, any path which exits the grid, by our choice of $m$, either exits $U$ and may be ignored or has transversal fluctuation greater than \smash{$k^{1/3}r^{2/3}$} and so suffers a weight loss of at least order $k^{2/3}r^{1/3}$.

Finally, we will show that this event has probability at least $\exp(-cmk)$ (since there are $mk$ values of $(i,j)$ for which $X_{I,\L}^{i,j}$ is made small) and set $k$ to be a multiple of $\theta^{3/2}$.

A more precise form of the above discussion starts with the following inclusion, where $c_2$ is as in Theorem~\ref{t.tf general}, $C'$ is as in Lemma~\ref{l.int to line bound}, and $Y_r^k$ is the weight of the best path with transversal fluctuation at least $k^{1/3}r^{2/3}$ (as defined above):
\begin{align*}
\MoveEqLeft[8]
\left\{Y_r^U \leq \mu r - \min(c_2, C')k^{2/3} r^{1/3}\right\}\\
&\supseteq \left\{Y_r^{k} \leq \mu r - c_2 k^{2/3} r^{1/3}\right\}\cap \bigcap_{\substack{1\leq i\leq k\\ 1\leq j \leq m}} \left\{X_{I,\L}^{i,j} \leq \mu r/k - C'(r/k)^{1/3}\right\}.
\end{align*}
Note that all the events on the right hand side are decreasing events. Hence, by the FKG inequality, 
\begin{align}
\P\left(Y_r^U \leq \mu r - \min(c_2, C')k^{2/3} r^{1/3}\right)
&\geq \P\left(Y_r^{k} \leq \mu r - c_2 k^{2/3} r^{1/3}\right) \nonumber\\
&\qquad\times\prod_{\substack{1\leq i\leq k\\ 1\leq j \leq m}}\P\left(X_{I,\L}^{i, j} \leq \mu r/k - C'(r/k)^{1/3}\right) \nonumber\\
&\geq (1-\exp(-\tilde ck^{2\alpha/3}))\cdot (\delta')^{mk}. \label{e.constrained lower tail lower bound}
\end{align}
The final inequality was obtained by applying Theorem~\ref{t.tf general} with $s = k^{1/3}$ and $t=0$ to lower bound the first term and Lemma~\ref{l.int to line bound} (with $r/k$ in place of $r$) to lower bound the remaining terms. Theorem~\ref{t.tf general} provides an absolute constant $s_0$ and its application requires $k^{1/3} > s_0$, a condition that will translate into a lower bound on $\theta$ after we set the value of $k$ next; Lemma~\ref{l.int to int lower bound} requires that $r/k>r_0$, which will translate to an upper bound on $\theta$.

Take $k= \left(\min(c_2, C')\right)^{-3/2}\theta^{3/2}$ and recall the value of $m$ from \eqref{e.m value}. Note that the assumed lower bound of $\ell\geq C\theta^{-1}$ ensures that $m\geq 1$; we additionally impose that $k\geq s_0^3$ to meet the requirement of Theorem~\ref{t.tf general} mentioned above. We also assume without loss of generality that $s_0\geq 1$ to encode that $k$ must be at least 1. Thus we obtain from \eqref{e.constrained lower tail lower bound}, for some constant $c_1>0$,
\begin{align}\label{e.almost final constrained lower tail lower bound}
\P\left(Y_r^U \leq \mu r - \theta r^{1/3}\right) \geq \exp\left(-c_1\min(\ell \theta^{5/2}, \theta^3)\right);
\end{align}
this holds for every $\theta$ which is consistent with $s_0^3 \leq k \leq r_0^{-1} r$, the latter inequality to ensure that $r/k$ is at least $r_0$ as obtained from Lemma~\ref{l.int to line bound}. Recalling the value of $k$, this condition on $\theta$ may be written as
$$s_0^2\cdot\min(c_2, C') \leq \theta \leq \min(c_2, C') r_0^{-2/3}\cdot r^{2/3}.$$
Recall that Theorem~\ref{t.constrained lower tail bounds} must be proven only for $\theta_0\leq \theta \leq \eta r^{2/3}$. Thus we may meet the condition of the last display by modifying $\theta_0$ to be greater than $s_0^2\min(c_2, C')$ and $\eta$ to be less than $\min(c_2, C')r_0^{-2/3}$, if required.

Now we turn to the final statement in Theorem~\ref{t.constrained lower tail bounds} on replacing $\mu r$ by $\E[Y_r^U]$ in \eqref{e.almost final constrained lower tail lower bound} when $\ell$ is bounded below by a constant $\varepsilon$. For this, all we require is that $\E[Y_r^U] \geq \mu r - Cr^{1/3}$ for $r>r_0$ for a $C$ and $r_0$ which may depend on $\varepsilon$. This is because with that bound we may absorb the $Cr^{1/3}$ into $\theta r^{1/3}$ by increasing the constant $c_1$ (which will then depend on $\varepsilon$). Now the required bound on $\E[Y_r^U]$ is provided by \eqref{e.constrained mean} of Proposition~\ref{p.constrained statements}.

This completes the proof of Theorem~\ref{t.lower tail lower bound} and the lower bound of Theorem~\ref{t.constrained lower tail bounds}
\end{proof}

The remaining task is to prove Lemma~\ref{l.int to line bound}, which lower bounds the lower tail for interval-to-line weights.  Recall from Figure~\ref{f.from point to interval} our strategy of obtaining bounds on interval-to-interval weights from similar bounds on point-to-point weights by backing up. Notice that we are presently seeking to lower bound the probability that  interval-to-line weights are low; in other words, expressing the line as a union of disjoint intervals, we want to lower bound the probability of the intersection of the decreasing events that all the corresponding  interval-to-interval weights are low. This will involve an application of the FKG inequality and the following two lemmas, which lower bound the lower tail of interval-to-interval weights.  The first one (Lemma~\ref{l.int to int lower bound}) treats the case when the anti-diagonal displacement between the two intervals is small. The second one (Lemma~\ref{l.Z tail bound 2}) handles intervals at greater anti-diagonal separation, exploiting the natural parabolic loss in the mean which makes the weights unlikely to be high in this case.

%

\begin{lemma}\label{l.int to int lower bound}
Let $I$ be the interval connecting the points $(-r^{2/3}, r^{2/3})$ and $(r^{2/3}, -r^{2/3})$, $J$ be the interval connecting the points $(r-z-r^{2/3}, r+z+r^{2/3})$ and $(r-z+r^{2/3}, r+z-r^{2/3})$, and $\rho$ be as in Assumption~\ref{a.limit shape assumption}.  Under Assumptions~\ref{a.limit shape assumption}, \ref{a.one point assumption lower}, and \ref{a.lower tail lower bound}, there exist positive constants $C''$, $\delta''$, and $r_0''$ such that, for $r>r_0''$ and $|z|\leq \rho r-2r^{2/3}$ (the antidiagonal displacement of $I$ and $J$ is $z$, so this condition means that we are in the case that $I$ and $J$ are close),
$$\P\left(X_{I,J} < \mu r - C'' r^{1/3}\right)\geq \delta''.$$
\end{lemma}

\begin{proof}
We first prove a similar statement for intervals of size $\varepsilon r^{2/3}$ for an $\varepsilon>0$ to be fixed later. This is a crucial first step as it is difficult to directly control the possible gain in weight afforded by allowing the endpoints to vary over an interval of size $2r^{2/3}$; initially using the leeway of making the interval sufficiently small (but on the scale $r^{2/3}$) makes this control achievable. This will be done using a strategy similar to the one illustrated in Figure~\ref{f.from point to interval} for the interval-to-interval upper tail bound Lemma~\ref{l.Z tail bound}, namely by considering backed up points (i.e., points further behind and ahead on either side of the intervals).

Let $|w|\leq \rho r$, and let $I_{\varepsilon}$ be the interval joining the points $(-\varepsilon r^{2/3}, \varepsilon r^{2/3})$ and $(\varepsilon r^{2/3}, -\varepsilon r^{2/3})$ and $J_\varepsilon$ be the interval joining the points $(r-w-\varepsilon r^{2/3}, r+w+\varepsilon r^{2/3})$ and $(r-w+\varepsilon r^{2/3}, r+w-\varepsilon r^{2/3})$. We will prove that there exist positive $C''$, $\delta$, and $\varepsilon$, independent of $w$, such that
\begin{equation}\label{e.small int to int lower bound}
\P\left(X_{I_\varepsilon, J_\varepsilon} < \mu r - C''r^{1/3}\right)\geq \delta/2.
\end{equation}
Let $u^*\in I_\varepsilon$ and $v^*\in J_\varepsilon$ be such that $X_{I_\varepsilon,J_\varepsilon} = X_{u^*, v^*}$.
Also let $\phi_1 = (-\varepsilon^{3/2}r, -\varepsilon^{3/2}r)$, $\phi_2 = (r-w+\varepsilon^{3/2}r, r+w+\varepsilon^{3/2}r)$ be the backed up points. Then we have the inequality
$$X_{\phi_1, \phi_2} \geq X_{\phi_1, u^*-(1,0)} + X_{I_\varepsilon,J_\varepsilon} + X_{v^*+(1,0), \phi_2}.$$
For $M$ to be fixed and $C$ as in Lemma~\ref{l.lower tail lower bound}, we will consider the constant order probability events
\begin{align*}
E_{p\to p} &:= \left\{X_{\phi_1, \phi_2} \leq \mu(1+2\varepsilon^{3/2})r - G\frac{w^2}{(1+2\varepsilon^{3/2}r)} - Cr^{1/3}\right\},\\
E_{p\to \mathrm{int}} &:= \left\{X_{\phi_1, u^*-(0,1)} > \mu \varepsilon^{3/2}r - M\varepsilon^{1/2}r^{1/3}\right\}, \qquad \text{and}\\
E_{\mathrm{int} \to p} &:= \left\{X_{v^*+(1,0), \phi_2} > \mu \varepsilon^{3/2}r - M\varepsilon^{1/2}r^{1/3}\right\}.
\end{align*}
On the intersection $E_{p\to p} \cap E_{p\to \mathrm{int}}\cap E_{\mathrm{int} \to p}$ we have
\begin{equation}\label{e.X_I,J bound}
X_{I_\varepsilon,J_\varepsilon} < \mu r - (C-2M\varepsilon^{1/2})r^{1/3}.
\end{equation}
We must lower bound the probability of this intersection. From Lemma~\ref{l.lower tail lower bound} we see
$$\P\left(E_{p\to p}\right) \geq \delta,$$
since $C < (1+\varepsilon^{3/2})^{1/3}C$. Next, recall that $u^*$ is a vertex of $I_{\varepsilon}$, which lies on the line $x+y=0$, which is the starting point of a heaviest path from $I_{\varepsilon}$ to $J_{\varepsilon}$. Thus $u^*$ is independent of the random field below the line $x+y = 0$. Now we see that, for large enough $M$ (depending only on $\delta$),
\begin{align*}
\P\left(E_{p\to \mathrm{int}}^c\right) 
&= \P\left(X_{\phi_1, u^*-(0,1)} \leq \mu\varepsilon^{3/2}r - M\varepsilon^{1/2}r^{1/3}\right)\\
&\leq \sup_{u\in I_\varepsilon}\P\left(X_{\phi_1, u-(0,1)} \leq \mu\varepsilon^{3/2}r - M\varepsilon^{1/2}r^{1/3}\right) \leq \frac{\delta}{4},
\end{align*}
with the mentioned independence allowing the uniform bound of the second line. 
The same bound with the same $M$ holds for $\P(E_{\mathrm{int} \to p}^c)$ as well; we fix this $M$. Now we set $\varepsilon > 0$ such that $C-2M\smash{\varepsilon^{1/2}} = C/2$. From \eqref{e.X_I,J bound} the above yields \eqref{e.small int to int lower bound} with $C'' = C/2$.

To move from $I_{\varepsilon}, J_{\varepsilon}$ to $I, J$, we let $I_{\varepsilon, i}$ for $i\in\intint{1,\varepsilon^{-1}}$ be the intervals of length $\varepsilon$ which make up the length one interval $I$ in the obvious way, and similarly for $J_{\varepsilon, j}$ and $J$. Next observe that
\begin{equation}\label{e.inclusion for int-to-int lower bound}
\left\{X_{I,J} < \mu r - \tfrac{1}{2}Cr^{1/3}\right\} \supseteq \bigcap_{i, j}\left\{X_{I_{\varepsilon, i}, J_{\varepsilon, j}}< \mu r - \tfrac{1}{2}Cr^{1/3}\right\}.
\end{equation}
Now, the bound \eqref{e.small int to int lower bound} holds as long as the intervals are of length \smash{$\varepsilon r^{2/3}$} and their antidiagonal displacement is at most $\rho r$. The intervals $I_{\varepsilon, i}$ and $J_{\varepsilon, i}$ have this length, and their antidiagonal displacement is at most $|z| + 2r^{2/3}$, where recall $z$ is the antidiagonal displacement between $I$ and $J$. This occurs, for example, when $I_{\varepsilon, i}$ is the left most subinterval of $I$ and $J_{\epsilon,j}$ is the right most sub interval of $J$. But since we have assumed $|z| +2r^{2/3}\leq \rho r$, the bound \eqref{e.small int to int lower bound} applies to $X_{I_{\varepsilon, i}, J_{\varepsilon,j}}$, and so the probability of each event in the intersection of \eqref{e.inclusion for int-to-int lower bound} is at least $\delta/2$. 

The intersection of \eqref{e.inclusion for int-to-int lower bound} is of decreasing events, and so we may invoke the FKG inequality and the just noted probability lower bound to conclude that the probability of the right hand side of \eqref{e.inclusion for int-to-int lower bound} is at least $(\delta/2)^{\varepsilon^{-2}}$. This completes the proof of Lemma~\ref{l.int to int lower bound} with $C'' = C/2$ and $\delta'' = (\delta/2)^{\varepsilon^{-2}}$.
\end{proof}

While the previous lemma provided control when the destination interval is relatively close to $(r,r)$, i.e., have $x$- and $y$-coordinates within $\rho r$ of $r$,  the next lemma will be used to treat pairs of intervals which have greater separation.
 Let $I\subseteq \Z^2$ be the interval connecting $(-r^{2/3}, r^{2/3})$ and $(r^{2/3}, r^{2/3})$, and $J\subseteq \Z^2$ be the interval connecting $(r-w-r^{2/3}, r+w+r^{2/3})$ and $(r-w+r^{2/3}, r+w-r^{2/3})$; thus $w$ represents the intervals' antidiagonal displacement, while $z$ will be used as a variable in the hypothesis.

\begin{lemma}\label{l.Z tail bound 2}
Suppose there exist $\alpha\in(0,1]$, $\lambda>0$ and positive constants $c$, $t_0$, and $r_0$ such that, for $t>t_0$, $r>r_0$, and $|z|\leq r/2$
$$\P\left(X^z_r \geq \mu r - \lambda\frac{Gz^2}{r} + t r^{1/3}\right) \leq \exp\left(-ct^\alpha\right).
$$
Then there exist positive $\tilde c$, $\tilde t_0 = \tilde t_0(t_0)$, and $\tilde r_0 = \tilde r_0(r_0)$ such that, for $r>\tilde r_0$, $|w|\leq |r|$, and $t>\tilde t_0$,
$$\P\left(X_{I,J} > \mu r -\frac{\lambda}{3}\cdot\frac{Gw^2}{r}+ t r^{1/3}\right) \leq \exp\left(-\tilde ct^\alpha\right).$$
\end{lemma}

We impose the condition that $|z|\leq r/2$ because the tail bound hypothesis will be provided by Assumption~\ref{a.one point assumption upper} in our application, and that assumption requires $|z|$ to be macroscopically (i.e., on the scale of $r$) away from $-r$ and $r$. But note that we allow $|w|$ to be as large as $r$, as we do need to allow the destination interval to be placed anywhere between the coordinate axes.

Lemma~\ref{l.Z tail bound 2} will be proved along with Lemma~\ref{l.Z tail bound}, as they are both interval-to-interval upper tails, in the appendix via a backing up argument.
With the bounds of the previous two lemmas, we can prove the interval-to-line bound in  Lemma~\ref{l.int to line bound} and thus complete the proof of Theorem~\ref{t.constrained lower tail bounds}.

As earlier, we will construct an event as an intersection of decreasing events which forces $X_{I,\L_r}$ to be small, and use the FKG inequality to lower bound its probability. So, we need to make any path starting in $I$ and ending on the line $\L_r$ have low weight, which we will do by forcing such paths which end on various intervals on $\L_r$ to separately have low weight. When the destination interval is close to the point $(r,r)$, the probability that all such paths have low weight will be lower bounded by Lemma~\ref{l.int to int lower bound}. When the destination interval is far from $(r,r)$, Lemma~\ref{l.Z tail bound 2} says that it is highly likely that the paths will have low weight, and it is a matter of checking that the probabilities approach 1 quickly enough that their product is lower bounded by a positive constant.

\begin{proof}[Proof of Lemma~\ref{l.int to line bound}]
For $j\in\intint{-r^{1/3}, r^{1/3}}$, let $J_j$ be the interval connecting the points $(r-jr^{2/3}, r+jr^{2/3})$ and $(r-(j+1)r^{2/3}, r+(j+1)r^{2/3})$. We observe that, with $C''$ as in Lemma~\ref{l.int to int lower bound},
\begin{equation}\label{e.int to line inclusion}
\left\{X_{I,\L_r} < \mu r - C''r^{1/3}\right\} \supseteq \bigcap_{|j|\leq r^{1/3}} \left\{X_{I, J_j} < \mu r - C'' r^{1/3}\right\}.
\end{equation}
%
%
Assumption~\ref{a.limit shape assumption} says that, for $|z|\leq \rho r$, $\E[X_r^z] \leq \mu r - Gz^2/r - g_1 r^{1/3}.$ But then observe that the concavity of the limit shape implies the existence of a small constant $\lambda\in(0,1)$ such that 
\begin{equation}\label{e.crude parabolic loss}
\E[X_r^z] \leq \mu r - \lambda\frac{Gz^2}{r}
\end{equation}
for $|z|\leq r$. For this value of $\lambda$, Assumption~\ref{a.one point assumption upper}, with $\varepsilon=1/2$, implies that there exists $c>0$ such that, for all $r>r_0$ and $|z|\leq r/2$,
$$\P\left(X_r^z > \mu r - \lambda\frac{Gz^2}{r} + t r^{1/3}\right) \leq \exp(-ct^\alpha).$$
With this value of $\lambda$ and the above bound as input, we will apply Lemma~\ref{l.Z tail bound 2} with $t = \lambda Gj^2/3-C''$. Lemma~\ref{l.Z tail bound 2} requires $t>\tilde t_0$, which implies that $|j|$ must be larger than some $j_0$. So for $|j|>j_0$, Lemma~\ref{l.Z tail bound 2} implies that
\begin{equation}\label{e.int to int small j}
\P\left(X_{I, J_j} < \mu r - C''r^{1/3}\right) \geq 1-\exp(-\tilde c|j|^{2\alpha}).
\end{equation}
Applying \eqref{e.int to int small j} and the bound from Lemma~\ref{l.int to int lower bound} to \eqref{e.int to line inclusion}, along with the FKG inequality, gives
\begin{align*}
\P\left(X_{I,\L_r} < \mu r - C''r^{1/3}\right)
&\geq \prod_{|j|\leq j_0} \P\left(X_{I, J_j} < \mu r - C'' r^{1/3}\right) \times \prod_{j_0<|j|\leq r^{1/3}} \P\left(X_{I, J_j} < \mu r - C'' r^{1/3}\right)\\
&\geq (\delta'')^{2j_0+1}\times \prod_{j_0\leq |j|\leq r^{1/3}} \left(1-\exp(-\tilde c|j|^\alpha)\right);
\end{align*}
we employed Lemma~\ref{l.int to int lower bound} for the factors in the first product and \eqref{e.int to int small j} for those in the second. The proof of Lemma~\ref{l.int to line bound} is complete by taking $C' = C''$ and $\delta'$ equal to the final line of the last display; note that the second product in the last display is bounded below by a positive constant independent of $r$ since $\exp(-\tilde c|j|^{\alpha})$ is summable over $j$ for any $\alpha>0$, and so $\delta'>0$.
\end{proof}



\appendix

\section{Transversal fluctuation, interval-to-interval, and crude lower tail bounds}\label{s.appendix}

In this appendix we explain how to obtain the first, second, and fourth tools of Section~\ref{s.tools}, i.e., Theorem~\ref{t.tf general} and Propositions~\ref{p.tf} and \ref{p.constrained statements}, and provide the outstanding proofs of Lemmas~\ref{l.Z tail bound}, \ref{l.Z tail bound 2}, and \ref{l.lower tail lower bound} from the main text. The third tool was already explained in Section~\ref{s.lower tail bounds}.

The proofs of the first and fourth tools follow verbatim from corresponding results in \cite{watermelon} by replacing the upper and lower tails used there with Assumption~\ref{a.one point assumption}; the parabolic curvature assumption there is provided by our Assumption~\ref{a.limit shape assumption}. In particular, Theorem~\ref{t.tf general} follows from \cite[Theorem 3.3]{watermelon} and Proposition~\ref{p.constrained statements} from \cite[Proposition 3.7]{watermelon}.



The proof of the second tool, Proposition~\ref{p.tf}, will be addressed in Section~\ref{s.app tf}, after we next provide the outstanding proofs of Lemmas~\ref{l.Z tail bound}, \ref{l.Z tail bound 2}, and \ref{l.lower tail lower bound} from the main text.

We start with the proofs of Lemmas~\ref{l.Z tail bound} and \ref{l.Z tail bound 2} on an upper tail bound of the interval-to-interval weight; this largely follows the proof of \cite[Proposition 3.5]{watermelon}. 
The strategy is to back up from the intervals appropriately and consider a point-to-point weight for which we have tail bounds by hypothesis; this strategy was illustrated in Figure~\ref{f.from point to interval} and in the proof of Lemma~\ref{l.int to int lower bound}. 

%



\begin{proof}[Proofs of Lemmas~\ref{l.Z tail bound} and \ref{l.Z tail bound 2}]
	We prove Lemma~\ref{l.Z tail bound} and indicate at the end the modifications for Lemma~\ref{l.Z tail bound 2}. We set $\lambda = \lambda_j$ and $\lambda' = \lambda_{j+1}$ to avoid confusion later when we describe the modifications for Lemma~\ref{l.Z tail bound 2}. Note that $\lambda' < \lambda$.

	Recall that $Z$ is defined as the largest last passage value between a pair of points when the points are each allowed to vary over an interval. We call $A$ the interval over which the starting point varies, $B$ the interval over which the ending point varies, and $U$ the region between $A$ and $B$ (i.e., the set of all points covered by some up-right path from $A$ to $B$). $A$ and $B$ were called $\linelower$ and $\lineupper$ in Lemma~\ref{l.Z tail bound} and $I$ and $J$ in Lemma~\ref{l.Z tail bound 2}.

	By considering the event that $Z$ is large and two events defined in terms of the environment outside of $U$, we find a point-to-point path which has large length; see Figure~\ref{f.from point to interval}. To define these events, first define points $\philower$ and $\phiupper$ on either side of the lower and upper intervals as follows, where $\delta = \frac{1}{2}\left(\frac{\lambda}{\lambda'}-1\right) > 0$:
	\begin{align*}
	\philower &:= \left(-\delta r, -\delta r\right)\\
	\phiupper &:= \left((1+\delta)r-w, (1+\delta)r + w\right).
	\end{align*}
	Let $u^*$ and $v^*$ be the points on $A$ and $B$ where the suprema in the definition of $Z$ are attained, and let the events $\Elow$ and $\Eup$ be defined as
	\begin{align*}
	\Elow &= \left\{X_{\philower, u^*-(1,0)}  > \mu \delta r - \frac{t}{3}r^{1/3}\right\}\quad \text{and}\quad
	\Eup = \left\{X_{v^*+(1,0), \phiupper}  > \mu \delta r - \frac{t}{3} r^{1/3}\right\}.
	\end{align*}
	Let $\tilde r = \frac{\lambda}{\lambda'}r = (1+2\delta)r$, and observe that the diagonal distance of $\philower$ and $\phiupper$ is $\tilde r$. Also note
	$$X_{\philower,\phiupper} \geq X_{\philower, u*-(1,0)} + Z + X_{v^*+(1,0),\phiupper}.$$
	Then we have the following:
	\begin{align}\label{e.side-to-side decomposition}
	\P\bigg(Z > \mu r -\lambda'\frac{Gw^2}{r}+t r^{1/3}, \Elow, \Eup\bigg)
	&\leq \P\left(X_{\philower, \phiupper} \geq \mu(1+2\delta) r - \lambda' \frac{Gw^2}{r}+ \frac{t}{3} r^{1/3}\right)\\
	&= \P\left(X_{\philower, \phiupper} \geq \mu \tilde r - \lambda \frac{Gw^2}{\tilde r} + \frac{t}{3}\cdot\left(\frac{\lambda'}{\lambda}\right)^{1/3}\cdot (\tilde r)^{1/3}\right)\nonumber\\
	&\leq 
	\begin{cases}
		\exp(-\tilde ct^\beta) & t_0 < t < r^{\zeta}\\
		\exp(-\tilde ct^\alpha) & t \geq r^{\zeta}.
	\end{cases}\nonumber
	\end{align}
	The final inequality uses the hypothesis \eqref{e.upper tail bootstrap hypothesis} on the point-to-point tail, which is applicable since the antidiagonal separation of $\philower$ and $\phiupper$ is $w$ while the diagonal separation is $(1+2\delta)r$, and clearly $|w|\leq r^{5/6}$ implies $|w| \leq (1+2\delta)^{5/6}r^{5/6}$. We applied \eqref{e.upper tail bootstrap hypothesis} with $\theta = t(\lambda'/\lambda)^{1/3}/3$, which is required to be greater than $\theta_0$. This translates to $t\geq t_0$ for a $t_0$ depending on $\theta_0$ and $\lambda'/\lambda$. Similarly, we absorbed the $\lambda'/\lambda$ dependency in the tail into the value of $\tilde c$, which thus depends on the original tail coefficient $c$  in \eqref{e.upper tail bootstrap hypothesis} and $\lambda'/\lambda$.

	Let us denote conditioning on the environment $U$ by the notation $\P(\,\cdot \mid U)$. By this we mean we condition on the weights of vertices interior to $U$ as well as those on the upper and lower sides $A$ and $B$.
	Then we see
	\begin{align*}
	\MoveEqLeft[7]
	\P\bigg(Z > \mu r - \lambda'\frac{Gw^2}{r}+t r^{1/3}, \Elow, \Eup\mid U\bigg)\\
	&= \P\left(Z > \mu r - \lambda'\frac{Gw^2}{r}+t r^{1/3}\mid U\right)\cdot \P\left(\Elow\mid U\right)\cdot \P\left(\Eup \mid U\right).
	\end{align*}
	So with \eqref{e.side-to-side decomposition}, all we need is a lower bound on $\P\left(\Elow\mid U\right)$ and $\P\left(\Eup\mid U\right)$. This is straightforward using independence of the environment below and above $U$ from $U$:
	\begin{align}
	\P\left(E^c_{\mathrm{lower}} \mid U\right) &\leq \sup_{u\in A} \,\P\left(X_{\philower, u-(1,0)} \leq \mu\delta r - \frac{t}{3}r^{1/3}\right)
	\leq \frac12 \label{e.elow comp bound}
	\end{align}
	for large enough $t$ (independent of $\delta$) and $r$ (depending on $\delta$), using Assumption~\ref{a.one point assumption lower}. 
	A similar upper bound holds for $\P\left(E^c_{\mathrm{upper}}\mid U\right)$. Together this gives
	\begin{align*}
	\P\bigg(Z > \mu r - \lambda'\frac{Gw^2}{r}+t r^{1/3}, \Elow, \Eup\mid U\bigg)
	&\geq \frac14\cdot \P\bigg(Z >\mu r - \lambda'\frac{Gw^2}{r}+ t r^{1/3} \mid U \bigg),
	\end{align*}
	and taking expectation on both sides, combined with \eqref{e.side-to-side decomposition}, gives Lemma~\ref{l.Z tail bound}. The fact that $\lambda'/\lambda$ depends only on $j$ and the previously mentioned dependencies gives the claimed dependencies of $\tilde t_0, \tilde r_0,$ and $\tilde c$.

	To prove Lemma~\ref{l.Z tail bound 2}, we take $\delta = 1$, which is equivalent to $\lambda' = \lambda/3$. Then in \eqref{e.side-to-side decomposition} the final bound is done with the hypothesized bound on $X_{\philower, \phiupper}$, i.e.,
	$$\P\left(X_{\philower, \phiupper} > \mu\tilde r - \lambda\frac{Gw^2}{\tilde r} + tr^{1/3}\right) \leq \exp(-\tilde ct^\alpha).$$
	Applying this bound requires $|w|\leq \tilde r/2$. Since $|w|\leq r$ and $\tilde r = \lambda r/\lambda' = 3r$, this is valid.
\end{proof}

Next we prove Lemma~\ref{l.lower tail lower bound}, on a constant probability lower bound on the lower tail, based on Assumptions~\ref{a.limit shape assumption} and \ref{a.one point assumption lower}.

\begin{proof}[Proof of Lemma~\ref{l.lower tail lower bound}]
Let $\tilde X_r^z = X_r^z - \mu r + Gz^2/r$. We know from Assumption~\ref{a.limit shape assumption} that $\E[\tilde X_r^z] \leq -g_2 r^{1/3}$. 
Let $E$ be the event
$$E = E(\theta) = \left\{X_r^z < \mu r - \frac{Gz^2}{r} -\theta r^{1/3}\right\},$$
so that $\P(E) \leq \exp(-c\theta^\alpha)$ for $\theta > \theta_0$, by Assumption~\ref{a.one point assumption lower}.

Observe that $-\tilde X_r^z\one_E$ is a positive random variable and so, by Assumption~\ref{a.one point assumption lower},
\begin{align*}
\E[-\tilde X_r^z\one_E] 
&= r^{1/3}\int_0^\infty \P\left(\tilde X_r^z\one_E < -tr^{1/3}\right)\, \mathrm dt\\
&= r^{1/3}\left[\theta \cdot\P\left(X_r^z < \mu r - \frac{Gz^2}{r}-\theta r^{1/3}\right) + \int_\theta^\infty \P\left(X_r^z <\mu r- \frac{Gz^2}{r} -tr^{1/3}\right)\, \mathrm dt\right]\\
&\leq r^{1/3}\left[\theta \exp(-c\theta^\alpha)+ \int_\theta^\infty \exp(-ct^\alpha)\, \mathrm dt\right];
\end{align*}
this may be made smaller than $0.5g_2r^{1/3}$ by taking $\theta$ large enough. We now set $\theta$ to such a value.

We also have $\E[\tilde X_r^z] = \E[\tilde X_r^z(\one_E + \one_{E^c})].$
Combining this, the above lower bound on $\E[\tilde X_r^z\one_E]$, and the upper bound on $\E[\tilde X_r^z]$, gives that
\begin{equation}\label{e.X_r^z E^c mean bound}
\E[\tilde X_r^z\one_{E^c}] \leq -\frac{1}{2}g_2r^{1/3}.
\end{equation}
The fact that the distribution of $\tilde X_r^z\one_{E^c}$ is supported on $[-\theta r^{1/3}, \infty)$ implies that
\begin{equation}\label{e.truncated X lower tail lower bound}
\P\left(X_r^z\one_{E^c} < \mu r - \frac{Gz^2}{r} - \frac{1}{4}g_2r^{1/3}\right) \geq \frac{g_2}{4\theta};
\end{equation}
this follows from \eqref{e.X_r^z E^c mean bound} and since
\begin{align*}
\E[\tilde X_r^z\one_{E^c}] 
&\geq -\theta r^{1/3}\cdot \P\left(\tilde X_r^z\one_{E^c} < -\frac{1}{4}g_2 r^{1/3}\right)  -\frac{1}{4}g_2r^{1/3}\P\left(\tilde X_r^z\one_{E^c} \geq -\frac{1}{4}g_2 r^{1/3}\right)\\
&\geq -\theta r^{1/3}\cdot \P\left(\tilde X_r^z\one_{E^c} < -\frac{1}{4}g_2 r^{1/3}\right)  -\frac{1}{4}g_2r^{1/3}.
\end{align*}
Since $\tilde X_r^z \one_E < 0$, it follows that $\tilde X_r^z \leq \tilde X_r^z\one_{E^c}$, so \eqref{e.truncated X lower tail lower bound} gives a lower bound on the lower tail of $X_r^z$, as desired, with $C = \frac{1}{4}g_2$ and $\delta = g_2/4\theta$.
\end{proof}

\subsection{Proof of transversal fluctuation bound, Proposition~\ref{p.tf}}\label{s.app tf}

In this section we prove Proposition~\ref{p.tf} on the tail (with exponent $2\alpha$) of the transversal fluctuation of the geodesic path on scale $r^{2/3}$; we closely follow the proof of Theorem~11.1 of the preprint \cite{slow-bond}, but adapted to our setting and assumptions. We give the argument for the left-most geodesic $\Gamma_r^z$ from $(1,1)$ to $(r-z,r+z)$; the argument is symmetric for the right-most geodesic. (Note that these are well-defined by the planarity and the weight-maximizing properties of all geodesics).

We start with a similar bound at the midpoint of the geodesic, which needs some notation. For $x\in\intint{1,r}$, let $\Gamma_r^z(x)$ be the unique point $y$ such that $(x-y,x+y) \in \Gamma_r^z$.

\begin{proposition}\label{p.midpoint tf}
Under Assumption~\ref{a.limit shape assumption} and \ref{a.one point assumption}, there exist positive  $c = c(\alpha)$, $r_0$, and $s_0$ such that, for $r>r_0$, $s>s_0$, and $|z|\leq {r^{5/6}}$,
$$\P\left(|\Gamma_r^z(r/2)| > z/2+sr^{2/3}\right) \leq 2\exp(-cs^{2\alpha}).$$
\end{proposition}

To prove this we will need a bound on the maximum, $\widetilde Z$, of fluctuations of the point-to-point weight as the endpoint varies over an interval, i.e.,
$$\widetilde Z = \sup_{v\in \lineupper}\Big(X_{v} - \E[X_{v}]\Big),$$
where $\lineupper$ is the interval of width $2r^{2/3}$ around $(r-w,r+w)$. Note that this is not the same as the point-to-interval weight.

\begin{lemma}\label{l.Z tail 3}
Let $K>0$ and $|w|\leq K r^{5/6}$. Under Assumptions~\ref{a.limit shape assumption} and \ref{a.one point assumption}, there exist positive $c$, $\theta_0 = \theta_0(K)$, and $r_0$, such that, for $\theta>\theta_0$ and $r>r_0$,
$$\P\left(\widetilde Z > \theta r^{1/3}\right) \leq \exp(-c\theta^{\alpha}).$$
\end{lemma}

\newcommand{\Mlow}{M_{\mathrm{low}}}
\newcommand{\Mup}{M_{\mathrm{up}}}

\begin{proof}
The proof is very similar to that of Lemma~\ref{l.Z tail bound} above.

We take $\phiupper = (2(r-w), 2(r+w))$ to be the backed up point. Let $v^*\in\lineupper$ be the maximizing point in the definition of $\widetilde Z$. For clarity, define the lower and upper mean weight functions $\Mlow$ and $\Mup$ by $\Mlow(v) = \E[X_{v}]$ and $\Mup(v) = \E[X_{v,\phiupper}]$; this is to use the unambiguous notation $\Mlow(v^*)$ (which is a function of $v^*$) instead of $\E[X_{v^*}]$. We also define 
\begin{align*}
\Eup &= \left\{X_{v^*+(1,0), \phiupper} - \Mup(v^*+(1,0)) > - \frac{\theta}{2}r^{1/3}\right\}.
\end{align*}
Now observe that
\begin{equation}\label{e.centred domination}
\begin{split}
\MoveEqLeft[12]
X_{v^*} - \Mlow(v^*) + X_{v^*+(1,0), \phiupper} - \Mup(v^*+(1,0))\\
&\leq X_{\phiupper} - \inf_{v\in\lineupper}\left(\Mlow(v) + \Mup(v+(1,0))\right).
\end{split}
\end{equation}
We want to replace the infimum on the right hand side by $\E[X_{\phiupper}]$. The latter is at most $2\mu r - 2Gw^2/r$. We need to show that the infimum term is at least something which is within $O(r^{1/3})$ of this expression. For this we do the following calculation using Assumption~\ref{a.limit shape assumption}. Parametrize $v\in\lineupper$ as $(r-w-tr^{2/3}, r+w+tr^{2/3})$ for $t\in[-1,1]$. Then, for all $t\in[-1,1]$,
\begin{align*}
\Mlow(v) + \Mup(v+(1,0))&\geq \left[\mu r - \frac{G(w+tr^{2/3})^2}{r} - H\frac{(w+tr^{2/3})^4}{r^3}\right]\\
&\qquad + \left[\mu r - \frac{G(w-tr^{2/3})^2}{r} - H\frac{(w-tr^{2/3})^4}{r^3}\right]\\
&\geq 2\mu r - \frac{2G w^2}{r}-2Gt^2r^{1/3} - 32HK^4r^{1/3},
\end{align*}
the last inequality since $|w\pm tr^{2/3}|\leq 2Kr^{5/6}$. Since $t\in[-1,1]$, $2Gt^2r^{1/3}\leq 2Gr^{1/3}$, and so the right hand side of \eqref{e.centred domination} is at most $X_{\phiupper} - \E[X_{\phiupper}] + \frac{\theta}{4}r^{1/3}$ for all large enough $\theta$ (depending on $K$). Thus, recalling the definition of $\Eup$,
\begin{align*}
\P\left(\widetilde Z > \theta r^{1/3}, \Eup\right) \leq \P\left(X_{\phiupper} - \E[X_{\phiupper}] > \frac{\theta}{4}r^{1/3})\right) \leq \exp(-c\theta^\alpha).
\end{align*}
We now claim that, conditionally on $v^*$, $\Eup$ almost surely has  probability at least $1/2$; since $\Eup$ is conditionally independent, given $v^*$, of $\widetilde Z$, this will imply with the previous display that $\smash{\P(\widetilde Z > \theta r^{1/3})\leq 2\exp(-c\theta^\alpha)}$. The proof of the claim is straightforward using the independence of $v^*$ with the environment above $\lineupper$ and Assumption~\ref{a.one point assumption lower}, for
\begin{align*}
\P\left(\Eup^c \mid v^*\right) \leq \sup_{v\in\lineupper}\P\left(X_{v+(1,0)} - \Mup(v+(1,0)) \leq -\frac{\theta}{2}r^{1/3}\right) \leq 1/2,
\end{align*}
for all $\theta$ larger than an absolute constant.
\end{proof}

\begin{proof}[Proof of Proposition~\ref{p.midpoint tf}]
We will prove the bound for the event that $\Gamma_r^z(r/2) > z/2 + sr^{2/3}$, as the event that it is less than $-z/2-sr^{2/3}$ is symmetric. 

For $j\in\intint{0,r^{1/3}}$, let $I_j$ be the interval 
\begin{align*}
\left(\frac{r}{2}-\frac{z}{2}-sr^{2/3}, \frac{r}{2}+\frac{z}{2}+sr^{2/3}\right) - [j,j+1]\cdot (r^{2/3}, -r^{2/3}).
\end{align*}
Let $A_j$ be the event that $\Gamma_r^z$ passes through $I_j$, for $j\in\intint{0,r^{1/3}}$. Observe that 
\begin{equation}\label{e.tf midpoint inclusion}
\left\{\Gamma_r^z(r/2) > z/2 + sr^{2/3}\right\} \subseteq \bigcup_{j=0}^{r^{1/3}} A_j.
\end{equation}
We claim that $\P(A_j) \leq \exp(-c(s+j)^{2\alpha})$ for each such $j$; this will imply Proposition~\ref{p.midpoint tf} by a union bound which we perform at the end. 

Let $Z_j^{(1)} = X_{(1,1), I_j}$ and \smash{$Z^{(2)}_j = X_{I_j, (r-z,r+z)}$}.  Also, let $\widetilde Z_j^{(1)} = \sup_{v\in I_j} (X_{v} - \E[X_v])$, and define $\widetilde Z_j^{(2)}$ analogously. 

We have to bound the probability of $A_j$. The basic idea is to show that any path from $(1,1)$ to $(r-z,r+z)$ which passes through $I_j$ suffers a weight loss greater than that which $X_r^z$ typically suffers (which is of order $Gz^2/r$), and so such paths are not competitive.  When $j$ is very large, it is possible to show this even if we do not have the sharp coefficient of $G$ for the parabolic loss; but for smaller values of $j$, we will need to be very tight with the coefficient of the parabolic loss. So we divide into two cases, depending on the size of $j$, and first address the case when $j$ is large (in a sense to be specified more precisely shortly).  Observe that, for a $c_2>0$ to be fixed,
\begin{align*}
\P\left(A_j\right) \leq \P\left(X_r^z < \E[X_r^z] - c_2 (s+j)^2r^{1/3}\right) + \P\left(Z_j^{(1)} + Z_j^{(2)} > \E[X_r^z] - c_2(s+j)^2r^{1/3}\right);
\end{align*}
the first term is bounded by $\exp(-c(s+j)^{2\alpha})$ by Assumption~\ref{a.one point assumption lower} for a $c$ depending on $c_2$, and we must show a similar bound for the second. Note that the second term is bounded by 
\begin{equation}\label{e.sum of Zs}
\P\left(Z_j^{(1)} + Z_j^{(2)} > \mu r - \frac{Gz^2}{r} - Hr^{1/3} - c_2(s+j)^2r^{1/3}\right),
\end{equation}
using Assumption~\ref{a.limit shape assumption} and since $|z|\leq r^{5/6}$.

Recall from \eqref{e.crude parabolic loss} and Lemma~\ref{l.Z tail bound 2} that there exists a $\lambda\in (0,1)$ such that, for $|z/2+(s+j)r^{2/3}|\leq r$, and $i=1$ and $2$,
\begin{equation}\label{e.single Z bound}
\P\left(Z_j^{(i)} > \nu_{i,j} + \theta r^{1/3}\right) \leq \exp(-c\theta^\alpha),
\end{equation}
where $\nu_{i,j} = \tfrac{1}{2}\mu r -\lambda\cdot \frac{G}{r/2}\cdot(\tfrac{1}{2}z\pm (s+j)r^{2/3})^2$ with the $+$ for $i=1$ and $-$ for $i=2$; $\nu_{i,j}$ captures the typical weight of these paths. Note that we are very crude with the parabolic coefficient, but the bound \eqref{e.single Z bound} holds for all $j$; and also that we measure the deviation from the same expression $\nu_{i,j}$ (which is obtained by evaluating \eqref{e.crude parabolic loss} at one endpoint) for all points in the interval. As we will see, comparing the full interval to a single point will not work for the second case of small $j$.

We want to show that the typical weight $\nu_{1,j} + \nu_{2,j}$ is much lower than $\mu r - Gz^2/r$. Simple algebraic manipulations show that, if $(s+j)r^{2/3} > (\lambda^{-1}-1)^{1/2} r^{5/6}$ (which is the largeness condition on $j$ defining the first case),
$$\sum_{i=1}^2\nu_{i,j} < \mu r- \lambda \frac{Gz^2}{r}- (1-\lambda)Gr^{2/3} - 3\lambda G(s+j)^2r^{1/3}  <\mu r- \frac{Gz^2}{r} - 3\lambda G(s+j)^2r^{1/3},$$
the final inequality since $|z|\leq r^{5/6}$. We have to bound \eqref{e.sum of Zs} with some value of $c_2$, and we take it to be $2\lambda G$; note that any bound we prove on \eqref{e.sum of Zs} will still be true if we later further lower $c_2$. The previous displayed bound shows that, for $(s+j)r^{2/3} > (\lambda^{-1}-1)^{1/2}r^{5/6}$,
\begin{align*}
\MoveEqLeft[10]
\P\left(Z_j^{(1)} + Z_j^{(2)} > \mu r - \frac{Gz^2}{r} - Hr^{1/3} - c_2(s+j)^2r^{1/3}\right)\\
&\leq \P\left(Z_j^{(1)} + Z_j^{(2)} > \nu_{1,j} + \nu_{2,j}  +\tfrac{1}{2}\lambda G(s+j)^2r^{1/3}\right).
\end{align*}
In the inequality we absorbed $-Hr^{1/3}$ into the last term by imposing that $s$ is large enough, depending on $\lambda$, $G$, and $H$. Now by a union bound and \eqref{e.single Z bound}, the last display, and hence \eqref{e.sum of Zs}, is bounded by $2\exp(-c(s+j)^{2\alpha})$.

Now we address the other case that $(s+j)r^{2/3} \leq (\lambda^{-1}-1)^{1/2} r^{5/6}$. Thus $I_j$ is close to the interpolating line, and we need a bound on the interval-to-interval weight with a much sharper parabolic term than in the previous case. Here above approach of the first case faces an issue. Since the gradient of $Gz^2/r$ at $z$ is $2Gz/r$, the weight difference across an interval of length $r^{2/3}$ at antidiagonal displacement $z$ is of order $z/r^{1/3}$, which is much larger than the bearable error of $O(r^{1/3})$ when $z$ is, say, $r^{5/6}$; so the crude approach of using the same expression (which we need to be less than $\mu r - Gz^2/r$) for the typical weight of all points in the interval, as we did in the first case, is insufficient---to have a single expression for which a tail bound exists for all points in the interval, we must necessarily include the linear gain of moving across the interval in the expression, and this will force it above $\mu r - Gz^2/r$. So, for this case, we will use Lemma~\ref{l.Z tail 3}, which avoids the problem by taking the supremum after centering by the point-specific expectation.

Let $X'_v = X_{v, (r-z,r+z)}$. Now we observe
\begin{align*}
\P\left(A_j\right) \leq \P\left(X_r^z < \E[X_r^z] - c_2 (s+j)^2r^{1/3}\right) + \P\left(\sup_{v\in I_j} (X_v + X'_{v}) > \E[X_r^z] - c_2(s+j)^2r^{1/3}\right);
\end{align*}
note that $X_v+X_v'$ counts the weight of $v$ twice, but this is acceptable as this sum dominates the weight of the best path through $v$. The first term is at most $\exp(-c(s+j)^{2\alpha})$ for a $c>0$ depending on $c_2$. We bound the second term as follows. First we note that $\E[X_r^z] \geq \mu r - Gz^2/r - Hr^{1/3}$ and that $\sup_{v\in I_j}\left(\E[X_v + X'_v]\right) \leq \mu r - Gz^2/r -G(s+j)^2r^{1/3}$ by a simple calculation with Assumption~\ref{a.limit shape assumption}, and so
\begin{align*}
\MoveEqLeft[10]
\P\left(\sup_{v\in I_j} (X_v + X'_{v}) > \E[X_r^z] - c_2(s+j)^2r^{1/3}\right)\\
&\leq \P\left(\sup_{v\in I_j} (X_v - \E[X_v] + X'_{v} - \E[X'_v]) > -Hr^{1/3} + (G - c_2)(s+j)^2r^{1/3}\right).
\end{align*}
We lower $c_2$ (if required) from its earlier value to be less than $G/2$. Now, we need to absorb the $-Hr^{1/3}$ term above into the $(s+j)^2 r^{1/3}$ term, which we can do for $s>s_0$ by setting $s_0$ large enough depending on $G$ and $H$. So for such $s$, by a union bound we see that the previous display is at most
\begin{align*}
\P\left(\sup_{v\in I_j} (X_v - \E[X_v]) > \tfrac{1}{6}G(s+j)^2r^{1/3}\right)
& + \P\left(\sup_{v\in I_j} (X'_v - \E[X'_v]) > \tfrac{1}{6}G(s+j)^2r^{1/3}\right).
\end{align*}
We bound this by applying Lemma~\ref{l.Z tail 3}, with $K = (\lambda^{-1}-1)^{1/2}$ and $\theta = \tfrac{1}{6}G(s+j)^2$. Recall that the bound of Lemma~\ref{l.Z tail 3} holds for $\theta>\theta_0(K)$. Thus we raise $s_0$ further if necessary so that $(s+j)^2 > \theta_0(K)$ for all $s>s_0$ and $j\geq 0$. Then we see that, for $s$ and $j$ such that $s>s_0$ and $(s+j)r^{2/3}\leq(\lambda^{-1}-1)r^{5/6}$, the last display is at most $2\exp(-c(s+j)^{2\alpha})$.

Returning to the inclusion \eqref{e.tf midpoint inclusion} and the bound of $\exp(-c(s+j)^{2\alpha})$ of $\P(A_j)$ for the two cases, we see that
$$\P\left(\Gamma_r^z(r/2) > z/2 + sr^{2/3}\right) \leq \sum_{j=1}^{r^{1/3}} \exp(-c(s+j)^{2\alpha})\leq C\exp(-cs^{2\alpha})$$
for some absolute constant $C<\infty$ and $c>0$ depending on $\alpha$. Here we used that, if $\alpha\in(0,1/2)$, then $(s+j)^{2\alpha} \geq 2^{2\alpha-1}(s^{2\alpha}+j^{2\alpha})$, while if $\alpha\geq 1/2$, then $(s+j)^{2\alpha} \geq s^{2\alpha}+j^{2\alpha}$; and finally $\exp(-cj^{2\alpha})$ is summable over $j$. This completes the proof of Proposition~\ref{p.midpoint tf}.
\end{proof}

To extend the transversal fluctuation bound from the midpoint (as in Proposition~\ref{p.midpoint tf}) to anywhere along the geodesic (as in Proposition~\ref{p.tf}), we follow very closely a multiscale argument previously employed in \cite[Theorem 11.1]{slow-bond} and \cite[Theorem 3.3]{watermelon} for similar purposes. For this reason, we will not write a detailed proof but only outline the idea. 

\begin{proof}[Proof sketch of Proposition~\ref{p.tf}]
First, the interpolating line is divided up into dyadic scales, indexed by $j$. The $j$\textsuperscript{th} scale consists of $2^j+1$ anti-diagonal intervals, placed at separation $2^{-j}r$, of length of order $s_jr^{2/3}:=\prod_{i=1}^j(1+2^{-i/3})sr^{2/3}$. By choosing the maximum $j$ for which this is done large enough, it can be shown that, on the event that $\tf(\Gamma_r^z) > sr^{2/3}$, there must be a $j$ such that there is a pair $(I_1, I_3)$ of consecutive intervals on the $j$\textsuperscript{th} scale, and the interval $I_2$ of the $(j+1)$\textsuperscript{th} scale in between such that the following holds: the geodesic passes through $I_1$ and $I_3$, but fluctuates enough that it avoids $I_2$, say by passing to its left. 

Planarity and that the  geodesic is a weight-maximising path then implies that the geodesic from the left endpoint of $I_1$ to that of $I_3$ is to the left of the geodesic $\Gamma_r^z$ (this observation is often called geodesic or polymer ordering), and so must have midpoint transversal fluctuation at least of order $(s_{j+1}-s_{j})r^{2/3} = 2^{-(j+1)/3}sr^{2/3}$. But since this transversal fluctuation happens across a scale of length $r' = 2^{-j}r$, in scaled coordinates it is of order $2^{j/3}s(r')^{2/3}$. Applying Proposition~\ref{p.midpoint tf} says that this probability is at most $\exp(-c2^{2\alpha j/3}s^{2\alpha})$. Now it remains to take a union bound over all the scales and the intervals within each scale. Since the number of intervals in the $j$\textsuperscript{th} scale is $2^j$, and since $2^j\exp(-c2^{2\alpha j/3}s^{2\alpha})\leq 2^{-j}\exp(-cs^{2\alpha})$ for all $s\geq s_0$ (by setting $s_0$ large enough) and $j\geq 1$, we obtain the overall probability bound of $\exp(-cs^{2\alpha})$ of Proposition~\ref{p.tf}.
\end{proof}

\subsection*{Conflict of interest} The authors have no competing interests to declare that are relevant to the content of this article.

\subsection*{Data availability} Data sharing not applicable to this article as no datasets were generated or analysed during the current study.

\bibliographystyle{alpha}
\bibliography{bootstrapping}

\newcommand{\etalchar}[1]{$^{#1}$}
\begin{thebibliography}{BDM{\etalchar{+}}01}

\bibitem[AD14]{auffinger2014simplified}
Antonio Auffinger and Michael Damron.
\newblock A simplified proof of the relation between scaling exponents in
  first-passage percolation.
\newblock {\em The Annals of Probability}, 42(3):1197--1211, 2014.

\bibitem[ADH17]{auffinger201750}
Antonio Auffinger, Michael Damron, and Jack Hanson.
\newblock {\em 50 years of first-passage percolation}, volume~68.
\newblock American Mathematical Soc., 2017.

\bibitem[Ale97]{alexander1997approximation}
Kenneth~S Alexander.
\newblock Approximation of subadditive functions and convergence rates in
  limiting-shape results.
\newblock {\em The Annals of Probability}, 25(1):30--55, 1997.

\bibitem[Ale20]{alexander2020geodesics}
Kenneth~S Alexander.
\newblock Geodesics, bigeodesics, and coalescence in first passage percolation
  in general dimension.
\newblock {\em arXiv preprint arXiv:2001.08736}, 2020.

\bibitem[BCD21]{barraquand2021fluctuations}
Guillaume Barraquand, Ivan Corwin, and Evgeni Dimitrov.
\newblock Fluctuations of the log-gamma polymer free energy with general
  parameters and slopes.
\newblock {\em Probability Theory and Related Fields}, 181(1):113--195, 2021.

\bibitem[BDM{\etalchar{+}}01]{baik2001optimal}
Jinho Baik, Percy Deift, Kenneth~DT McLaughlin, Peter Miller, and Xin Zhou.
\newblock Optimal tail estimates for directed last passage site percolation
  with geometric random variables.
\newblock {\em Advances in Theoretical and Mathematical Physics}, 5(6):1--41,
  2001.

\bibitem[BGHH20]{watermelon}
Riddhipratim Basu, Shirshendu Ganguly, Alan Hammond, and Milind Hegde.
\newblock Interlacing and scaling exponents for the geodesic watermelon in last
  passage percolation.
\newblock {\em arXiv preprint arXiv:2006.11448}, 2020.

\bibitem[BGHK19]{basu2019lower}
Riddhipratim Basu, Shirshendu Ganguly, Milind Hegde, and Manjunath Krishnapur.
\newblock Lower deviations in {$\beta$}-ensembles and law of iterated logarithm
  in last passage percolation.
\newblock {\em Israel Journal of Mathematics}, 2019.
\newblock To appear.

\bibitem[BGS17]{basu2017upper}
Riddhipratim Basu, Shirshendu Ganguly, and Allan Sly.
\newblock Upper tail large deviations in first passage percolation.
\newblock {\em arXiv preprint arXiv:1712.01255}, 2017.

\bibitem[BGS19]{basu2019delocalization}
Riddhipratim Basu, Shirshendu Ganguly, and Allan Sly.
\newblock Delocalization of polymers in lower tail large deviation.
\newblock {\em Communications in Mathematical Physics}, 370(3):781--806, 2019.

\bibitem[BH19]{brito2019geodesic}
Gerandy Brito and Christopher Hoffman.
\newblock Geodesic rays and exponents in ergodic planar first passage
  percolation.
\newblock {\em arXiv preprint arXiv:1912.06338}, 2019.

\bibitem[BHS18]{BHS18}
Riddhipratim Basu, Christopher Hoffman, and Allan Sly.
\newblock Nonexistence of bigeodesics in integrable models of last passage
  percolation.
\newblock {\em arXiv preprint arXiv:1811.04908}, 2018.

\bibitem[BKS03]{benjamini2003first}
Itai Benjamini, Gil Kalai, and Oded Schramm.
\newblock First passage percolation has sublinear distance variance.
\newblock {\em The Annals of Probability}, 31(4):1970--1978, 2003.

\bibitem[BLM13]{boucheron2013concentration}
St{\'e}phane Boucheron, G{\'a}bor Lugosi, and Pascal Massart.
\newblock {\em Concentration inequalities: A nonasymptotic theory of
  independence}.
\newblock Oxford University Press, 2013.

\bibitem[BR08]{benaim2008exponential}
Michel Benaim and Rapha{\"e}l Rossignol.
\newblock Exponential concentration for first passage percolation through
  modified {P}oincar{\'e} inequalities.
\newblock In {\em Annales de l'IHP Probabilit{\'e}s et statistiques},
  volume~44, pages 544--573, 2008.

\bibitem[BSS14]{slow-bond}
Riddhipratim Basu, Vladas Sidoravicius, and Allan Sly.
\newblock Last passage percolation with a defect line and the solution of the
  slow bond problem.
\newblock {\em arXiv preprint arXiv:1408.3464}, 2014.

\bibitem[BSS19]{coalescence}
Riddhipratim Basu, Sourav Sarkar, and Allan Sly.
\newblock Coalescence of geodesics in exactly solvable models of last passage
  percolation.
\newblock {\em Journal of Mathematical Physics}, 60(9):093301, 2019.

\bibitem[CD13]{chatterjee2013central}
Sourav Chatterjee and Partha~S Dey.
\newblock Central limit theorem for first-passage percolation time across thin
  cylinders.
\newblock {\em Probability Theory and Related Fields}, 156(3-4):613--663, 2013.

\bibitem[CG20]{corwin2018kpz}
Ivan Corwin and Promit Ghosal.
\newblock {KPZ} equation tails for general initial data.
\newblock {\em Electronic Journal of Probability}, 25, 2020.

\bibitem[Cha13]{chatterjee2013universal}
Sourav Chatterjee.
\newblock The universal relation between scaling exponents in first-passage
  percolation.
\newblock {\em Annals of Mathematics}, pages 663--697, 2013.

\bibitem[DH14]{damron2014busemann}
Michael Damron and Jack Hanson.
\newblock Busemann functions and infinite geodesics in two-dimensional
  first-passage percolation.
\newblock {\em Communications in Mathematical Physics}, 325(3):917--963, 2014.

\bibitem[DHS14]{damron2014subdiffusive}
Michael Damron, Jack Hanson, and Philippe Sosoe.
\newblock Subdiffusive concentration in first passage percolation.
\newblock {\em Electronic Journal of Probability}, 19, 2014.

\bibitem[DHS15]{damron2015sublinear}
Michael Damron, Jack Hanson, and Philippe Sosoe.
\newblock Sublinear variance in first-passage percolation for general
  distributions.
\newblock {\em Probability Theory and Related Fields}, 163(1-2):223--258, 2015.

\bibitem[EJS20]{emrah2020right}
Elnur Emrah, Chris Janjigian, and Timo Sepp{\"a}l{\"a}inen.
\newblock Right-tail moderate deviations in the exponential last-passage
  percolation.
\newblock {\em arXiv preprint arXiv:2004.04285}, 2020.

\bibitem[FO18]{FO17}
Patrik~L Ferrari and Alessandra Occelli.
\newblock Universality of the {GOE} {T}racy-{W}idom distribution for {TASEP}
  with arbitrary particle density.
\newblock {\em Electronic Journal of Probability}, 23, 2018.

\bibitem[Gan20]{gangopadhyay2020fluctuations}
Ujan Gangopadhyay.
\newblock Fluctuations of transverse increments in two-dimensional first
  passage percolation.
\newblock {\em arXiv preprint arXiv:2011.14686}, 2020.

\bibitem[GDBR14]{gueudre2014revisiting}
Thomas Gueudre, Pierre~Le Doussal, Jean-Philippe Bouchaud, and Alberto Rosso.
\newblock Revisiting directed polymers with heavy-tailed disorder.
\newblock {\em arXiv preprint arXiv:1411.1242}, 2014.

\bibitem[Ham19a]{hammond2017modulus}
Alan Hammond.
\newblock Modulus of continuity of polymer weight profiles in {B}rownian last
  passage percolation.
\newblock {\em The Annals of Probability}, 47(6):3911--3962, 2019.

\bibitem[Ham19b]{hammond2017patchwork}
Alan Hammond.
\newblock A patchwork quilt sewn from {B}rownian fabric: regularity of polymer
  weight profiles in {B}rownian last passage percolation.
\newblock In {\em Forum of Mathematics, Pi}, volume~7. Cambridge University
  Press, 2019.

\bibitem[Ham20]{brownianLPPtransversal}
Alan Hammond.
\newblock Exponents governing the rarity of disjoint polymers in {B}rownian
  last passage percolation.
\newblock {\em Proceedings of the London Mathematical Society},
  120(3):370--433, 2020.

\bibitem[Joh00]{johansson2000shape}
Kurt Johansson.
\newblock Shape fluctuations and random matrices.
\newblock {\em Communications in mathematical physics}, 209(2):437--476, 2000.

\bibitem[KC18]{stretched-exp-concentration}
Arun~Kumar Kuchibhotla and Abhishek Chakrabortty.
\newblock Moving beyond sub-gaussianity in high-dimensional statistics:
  Applications in covariance estimation and linear regression.
\newblock {\em arXiv preprint arXiv:1804.02605}, 2018.

\bibitem[Kes86]{kesten1986aspects}
Harry Kesten.
\newblock Aspects of first passage percolation.
\newblock In {\em {\'E}cole d'{\'e}t{\'e} de probabilit{\'e}s de Saint Flour
  XIV-1984}, pages 125--264. Springer, 1986.

\bibitem[Led18]{ledoux2018law}
Michel Ledoux.
\newblock A law of the iterated logarithm for directed last passage
  percolation.
\newblock {\em Journal of Theoretical Probability}, 31(4):2366--2375, 2018.

\bibitem[LM01]{lowe-moderate-upper}
Matthias L{\"o}we and Franz Merkl.
\newblock Moderate deviations for longest increasing subsequences: the upper
  tail.
\newblock {\em Communications on Pure and Applied Mathematics: A Journal Issued
  by the Courant Institute of Mathematical Sciences}, 54(12):1488--1519, 2001.

\bibitem[LMR02]{lowe-moderate-lower}
Matthias L{\"o}we, Franz Merkl, and Silke Rolles.
\newblock Moderate deviations for longest increasing subsequences: the lower
  tail.
\newblock {\em Journal of Theoretical Probability}, 15(4):1031--1047, 2002.

\bibitem[LR10]{ledoux2010}
Michel Ledoux and Brian Rider.
\newblock Small deviations for beta ensembles.
\newblock {\em Electron. J. Probab.}, 15:1319--1343, 2010.

\bibitem[LS22]{landon2022tail}
Benjamin Landon and Philippe Sosoe.
\newblock Tail bounds for the {O'C}onnell-{Y}or polymer.
\newblock {\em arXiv preprint arXiv:2209.12704}, 2022.

\bibitem[Mar06]{martin2006last}
James~B Martin.
\newblock Last-passage percolation with general weight distribution.
\newblock {\em Markov Process. Related Fields}, 12(2):273--299, 2006.

\bibitem[New95]{newman1995surface}
Charles~M Newman.
\newblock A surface view of first-passage percolation.
\newblock In {\em Proceedings of the international congress of mathematicians},
  pages 1017--1023. Springer, 1995.

\bibitem[NP95]{newman1995divergence}
Charles~M Newman and Marcelo~ST Piza.
\newblock Divergence of shape fluctuations in two dimensions.
\newblock {\em The Annals of Probability}, pages 977--1005, 1995.

\bibitem[OY02]{o2002representation}
Neil O'Connell and Marc Yor.
\newblock A representation for non-colliding random walks.
\newblock {\em Electronic communications in probability}, 7:1--12, 2002.

\bibitem[RRV11]{ramirez2011beta}
Jose Ramirez, Brian Rider, and B{\'a}lint Vir{\'a}g.
\newblock Beta ensembles, stochastic {A}iry spectrum, and a diffusion.
\newblock {\em Journal of the American Mathematical Society}, 24(4):919--944,
  2011.

\bibitem[Sep98a]{seppalainen1998coupling}
Timo Sepp{\"a}l{\"a}inen.
\newblock Coupling the totally asymmetric simple exclusion process with a
  moving interface.
\newblock {\em Markov Process. Related Fields}, 4(4):593--628, 1998.

\bibitem[Sep98b]{seppalainen1998large}
Timo Sepp{\"a}l{\"a}inen.
\newblock Large deviations for increasing sequences on the plane.
\newblock {\em Probability theory and related fields}, 112(2):221--244, 1998.

\bibitem[Ver18]{vershynin2018high}
Roman Vershynin.
\newblock {\em High-dimensional probability: An introduction with applications
  in data science}, volume~47.
\newblock Cambridge University Press, 2018.

\bibitem[Wai19]{wainwright-concentration}
Martin~J. Wainwright.
\newblock {\em High-{D}imensional {S}tatistics: {A} {N}on-{A}symptotic
  {V}iewpoint}.
\newblock Cambridge Series in Statistical and Probabilistic Mathematics.
  Cambridge University Press, 2019.

\end{thebibliography}
\end{document}